\documentclass[reqno]{amsart}

\usepackage{amsfonts}
\usepackage{amssymb}
\usepackage{amsmath}

\usepackage{enumitem}
\usepackage{xcolor}
\usepackage{todonotes}

\setenumerate{label={\rm (\alph{*})}}
\usepackage{bbm,bm,euscript,mathrsfs}
\usepackage{paper_diening}
\usepackage{graphicx}
\usepackage[top=1in, bottom=1.25in, left=1.10in, right=1.10in]{geometry}
\usepackage{relsize}
\usepackage{hyperref} 
\usepackage{cases, color,amsthm} 
\hypersetup{linkcolor=blue, colorlinks=true,citecolor = red}

\allowdisplaybreaks

\numberwithin{equation}{section}

\usepackage{tikz-cd}
\usetikzlibrary{lindenmayersystems}
\usetikzlibrary{decorations.pathreplacing}

\DeclareMathOperator{\Div}{div}

\newcommand{\dd}{\mathrm d}
\newcommand{\dt}{\,\mathrm{d} t}

\newcommand{\bx}{\mathbf{x}}
\newcommand{\by}{\mathbf{y}}
\newcommand{\bz}{\mathbf{z}}
\newcommand{\bq}{\mathbf{q}}
\newcommand{\bu}{\mathbf{u}}

\newcommand{\bv}{\mathbf{v}}
\newcommand{\bn}{\mathbf{n}}

\newcommand{\dy}{\, \mathrm{d}\mathbf{y}}

\newcommand{\dx}{\, \mathrm{d} \mathbf{x}}

\newcommand{\divx}{\mathrm{div} }

\newcommand{\nabx}{\nabla }
\newcommand{\naby}{\nabla_{\mathbf{y}}}

\newcommand{\Delx}{\Delta }
\newcommand{\Dely}{\Delta_{\mathbf{y}}}

\newcommand{\dH}{\,\mathrm{d}\mathbf{\mathcal{H}}^2}

\newcommand{\ds}{\,\mathrm{d}s}

\newcommand{\Oeta}{\Omega_\eta}

\newtheorem{theorem}{Theorem}[section]
\newtheorem{lemma}[theorem]{Lemma}
\newtheorem{proposition}[theorem]{Proposition}

\newtheorem{remark}[theorem]{Remark}

\theoremstyle{definition}
\newtheorem{definition}[theorem]{Definition}

\begin{document}

\title[Local  Well-Posedness  for compressible Fluid-viscoelastic Shell Interactions]{ Local  Well-Posedness for Barotropic Compressible Fluid-viscoelastic Shell Interactions}

\author{Pierre Marie Ngougoue Ngougoue}
\address{Faculty of Mathematics, University of Duisburg-Essen, Thea-Leymann-Straße 9, 45127 Essen, Germany}
\email{pierre.ngougouengougoue@uni-due.de}


\date{\today}


\begin{abstract}
We study a three-dimensional barotropic compressible Navier--Stokes flow interacting with a viscoelastic  shell that occupies a portion of the fluid boundary. The analysis is entirely  Eulerian and the moving interface is parametrised by a localised Hanzawa transform supported near the shell patch, which preserves the transport structure of the continuity equation and avoids a global Lagrangian map. We prove local-in-time existence and uniqueness of strong solutions for compatible data and without imposing a vanishing initial shell displacement. The proof combines a well-posedness theory for the continuity equation, solved by the method of characteristics in the Hanzawa frame, with an analysis of the momentum--structure subproblem carried out by the classical linearisation--energy estimate--fixed-point scheme on a fixed reference domain. A Banach fixed point then couples the two steps and closes the argument on  a short time interval. We work with a viscosity-weighted energy that makes the scaling in the shear and bulk viscosities explicit. This yields bounds whose constants grow at most linearly in these parameters. In particular, the estimates are inviscid-limit compatible. The result complements the global finite-energy weak theory for compressible fluid--shell interaction by providing a local strong well-posedness statement in the same geometric configuration. It also extends strong boundary results beyond beams and plates to bending shells, and addresses the wave-to-bending direction suggested in the literature on compressible fluid--structure interaction at structural boundaries, while remaining fully Eulerian.
  
\end{abstract}

\maketitle

\section{Introduction}

Fluid--structure interaction (FSI) underpins aeroelastic panels, deformable microfluidic devices, arterial flows, and acoustic membranes, where a viscous fluid exchanges momentum with a deformable thin structure through a moving interface. In the incompressible regime the landscape is comparatively mature and naturally splits according to geometry. For an elastic solid completely immersed in a viscous fluid, \cite{coutand2005motion} established local existence and uniqueness of strong solutions in Lagrangian formulation.  Their analysis explicitly addresses the parabolic (fluid) -- hyperbolic (solid) coupling by first solving the linear system (unique weak solution) and then upgrading regularity and closing a fixed point for the nonlinear strong problem.  For structures coupled to the fluid along a portion of the boundary, two benchmark  results address both long-time control under damping and broader geometric settings. \cite{grandmont2016existence} proved global-in-time strong solutions for a two-dimensional viscous fluid interacting with a linear viscoelastic one-dimensional  beam, and in particular obtained a no-contact result:  the elastic wall does hit the opposite boundary in finite time . Complementing this,  \cite{grandmont2019existence} established a unified local well-posedness theory for incompressible viscous fluid interacting with either a  beam or a rod, proving local-in-time strong existence and uniqueness for sufficiently regular initial data via sharp  maximal $L^p$--$L^q$ regularity and a contraction scheme. 
At the level of weak solutions, \cite{Chambolle2005existence} proved existence of at least one weak solution for a viscous incompressible fluid coupled to an elastic plate as long as the plate does not touch the fixed part of the boundary.  For thin shells sitting on the boundary, \cite{lengeler2014weak, muha2022existence} established global-in-time weak solutions for an incompressible Newtonian fluid interacting with a Koiter shell. The existence is up to geometric breakdown such as self-intersection, thereby extending the plate framework to bending shells on a moving boundary.  Around these weak theories, subsequent works developed the framework further and improved regularity results. For instance \cite{breit2021incompressible}  obtained  weak solutions for an \mbox{incompressible}    polymeric (Navier--Stokes--Fokker--Planck) fluid interacting with a Koiter shell, with existence until either the shell approaches self-intersection or the energy degenerates. More recently \cite{breit2023ladyzhenskaya} proved a Ladyzhenskaya--Prodi--Serrin conditional criterion that ensures regularity  and uniqueness for incompressible fluid--shell systems by combining local strong solvability, an acceleration estimate and weak--strong uniqueness, which together pinpoint regimes where solutions are unique and strong.

The step from incompressible to compressible FSI is not merely incremental; two structural features change and drive new difficulties. First, the density $\rho$ becomes an unknown governed by a continuity equation, so one must ensure positivity and Lipschitz-type dependence of $\rho$ on the fluid velocity and on the interface motion (after geometric transformation), while losing the divergence-free cancellations that simplify nonlinear estimates in the incompressible case. Second, the pressure law couples nonlinearly with the velocity;  in weak-solution settings, this is often controlled by the renormalised continuity equation  and the effective viscous-flux structure (see e.g., \cite{DiPerna1989}, \cite{lions1996mathematical},\cite{feireisl2004dynamics} ).  With these obstacles in view, the compressible literature has coalesced -- again split by geometry -- into a thinner but coherent body of results. For immersed solids, \cite{kukavica2012well} proved strong well-posedness for a compressible viscous fluid coupled across a free interface with a linear elastic (Lamé) solid, establishing short-time control with continuity of velocity and stress across the interface. For boundary-mounted thin structures, \cite{breit2018compressible} obtained global finite-energy  weak solutions for an isentropic fluid interacting with a linear Koiter shell that forms part of the boundary in three dimensions for adiabatic exponent $\gamma > \frac{12}{ 7}$ -- with existence up to geometric degeneration of the shell --  which now serves as a baseline weak theory in the compressible boundary class.   At higher regularity there are local strong results for plates and beams, notably  \cite{mitra2020local} compressible channel flow coupled to a damped Euler--Bernoulli beam. Here, solutions are constructed by first decoupling the parabolic (fluid), continuity, and beam (structure) equations and then closing a fixed-point argument, typically under zero initial displacement of the beam. Likewise, \cite{maity2021existence} studied a compressible Navier--Stokes--Fourier fluid with a structurally damped plate and, via a fixed-domain transformation, established  $L^{p}-L^{q}$ maximal regularity,  leading to local strong solutions and -- under  small data -- global ones.  A closely related boundary model replaces bending by a longitudinal wave. In that setting \cite{maityroy2021existence} proved local strong solutions on a fixed reference domain via a modified Lagrangian--geometric mapping, explicitly noting two advantages of the pullback: (i) one can work in the reference configuration and, (ii) the convection term in the density equation disappears. They further explicitly remarked that replacing the wave-structure by the plate would be natural yet nontrivial within their approach and may require stronger hypotheses on the initial data. The literature on compressible FSI at structural boundaries  has broadened further through multicomponent models such as the bi-fluid Koiter-shell system of  \cite{kalousek2024existence}  (existence of weak solutions assuming  different adiabatic exponents and dissipativity of the structure -- up to degeneracy of the energy or self-intersection of the structure) and stability mechanisms like \cite{trifunovic2023compressible} weak--strong uniqueness for compressible plate interaction via a relative-entropy framework.  

Against this backdrop, the present paper develops a compressible boundary-mounted bending model in a way that is both structurally and methodologically distinct.  We consider a three-dimensional isentropic viscous compressible fluid coupled with a viscoelastic shell supported on a portion of the boundary, and we establish local-in-time existence and uniqueness of strong solutions in a purely Eulerian framework. The moving interface is straightened only in a neighbourhood of the boundary by a localised Hanzawa transform, which leaves the interior mostly unaltered and keeps the geometry tractable.  Our main result is precisely stated in Theorem~\ref{theo:mainresult} and the proof proceeds through a two-steps construction glued by a fixed-point argument. First, assuming we are given a prescribed fluid velocity and a prescribed shell displacement, we solve the continuity equation in the straightened Hanzawa frame by the method of characteristics. This yields existence, uniqueness and sharp stability estimates for the density in the space used later  for the contraction argument (cf.~ Section \ref{sec:continuitysubprob}). Second,  in Section \ref{sec:MomStrucSubProb},  assuming a known density, we prove the local-in-time well-posedness of the momentum--structure subproblem via the classical linearisation, maximal regularity and fixed-point scheme.  Finally, a Banach contraction then couples the transport  and the momentum--structure steps and yields the full solution on a small time interval (cf. ~ Section \ref{sec:localstrongfixedpoint}). 

Relative to existing compressible boundary results,  the methodological choice to remain Eulerian yields concrete benefits and clarifies the scope of the model. In contrast with the Lagrangian framework used in  \cite{maity2021existence}  for a structurally damped plate and in  \cite{maityroy2021existence} for a boundary wave, we do not construct a global $C^{1}$-in-time diffeomorphism of the moving boundary nor rely on the cancellation of the convective term after pullback. Remaining Eulerian preserves the natural transport structure; the geometric effects are absorbed into the effective transport field, so the continuity analysis proceeds by characteristics without explicit commutator estimates. The method also addresses the `wave $\to$ plate' direction highlighted in the latter work, while staying in the compressible isentropic setting. Compared with the  beam model \cite{mitra2020local} in a two-dimensional channel, which is built around longitudinal bending and typically initialises the structure at rest, we treat a viscoealastic shell in three dimensions without assuming zero initial displacement and without restricting ourselves to a small neighbourhood of the trivial configuration. Our work also complements the finite-energy weak theory for compressible shells of \cite{breit2018compressible} by providing a local strong result in a closely related geometry. Moreover, we retain throughout, the explicit dependence on the  shear and bulk viscosities $\mu$ and $\lambda$ in all a priori estimates,  which makes the dissipation dependence transparent and prepares the framework for an inviscid-limit discussion.  The modular two-step construction -- method of characteristics for the continuity subproblem and classical linearisation - maximal regularity - fixed-point argument  for the momentum--structure system -- also makes the analysis adaptable to nearby thin-structure laws and mixed boundary configurations.


\subsection{The fluid-structure interaction problem}
We consider the interaction of the compressible fluid and the flexible shell where the shell is located on a part of the boundary of the reference domain $\Omega\subset \mathbb{R}^{3}$. As the shell deforms, the reference fluid domain $\Omega$ changes to $\Omega_\eta$.
Accordingly, the shell function $\eta:(t, \by)\in I \times \omega \mapsto   \eta(t,\by)\in \mathbb{R}$ with $I=(0,T)$ for some $T>0$ solves\footnote{We use nondimensional variables and set the shell’s surface mass density, structural damping coefficient, and bending stiffness to 1, so no explicit constants appear in  \eqref{eq:ShellEq}.}
\begin{equation}\label{eq:ShellEq}
\left\{\begin{aligned}
& \partial_t^2\eta - \partial_t\Dely \eta + \Dely^2\eta=-\bn^\intercal(\bm{\tau}\bn_\eta)\circ\bm{\varphi}_\eta
\mathrm{det}(\naby \bm{\varphi}_{\eta}) 
&\text{ for all }  (t,\by)\in I\times\omega
,\\
&\eta(0,\by)=\eta_0(\by), \quad (\partial_t\eta)(0, \by)=\eta_*(\by)
&\text{ for all } \by\in\omega,
\end{aligned}\right.
\end{equation}
with periodic boundary conditions in space. Here, $\omega \subset \mathbb{R}^2$ is such that there is $\bfvarphi_\eta :\omega\to \partial \Omega_\eta$ that parametrizes the boundary of the deformed domain $\Omega_\eta$. We denote the fluid velocity and density  by 
\[\mathbf{v}:(t, \mathbf{x})\in I \times \Oeta \mapsto  \mathbf{v}(t, \mathbf{x}) \in \mathbb{R}^3  \quad \text{and} \quad  \rho:(t, \mathbf{x})\in I \times \Oeta \mapsto  \rho(t, \mathbf{x}) \in \mathbb{R}.  \]
The vectors $\bn$ and $\bn_\eta$ are the normal vectors  of the reference boundary and of the deformed boundary, respectively, whereas $\bm{\tau}$ denotes the Cauchy stress of the fluid given by {\em Newton's rheological law}, that is
$$\bm{\tau} =\mathbb{S}(\nabla\bv)-p(\rho)\mathbb I_{3\times 3}$$
with the viscous stress tensor
$$\mathbb{S}(\nabla\bv)=2\mu\left(\frac{1}{2}\big(\nabla\bv+(\nabla\bv)^\intercal\big)-\frac{1}{3}\Div \bv\mathbb{I}_{3\times 3}\right)+\left(\lambda+\frac{2}{3}\mu\right)\Div\bv\mathbb{I}_{3\times 3}. $$
The viscosity coefficients $\mu$ and $\lambda$ satisfy
$$\mu>0, \quad \lambda+\frac{2}{3}\mu\geq0. $$ 

The motion of the fluid in $\Omega_\eta$ is governed by the compressible Navier--Stokes equations:
\begin{equation}\label{eq:ContMomentEq}
\left\{\begin{aligned}
&\partial_t \rho  + \divx(\rho\mathbf{v} \big)
= 
0 &\text{ for all }(t,\bx)\in I\times\Omega_\eta,\\
&\partial_t(\rho \bv)+\Div(\rho\bv\otimes\bv)
= 
\mu\Delx \bv +(\lambda+\mu)\nabx\Div \bv-\nabx p(\rho) &\text{ for all }(t,\bx)\in I\times\Omega_\eta,\\
&\rho(0,\bx)=\rho_0(\bx), \qquad
(\rho\bv)(0,\bx)=\bq_0(\bx) &\text{ for all } \bx\in \Omega_{\eta_0},
\end{aligned}\right.
\end{equation}
where the pressure $p(\rho)$ is described by the standard isentropic state equation
$$p(\rho)=a\rho^\gamma \quad a>0,~ \gamma>1, $$
and $\gamma$ is called the adiabatic exponent.
The equations \eqref{eq:ShellEq} and \eqref{eq:ContMomentEq} are coupled through the kinematic boundary condition
\begin{align}
\label{interfaceCond}
\bv\circ \bfvarphi_\eta= \big(\partial_t\eta\big)\bn \quad\text{ for all } (t,\by)\in I\times \omega.
\end{align} 
Our main result is the following:

\begin{theorem}\label{theo:mainresult}
Assume the initial data $(\rho_0,\bv_0,\eta_0,\eta_*)$ satisfy 
\begin{align}
&\rho_0 \in W^{3,2}(\Omega_{\eta_0}),\qquad \bv_0 \in W^{3,2}(\Omega_{\eta_0}), \qquad\eta_0\in W^{5,2}(\omega), \label{eq:InitialCondSpace}
\\&
\eta_*\in W^{3,2}(\omega),\qquad\Vert\eta_0\Vert_{L^\infty(\omega)}<L, \qquad\bv_0\circ \bm{\varphi}_{\eta_0}=\eta_*\bn \text{ on } \omega, \label{eq:InitialCondInterface}
\end{align}
Then there  exists   $T_* \in I $ such that  \eqref{eq:ShellEq}--\eqref{interfaceCond}  admits a unique strong solution  $(\rho,\bv,\eta)$ on $I_* := (0, T_*]$ satisfying 
\begin{align*}
&\rho\in W^{1,2}\big(I_*;W^{3,2}(\Omega_\eta)\big)\cap W^{1,\infty}\big(I_*;W^{2,2}(\Omega_\eta)\big),
\\
&\bv\in L^{2}\big(I_*;W^{4,2}(\Omega_\eta)\big)
\cap W^{2,2}\big(I_*;L^{2}(\Omega_\eta)\big),
\\
&\eta \in L^2\big(I_*;W^{6,2}(\omega)\big)  
\cap W^{3,2}\big(I_*;L^{2}(\omega)\big)
  \cap W^{2,\infty}\big(I_*;W^{1,2}(\omega)\big),
\\
&\partial_t\eta \in L^\infty\big(I_*;W^{3,2}(\omega)\big) \cap L^{2}\big(I_*;W^{4,2}(\omega)\big).
\end{align*}
\end{theorem}


\subsection{Geometric Framework} 
We adopt the framework of   \cite[Section 2.4]{breit2024regularity}, and  assume that the \phantom{refer} reference domain $\Omega \subset \mathbb{R}^3 $ is open, bounded, and has a smooth, orientable boundary $\partial\Omega$, so that a consistent outward unit normal vector $\bn$ is well-defined and smooth.  To formalise the structure of the boundary $\partial\Omega$, we assume that it can be  parametrised by an injective map  $\bm{\varphi} \in C^{k}(\omega; \mathbb{R}^3) $  for some sufficiently large $k \in \mathbb{N}$.  For each $ \by = (y_1, y_2) \in \omega $, the pair of tangent vectors \( \partial_1 \bm{\varphi}(\by) \) and \( \partial_2\bm{\varphi}(\by) \) is assumed to be linearly independent. This ensures that $\bm{\varphi}$ defines a regular surface embedded in $\mathbb{R}^3 $, and the corresponding unit normal vector can be explicitly determined from $\bm{\varphi} $ by 
\[\bn(\by) = \dfrac{\partial_1\bm{\varphi}(\by) \times \partial_2 \bm{\varphi}(\by)}{| \partial_1 \bm{\varphi}(\by) \times \partial_2 \bm{\varphi}(\by) |}.  \]
Importantly, as a compact connected $C^k-$hypersurface in $\mathbb{R}^3$, $\partial\Omega$ satisfies the uniform interior and exterior ball condition (cf. \cite[Section 4.1]{PruessSimonett2013}), that is,  $\exists\, L > 0 $ such that for each point $\by \in \partial\Omega$, there are balls  $B(\bx, L) \subset \Omega $ and $B(\mathbf{z}, L) \subset \overline{\Omega}^c$ such that 
\[\partial\Omega \cap \overline{B}(\bx, L) \cap \overline{B}(\mathbf{z}, L) = \{\by\}. \]
Hence, as in \cite[Section 3.1]{PruessSimonett2016} we conclude that $\partial\Omega $ admits a tubular neighbourhood of radius $L >0 $, defined as 
\begin{equation}\label{eq:TubNeighbor}
\mathcal{N}_L := \left\{ \bx \in \mathbb{R}^3  \colon \mathrm{d}(\bx, \partial\Omega) < L \right\}.
\end{equation}
Within $\mathcal{N}_L $, every point $\bx $ can be uniquely projected onto the boundary $ \partial \Omega $  via the closest point projection. More precisely, for each $ \bx \in \mathcal{N}_L $, we define the map \[\by(\bx) = \arg\min\limits_{\by \in \omega} |\bx - \bm{\varphi}(\by)|,  \] so that the closest point on the boundary is given by $ \mathbf{p}(\bx) := \bm{\varphi}(\by(\bx)) $, and the signed distance to the boundary is defined as $s(\bx) := (\bx - \by(\bx))\cdot \bn(\by(\bx))$.

\noindent For $L > 0$ small enough, the projection $\mathbf{p}(\bx) $ is well-defined, and we have 
\[
|s(\bx)| = |\bx - \by(\bx)| = \min\limits_{\by \in \omega} |\bx - \bm{\varphi}(\by)| =  \mathrm{d}(\bx, \partial\Omega) \qquad \text{for all } \; \bx \in \mathcal{N}_L.
\]
Indeed, for large $L >0$, the normal vectors from different points on $\partial\Omega $ might intersect, leading to ambiguity in projection.  Consequently, $\mathcal{N}_L $ provides a natural coordinate system for points $\bx \in \mathcal{N}_L$. Specifically, each point $\bx \in \mathcal{N}_L $ can be uniquely written as 
\[\bx = \mathbf{p}(\bx) + s(\bx)\bn(\by(\bx)),  \] 
which implies -- together with \eqref{eq:TubNeighbor} -- the representation 
\begin{equation}\label{eq:TubNeighborFinal}
\mathcal{N}_L := \big\{ \by + s\bn(\by) \colon (s, \by) \in (-L, L)\times \omega \big\}.
\end{equation}

Furthermore, recall that the shell displacement $\eta$ defines a deformation of the boundary in the normal direction, leading to the time-dependent boundary 
\[\partial\Omega_{\eta (t)} = \big\{ \bm{\varphi}(\by) + \eta(t, \by)\bn(\by) \colon \by \in \omega\big\},  \]
with associated deformation map $ \bm{\varphi}_\eta (t, \by) = \bm{\varphi}(\by) + \eta(t, \by)\bn(\by). $ \\
However, to ensure that $\eta$ defines a valid normal parametrisation and that the deformed boundary $\partial\Omega_{\eta (t)}$ remains well-defined and regular over time, we required the following conditions:\\
\begin{itemize}

    \item[$\bullet$] The deformed boundary must remain within the tubular neighbourhood, that is, $$ \partial\Omega_{\eta (t)}  \subset \mathcal{N}_L \qquad \text{for all} \; t \in I. $$
    For this purpose, the displacement $\eta $ must satisfy $$\sup\limits_I \Vert \eta \Vert_{L^{\infty}(\omega)} < L .$$
    However, this condition is not enough as the deformed boundary $ \partial\Omega_{\eta } $ -- even for small Hausdorff distance -- can still fold or self intersect if the normals at different points intersect. To prevent this, we need to control the derivatives of $\eta$ to ensure the mapping $\by \longmapsto \bm{\varphi}_{\eta}(t, \by) $ remains \phantom{injec} injective. More precisely: \\
    
    \item[$\bullet$] The tangential vectors of the deformed boundary must remain non-degenerate, that is, 
    \begin{equation}\label{eq:Condl}
    \partial_1\bm{\varphi}_{\eta} \times \partial_2 \bm{\varphi}_{\eta} (t, \by) \neq 0, \quad  \text{for all} \; (t, \by) \in I \times \omega . 
    \end{equation}
    Moreover, to avoid boundary flipping, that is, to ensure the orientation of the deformed boundary $\partial\Omega_\eta$ matches the reference boundary $\partial\Omega$, we require that:
    \begin{equation}\label{eq:Condr} 
    \bn(\by)\cdot \bn_{\eta(t)}(\by) > 0 \quad \text{for all }\;  (t, \by) \in I \times \omega . 
    \end{equation}
    
\end{itemize}    
Hence, up to decreasing $L > 0$, one easily deduces that \eqref{eq:Condl} and \eqref{eq:Condr} hold provided that 
\[
\sup_{t \in I} \| \eta(t, \cdot) \|_{W^{1,\infty}(\omega)} < L.     
\] 

Within the preceding geometric setup, the boundary displacement $\eta$  admits a smooth normal extension into the interior of $\Omega$  via the Hanzawa transform 
\begin{equation}\label{eq:HanzawaT}
\bfPsi_\eta(t, \bx) :=
\begin{cases}
\mathbf{p}(\bx) + \Big( s(\bx) + \eta\big(\by(\bx)\big)\phi\big(s(\bx) \big)  \Big)\bn\big(\by(\bx)\big) & \text{if } \bx \in \mathcal{N}_L, \\
\bx & \text{elsewhere},
\end{cases}
\end{equation}
where \( \phi \in C^\infty(\mathbb{R}) \) is a cut-off function satisfying \( \phi \equiv 1 \) near  the boundary, that is,   \( s \approx 0 \), and far from the boundary, that is, $s \leq -L$,  $\phi \equiv 0$, so the transformation reduces to the identity map. In other words, the deformation $\eta$ is smoothly `turned off' as we move deeper into the interior of $\Omega$.

\noindent The following proposition gives some estimates for $\bfPsi_\eta$ proved in  \cite[Section 2.4]{breit2024regularity}. 

\begin{proposition}\label{prop:estimatePsiEta}
Let \( k \in \mathbb{N} \), \( p \in [1, \infty] \), and assume that \( \eta,\, \zeta \in W^{k,p}(\omega) \). Then  \( \bfPsi_\eta \) satisfies 
\begin{align}
\Vert  \bfPsi_\eta \Vert_{W^{k,p}(\Omega)} &\lesssim 1 + \Vert \eta \Vert_{W^{k,p}(\omega)},  \label{eq:HanzEstim1}
\\
\Vert \bfPsi_\eta - \bfPsi_\zeta \Vert_{W^{k,p}(\Omega)} &\lesssim \Vert \eta - \zeta \Vert_{W^{k,p}(\omega)}, \label{eq:HanzEstim2}
\end{align}
where the hidden constants depend on the curvature of \( \partial \Omega \), and the radius \( L > 0 \).\\

\noindent Furthermore, provided that   \;$ \sup\limits_{t \in I} \Vert \eta(t, \cdot) \Vert_{W^{1,\infty}(\omega)} < L $,  the inverse map \( \bfPsi_\eta^{-1} \) exists and satisfies
\begin{align}
\Vert  \bfPsi_{\eta}^{-1} \Vert_{W^{k,p}(\Omega)} &\lesssim 1 + \Vert \eta \Vert_{W^{k,p}(\omega)}, \label{eq:HanzEstim3}
\\
\Vert \bfPsi_{\eta}^{-1}\circ\bfPsi_{\eta} - \bfPsi_{\zeta}^{-1}\circ\bfPsi_{\zeta} \Vert_{W^{k,p}(\Omega)} &\lesssim \Vert \eta - \zeta \Vert_{W^{k,p}(\omega)}. \label{eq:HanzEstim4}
\end{align}
\end{proposition}


\section{The continuity subproblem} \label{sec:continuitysubprob}


For a known flexible domain $\Omega_\zeta $ and a known velocity field $\bu, $ we aim in this section to construct a strong  solution to the following subproblem: 
\begin{equation}\label{rhoEquAlone}
\partial_t \rho +  (\bu\cdot\nabla )\rho = -\rho\, \divx\bu,  
\end{equation} 
in $I\times\Omega_\zeta\subset \mathbb R^{1+3}$ subject to the following initial condition
\begin{align} 
&\rho(0,\cdot)=\rho_0(\cdot) &\text{in }\Omega_{\zeta(0)}.
\label{initialCondSolvSubPro}  
\end{align} 
Although a characteristic-based existence theory on the evolving domain $\Omega_{\zeta}$ is, in principle, feasible, it necessitates nontrivial control of trajectories to ensure they remain within $\Omega_{\zeta(t)}$ for all suitable time $t \in I$. To avoid this geometric subtlety, we reformulate the problem on the reference configuration $\Omega$ via the Hanzawa transform. For this purpose, we introduce the pull-back variables $\overline\rho := \rho\circ\bfPsi_{\zeta}, \; \overline\bu := \bu\circ\bfPsi_{\zeta}$ and define $\mathbf{B}_{\zeta} = J_\zeta \left(\nabx \bfPsi_\zeta^{-1}\circ \bfPsi_\zeta\right)^\intercal$ where  $ J_\zeta=\mathrm{det}(\nabla\bfPsi_\zeta)$. 

\noindent Thus, the pull-back density $\overline\rho$ solves 
\begin{equation}\label{rhoEquAloneTransform}
\partial_t \overline\rho +  (\overline{\bu}_{eff} \cdot\nabla )\overline\rho  + \delta_{\Div}\,\overline\rho = 0,  
\end{equation} 
on $I\times\Omega \subset \mathbb R^{1+3}$, subject to the initial condition
\begin{align} 
&\overline{\rho}(0,\cdot)=\overline{\rho}_0(\cdot) &\text{in }\Omega,
\label{initialCondSolvSubProTransform}  
\end{align} 
where the effective transport field is given by
\begin{equation}
\overline{\bu}_{eff} := \partial_t \bfPsi_{\zeta}^{-1}\circ \bfPsi_{\zeta} + \dfrac{1}{J_{\zeta} } \mathbf{B}_{\zeta}^\intercal \overline\bu,
\end{equation}
and the local rate of compression -- that is, the transformed divergence of the Eulerian velocity field $\bu$ -- is defined as
\begin{equation}
\delta_{\Div} := \dfrac{1}{J_{\zeta} } \mathbf{B}_{\zeta} \colon \nabla\overline\bu.
\end{equation}
To proceed rigorously, we first introduce our notion of a strong solution.

\begin{definition} 
\label{def:strsolmartFP}
Let the data triple  $(\overline{\rho}_{0}, \overline\bu, \zeta)$ satisfy:
\begin{equation}
\begin{aligned}
\label{fokkerPlanckDataAlone}
&\overline{\rho}_{0} \in  W^{3,2}( \Omega ), 
\\&
\overline\bu\in
W^{2,2}\big( I;L^2(\Omega)  \big) \cap L^2\big(I;W^{4,2}(\Omega )  \big),
\\
&\zeta\in W^{2,\infty}\big(I;W^{1,2}(\omega)  \big) \cap W^{3,2}\big(I;L^{2}(\omega)  \big)\cap L^{2}\big(I;W^{6,2}(\omega)  \big)
,
\\& \overline\bu  \circ \bm{\varphi} =(\partial_t\zeta)\bn
\quad \text{on }I \times \omega,  \quad\|\zeta\|_{L^\infty(I\times\omega)} < L,  \;\; \text{with}  \;\;  \bm{\varphi} \colon \omega \to \partial\Omega.
\end{aligned}
\end{equation}
We call
$\overline\rho$
a \textit{strong solution} of   \eqref{rhoEquAloneTransform}-\eqref{initialCondSolvSubProTransform}  with dataset $(\overline{\rho}_0, \overline\bu, \zeta)$ if 
\begin{itemize}
\item[(a)] $\overline\rho \in   W^{1,\infty} \big(I; W^{2,2}(\Omega ) \big)\cap L^\infty\big(I;W^{3,2}(\Omega )  \big)$; 
\item[(b)] Equation \eqref{rhoEquAloneTransform} holds a.e. in $I\times\Omega$.
\end{itemize}
\end{definition}

\begin{remark}
The regularity assumptions on $(\overline\bu, \zeta) $ in Definition~\ref{def:strsolmartFP} are chosen to ensure compatibility with the  momentum--structure subproblem (cf. Section~\ref{sec:MomStrucSubProb}). In particular, standard interpolation theory yields the continuous embeddings
\begin{align*}
&W^{2,2}\big(I;L^2(\Omega)\big) \cap L^2\big(I;W^{4,2}(\Omega)\big) 
\hookrightarrow X, 
\\
&W^{3,2}\big(I;L^2(\omega)\big) \cap L^2\big(I;W^{6,2}(\omega)\big) 
\hookrightarrow Y,
\end{align*}
for any choice of 
\begin{align*}
X &\in \Big\{ W^{1,\infty}\big(I;W^{1,2}(\Omega)\big),\quad W^{1,2}\big(I;W^{2,2}(\Omega)\big) \Big\},
\\
Y &\in \Big\{ W^{1,\infty}\big(I;W^{3,2}(\omega)\big),\quad W^{2,2}\big(I;W^{2,2}(\omega)\big),\quad W^{1,2}\big(I;W^{4,2}(\omega)\big) \Big\}.
\end{align*}
Hence, we immediately get all the desired regularities for $(\overline\bu, \zeta )$ to be a strong solution of the momentum--structure subproblem. 
\end{remark}

\noindent We now formulate our result on the existence of a unique strong solution to the pull-back continuity equation \eqref{rhoEquAloneTransform}--\eqref{initialCondSolvSubProTransform}.

\begin{theorem}\label{thm:mainFP}
Let  $(\overline{\rho}_0,  \overline\bu, \zeta)$ satisfy  \eqref{fokkerPlanckDataAlone}.
Then there exists a unique strong solution $\overline\rho $ to the pull-back continuity equation  \eqref{rhoEquAloneTransform}--\eqref{initialCondSolvSubProTransform}, in the sense of Definition ~\ref{def:strsolmartFP}, 
 such that
 \begin{equation}
\begin{aligned} \label{eq:ContSubProbEstimate}
\sup_{t\in I} &\Big( \Vert \overline{\rho}\Vert_{W^{3,2}(\Omega )}^2 
+
\Vert \partial_t\overline{\rho}\Vert_{W^{2,2}(\Omega )}^2 
\Big)
\\&\lesssim
 \Vert  \overline{\rho}_0\Vert_{W^{3,2}(\Omega)}^2  
\Bigg(1 + \sup\limits_I \Vert \partial_t \zeta \Vert_{W^{3,2}(\omega)}^2  + \int_I\Vert  \overline\bu
\Vert_{W^{4,2}(\Omega )}^2 \dt
\\
&\quad + 
\int_I\Vert  \partial_{t}^2 \overline\bu
\Vert_{L^{2}(\Omega)}^2 \dt
\Bigg)
  \exp{\bigg( c\int_I \big(\Vert \partial_t \zeta \Vert_{W^{4,2}(\omega)} + 
\Vert  \overline\bu\Vert_{W^{4,2}(\Omega)}\big)  \dt \bigg)}
\end{aligned}
\end{equation} 
holds for some constant $c = c(L) > 0$. 
\end{theorem}


\begin{proof}
The existence of a strong solution follows from the method of characteristics. To this end, for all $\bz \in \overline\Omega$, we define the characteristic flow map
\[\mathbf{\Phi} \colon (s, t, \bz) \longmapsto  \mathbf{\Phi}(s,t,\bz), \quad t, s \in \overline{I} \]
as the unique solution of the initial value problem: 
\begin{align}
\frac{\dd}{\ds}\mathbf{\Phi}(s,t,\bz) &= \overline{\bu}_{eff}\big(s ,\mathbf{\Phi}(s,t,\bz)\big),  \label{eq:CharacteristicEq}
\\
\mathbf{\Phi}(t,t,\bz) & =\bz. \label{eq:InitialConCharacteristicEq} 
\end{align}
Importantly, due to the regularity assumptions on $\zeta$ and $\overline\bu$, we have $ \overline{\bu}_{eff} \in L^{1}\big( I; W^{4,2}(\Omega)  \big)$, which ensures by the Cauchy-Lipschitz (Picard-Lindel\"of) theorem (see \cite[Chapter 2, Theorem 7.6]{amann2011ordinary} ), that \eqref{eq:CharacteristicEq}--\eqref{eq:InitialConCharacteristicEq} admits a unique solution.

\noindent Then, along the characteristic curves $s \longmapsto  \mathbf{\Phi}(s,t,\bz) $, the unique strong solution $\overline\rho$ to the PDE  \eqref{rhoEquAloneTransform}--\eqref{initialCondSolvSubProTransform} is given explicitly by
\begin{align}
 \overline\rho(t,\bz) =  \overline{\rho}_0\big(\mathbf{\Phi}(0,t,\bz)\big) \exp{ \bigg( -\int_0^t \delta_{\Div}\big(s,\mathbf{\Phi}(s,t,\bz)\big)\dd s \bigg)}. 
\end{align} 

\noindent Moreover, since $ \overline{\bu}_{eff} $ is a vector field tangent to $\partial\Omega $ -- a property ensured by the identity
\[\partial_t \bfPsi_{\zeta}^{-1}\circ  \bfPsi_{\zeta} = - \dfrac{1}{J_{\zeta}} \mathbf{B}_{\zeta}^\intercal \partial_t \bfPsi_\zeta, \]
  and the definition of  \emph{Hanzawa transform} $\bfPsi_\zeta$ (cf.~\eqref{eq:HanzawaT}), which in particular imply that $\overline{\bu}_{eff} = 0 $ on the boundary --  it follows from \cite[Appendix, Lemma A.6]{bourguignon1974remarks}  that    $ \mathbf{\Phi}  \in C\big( \overline{I} \times \overline{I} ;\,  W^{4,2}(\Omega )   \big).  $  Additionally,  for each $t, s \in \overline{I}, \; \mathbf{\Phi}(t, s, \cdot)  $ is a $C^1-$diffeomorphism from  $\overline{\Omega} $ onto itself, satisfying  $\mathbf{\Phi}(t, s, \partial\Omega) \subseteq \partial\Omega $. 

\noindent Although \cite[Appendix, Lemma A.6]{bourguignon1974remarks}  assumes a vector field in $C\big(I; W^{4,2}(\Omega_\zeta)\big)$, our setting only provides $ \overline{\bu}_{eff} \in L^{2}\big( I; W^{4,2}(\Omega) \big) $.  Nevertheless, since  $C\big(I; W^{4,2}(\Omega)\big) $  is dense in  $ L^{2}\big( I; W^{4,2}(\Omega) \big)$,  the result  can still be recovered by a standard approximation argument and the continuity of the flow map with respect to the transport field. 

\noindent Moreover, from \cite[Appendix, Lemma A.4 \& Lemma A.5]{bourguignon1974remarks}, we obtain that  $\overline{\rho}_{0} \circ \mathbf{\Phi}(0, t, \cdot)   \in W^{3,2}(\Omega)$ and, provided  $\delta_{\Div}(t,\cdot) \in W^{3,2}(\Omega)$  for all $t \in \overline{I}$,  the map  
\[\mathbf{\Phi} \longmapsto  \delta_{\Div}(t,\cdot)\circ \mathbf{\Phi} \]
is continuous from $\mathcal{D}^{3,2}(\Omega) $ into  $W^{3,2}(\Omega) $, where  
\[
\mathcal{D}^{3,2}(\Omega) :=  \Big\{ \mathbf{\Phi} \in W^{3,2}(\Omega) \;\big|\, \mathbf{\Phi} \text{ is a } C^{1}\text{-diffeomorphism from } \partial\Omega \text{ onto itself, with } \mathbf{\Phi}(\partial\Omega) \subseteq \partial\Omega \Big\}.
\]

\noindent Hence, $ \overline\rho\in C\big(\overline{I};W^{3,2}(\Omega)\big)$.\\

For the derivation of the estimate \eqref{eq:ContSubProbEstimate}, we consider the pull-back continuity equation  \eqref{rhoEquAloneTransform}, which yields 
 \begin{align}
\sup_I\Vert\partial_t \overline\rho \Vert_{W^{2,2}(\Omega)}^2
     &\lesssim
\sup_I \Big( \Vert (\overline{\bu}_{eff}\cdot \nabx) \overline\rho\Vert_{W^{2,2}(\Omega)}^2
+\Vert \delta_{\Div}\,\overline\rho 
\Vert_{W^{2,2}(\Omega)}^2 \Big)     \nonumber
\\&
\lesssim 
\sup_I\Vert \overline\rho \Vert_{W^{3,2}(\Omega)}^2
 \sup_I \Big( \Vert  \overline{\bu}_{eff}
\Vert_{W^{2,2}(\Omega)}^2  +   \Vert  \delta_{\Div}
\Vert_{W^{2,2}(\Omega)}^2  \Big)  \nonumber
\\&
\lesssim 
\sup_I\Vert \overline\rho \Vert_{W^{3,2}(\Omega)}^2
\Big(  \sup_I \Vert  \partial_t\zeta
\Vert_{W^{3,2}(\omega)}^2  +   \sup_I  \Vert  \overline\bu
\Vert_{W^{3,2}(\Omega)}^2  \Big)  \nonumber
\\&
\lesssim 
\sup_I\Vert \overline\rho \Vert_{W^{3,2}(\Omega)}^2  
\bigg(\sup_I \Vert  \partial_t\zeta
\Vert_{W^{3,2}(\omega)}^2  + \int_I\Vert  \overline\bu
\Vert_{W^{4,2}(\Omega)}^2\dt
+
\int_I\Vert  \partial_t^2\overline\bu
\Vert_{L^{2}(\Omega)}^2 \dt
\bigg),  \label{eq:rhoEquAloneFinalEstim}
 \end{align}
where \eqref{eq:rhoEquAloneFinalEstim} follows from the continuous embeddings
\begin{align*}
W^{2,2}\big(I;L^2(\Omega)  \big)\cap L^2\big(I;W^{4,2}(\Omega)  \big) 
\hookrightarrow
W^{1,2}\big(I;W^{3,2}(\Omega)\big)
    \hookrightarrow
    L^{\infty}\big((I;W^{3,2}(\Omega)\big). 
\end{align*} 
Furthermore, since for a.e. $t \in I$,  the pointwise estimate  \[ \Vert \delta_{\Div}\,\overline\rho (t)  \Vert_{W^{3,2}(\Omega)} \leq  \Vert \delta_{\Div}\Vert_{W^{3,2}(\Omega)}  \Vert \overline{\rho}(t)   \Vert_{W^{3,2}(\Omega)}  \quad  \text{holds},     \]  
 and  $\Vert \delta_{\Div} \Vert_{W^{3,2}(\Omega)} \in L^{1}(I)$, we may apply \cite[Chapter 3, Theorem 3.14 \& Remark 3.17]{BahouriCheminDanchin2011} to deduce
\begin{align}
  \sup_I \Vert  \overline\rho\Vert_{W^{3,2}(\Omega)}^2 
  \lesssim
  \Vert  \overline\rho_0\Vert_{W^{3,2}(\Omega)}^2 
  \exp{ \bigg( c\int_I \Big(
\Vert  \nabx\overline{\bu}_{eff} \Vert_{W^{3,2}(\Omega)}  + \Vert  \delta_{\Div} \Vert_{W^{3,2}(\Omega)}  \Big) \dt  \bigg) },
\end{align}
for some constant $c  > 0 .$ Using the structure of $ \overline{\bu}_{eff} $ and $\delta_{\Div}$, this further implies
\begin{align}\label{eq:rhoEquAloneFirstEstim}
  \sup_I \Vert  \overline\rho\Vert_{W^{3,2}(\Omega)}^2 
  \lesssim
  \Vert  \overline\rho_0\Vert_{W^{3,2}(\Omega)}^2 
  \exp{ \bigg( c\int_I \Big(
\Vert  \partial_t\zeta \Vert_{W^{4,2}(\Omega)}  + \Vert  \overline\bu \Vert_{W^{4,2}(\Omega)}  \Big) \dt  \bigg) },
\end{align}
where the constant $c = c(L) > 0$ depends on the tubular neighbourhood radius. 

\noindent Combining \eqref{eq:rhoEquAloneFinalEstim} and \eqref{eq:rhoEquAloneFirstEstim} yields the desired estimate.

\end{proof}
To control the nonlinear pressure term in the momentum equation, we first establish the following regularity estimate:

\begin{proposition}\label{prop:mainFP}
Let $\overline\rho $ be the unique strong solution  of  \eqref{rhoEquAloneTransform}-\eqref{initialCondSolvSubProTransform},
with dataset  $(\overline{\rho}_{0},  \overline\bu, \zeta)$, in the sense of Definition ~\ref{def:strsolmartFP}. Then, for
\begin{align*}
p(\overline\rho)=a\overline\rho^{\gamma}, \qquad a>0,\, \gamma > 1,
\end{align*}
it holds that
 \begin{equation}\label{eq:rhogammaregularity}
\begin{aligned} 
p(\overline\rho)  \in  W^{1,\infty} \big(I; W^{2,2}(\Omega ) \big)\cap L^\infty\big(I;W^{3,2}(\Omega )  \big).
\end{aligned}
\end{equation} \\
Moreover, the following estimate is satisfied  
\begin{equation}
\begin{aligned} \label{eq:pressureEstimate}
\sup_{t\in I} &\Big( \Vert \overline\rho^{\gamma}(t)\Vert_{W^{3,2}(\Omega )}
+
\Vert \partial_t\overline\rho^{\gamma}(t)\Vert_{W^{2,2}(\Omega )}
\Big)
\\&\lesssim
 \sup_{t\in I} \Big( \Vert \overline\rho(t)\Vert_{W^{3,2}(\Omega )}^{\gamma}
+
\Vert \partial_t\overline\rho(t)\Vert_{W^{2,2}(\Omega )}^{2 } + \Vert\overline\rho(t)\Vert_{W^{2,2}(\Omega )}^{2(\gamma-1) }
\Big) .
\end{aligned}
\end{equation}

\end{proposition}

\begin{proof}
Let $D \subset (0, \infty) $ be a bounded domain with smooth boundary such that for a.e. $ t \in I, \;  \bx \in \Omega $, we have  $\overline\rho(t, \bx ) \in D$.   
Furthermore, consider the map  \[ F \colon D \to \mathbb{R} ,  \quad r \longmapsto r^{\gamma}.  \]
Since $F \in C^{\infty}(D),  $ and  $ \overline\rho(t) \in W^{3,2}(\Omega) $ for a.e. $ t \in I, $ it follows from \cite[Section 2, Lemma A.2]{bourguignon1974remarks} that 
\[\overline\rho^{\gamma}(t) = F\circ\overline\rho(t)  \in W^{3,2}(\Omega), \; \text{for a.e. } \; t \in I.   \]
Thus,  \[ \overline\rho^{\gamma} \in L^{\infty}\left( I; W^{3,2}(\Omega) \right).   \]
Moreover,   for a.e. $t \in I, \; \overline{\rho}^{\gamma - 1}(t)   \in W^{2,2}(\Omega) $.  By the Banach algebra property of  $W^{2,2}(\Omega), $ it follows that 
\[\overline\rho^{\gamma - 1}(t) \partial_t \overline\rho(t) \in W^{2,2}(\Omega), \; \text{for a.e. } \, t \in I. 
 \]
Thus  \[ \partial_t \overline\rho^{\gamma} \in L^{\infty}\left( I; W^{2,2}(\Omega) \right),   \]
Therefore,
\[\overline\rho^{\gamma} \in W^{1, \infty}\left( I; W^{2,2}(\Omega) \right)\cap L^{\infty}\left( I; W^{3,2}(\Omega) \right) .   \]

To derive the estimate \eqref{eq:pressureEstimate},  observe that each derivative of $ \overline\rho^{\,\gamma}(t) $ up to order 2 or 3 is a  sum of products of  powers and derivatives of $\overline\rho(t) $ , were all but the highest-order term  are bounded. Hence each term is a product of a bounded and $L^{2}(\Omega)-$function.

\end{proof}


\section{The momentum-structure subproblem}\label{sec:MomStrucSubProb}
\noindent Given a scalar density   
\begin{equation}\label{eq:FixedDensitySpace}
 \overline{\varrho} \in L^\infty\big(I;W^{3,2}(\Omega)\big)\cap W^{1,\infty}\big( I;W^{2,2}(\Omega)\big), 
 \end{equation}
we study the well-posedness of the coupled momentum-structure subproblem with unknowns $( \bv, \eta) $.
In contrast to the continuity equation, whose analysis proceeds via the method of characteristics in the reference configuration $\Omega$, for the momentum-structure subproblem, we rely on the classical \textit{linearisation--fixed-point} framework through which we establish  existence and uniqueness of solutions.

\begin{definition}\label{def:strongSolution}
Let the dataset $(\overline{\varrho}, \eta_0, \eta_*, \bv_0)$ satisfies \eqref{eq:FixedDensitySpace},  \eqref{eq:InitialCondSpace} and \eqref{eq:InitialCondInterface}. 
We call a pair $( \bv,  \eta ) $ a strong solution to the resulting momentum--structure subproblem  \eqref{eq:ShellEq}--\eqref{interfaceCond}  with data $(\overline{\varrho}, \eta_0, \eta_*, \bv_0)$ provided that the following conditions hold:
\begin{itemize}

\item[(a)] The structure displacement $\eta $ satisfies 
\begin{align*}
&\eta \in L^2\big(I; W^{6,2}(\omega)\big)  
\cap W^{3,2}\big(I ; L^{2}(\omega)\big)
  \cap W^{2,\infty}\big(I; W^{1,2}(\omega)\big),
\\
&\partial_t\eta \in L^\infty\big(I; W^{3,2}(\omega)\big) \cap L^{2}\big(I; W^{4,2}(\omega)\big). 
\end{align*}

\item[(b)]  The velocity field $\bv$ satisfies 
\begin{align*}
&\bv\in L^{2}\big(I; W^{4,2}(\Omega_\eta)\big)
\cap W^{2,2}\big(I; L^{2}(\Omega_\eta)\big). 
\end{align*}

\item[(c)]  The momentum--structure subsystem \eqref{eq:ShellEq}--\eqref{interfaceCond} -- with prescribed density $\overline{\varrho}$ -- holds a.e. in $I\times \Omega_\eta$. 

\end{itemize}

\end{definition}

\subsection{Transformation to the reference domain}\label{locaosec}
For a solution $( \varrho, \bv,  \eta )$  of \eqref{eq:ShellEq}--\eqref{interfaceCond}, we set the new function
$
\overline{\bv}=\bv\circ \bfPsi_\eta$, while $\overline{\varrho}=\varrho\circ \bfPsi_\eta.$
Moreover, we define
\begin{equation*}\label{matrices}
\begin{aligned}
\mathbf{A}_\eta=J_\eta \big(\nabx \bfPsi_\eta^{-1}\circ \bfPsi_\eta\big)\big( \nabx \bfPsi_\eta^{-1}\circ \bfPsi_\eta \big)^\intercal,&\\
\mathbf{B}_\eta=J_\eta \left(\nabx \bfPsi_\eta^{-1}\circ \bfPsi_\eta\right)^\intercal,
&\\
\mathbf{h}_\eta(\overline{\bv}) = \left(J_{\eta_0}\overline{\varrho}_0 - J_{\eta}\overline{\varrho} \right)\partial_t\overline{\bv}-J_\eta\overline{\varrho}\,\nabla\overline{\bv}\cdot\partial_t\bfPsi_\eta^{-1}\circ\bfPsi_\eta - \overline{\varrho}\, \overline{\bv}\big( \nabx\overline{\bv} \colon \mathbf{B}_\eta \big),&
\\
\mathbf{H}_\eta(\overline{\bv})
=\mu\left(\mathbf{A}_{\eta_0}-\mathbf{A}_{\eta}\right)\nabla\overline{\bv} + (\lambda+\mu)\left[ \dfrac{1}{J_{\eta_0}}\left(\mathbf{B}_{\eta_0}\colon \nabla\overline\bv \right)\mathbf{B}_{\eta_0} -  \dfrac{1}{J_\eta}\left(\mathbf{B}_\eta\colon\nabla\overline\bv \right)\mathbf{B}_\eta  \right]  + a\big(\mathbf{B}_{\eta} - \mathbf{B}_{\eta_0}  \big)\overline{\varrho}^\gamma,
\end{aligned}
\end{equation*}
where $J_\eta=\mathrm{det}(\nabla\bfPsi_\eta)$.

\noindent To implement the linearisation--maximal-regularity fixed-point scheme, we rewrite the system on the fixed reference domain. The following lemma recasts the momentum--structure subsystem on the fixed reference domain through \(\bfPsi_\eta\).  Its derivation  is a routine change-of-variables computation, we  therefore, omit the proof.

\begin{lemma}\label{lem:EquivProblem}
	Suppose that the dataset
	$(\overline{\varrho}, \eta_0, \eta_*, \bv_0)$
	satisfies  \eqref{eq:FixedDensitySpace}, \eqref{eq:InitialCondSpace} and \eqref{eq:InitialCondInterface}.
	Then $( \bv,  \eta )$ is a strong solution to the momentum--structure subproblem \eqref{eq:ShellEq}--\eqref{interfaceCond} in the sense of Definition \ref{def:strongSolution}, if and only if $( \overline{\bv}, \eta  )$ is a strong solution of
	\begin{align}
	J_{\eta_0}\overline{\varrho}_0 \,\partial_t\overline{\bv} - \Div\left[\mu \mathbf{A}_{\eta_0}\nabla\overline{\bv} + \dfrac{ \lambda + \mu }{J_{\eta_0}} \left(\mathbf{B}_{\eta_0}\colon \nabla\overline\bv\right)\mathbf{B}_{\eta_0} - a\mathbf{B}_{\eta_0} \overline{\varrho}^\gamma \right] 
	= \mathbf{h}_\eta(\overline{\bv})-\Div{\mathbf{H}_\eta(\overline{\bv})} ,
	\label{momEqAloneBar}
	\\
	\partial_t^2\eta - \partial_t\Dely \eta + \Dely^2\eta
	=
   \bn^\intercal \left[\mathbf{H}_\eta( \overline{\bv}) - \mu \mathbf{A}_{\eta_0}\nabla\overline{\bv} - \dfrac{\lambda + \mu}{J_{\eta_0}}\left(\mathbf{B}_{\eta_0}\colon \nabla\overline\bv \right)\mathbf{B}_{\eta_0}  + a \mathbf{B}_{\eta_0}\overline{\varrho}^\gamma  \right]\circ\bm{\varphi} \bn ,
	\label{shellEqAloneBar}
	\end{align}
	with  $\overline{\bv}  \circ \bm{\varphi}  =(\partial_t\eta)\bn$ on $I\times \omega$.
\end{lemma}
\begin{proof}
The proof is deferred to Appendix \ref{appendix:proof}.
\end{proof}

\noindent We now state the main result of this section:

\begin{theorem}\label{thm:transformedSystem1}
Let  the dataset
	$(\overline{\varrho}, \eta_0, \eta_*, \bv_0)$
	satisfy  \eqref{eq:FixedDensitySpace}, \eqref{eq:InitialCondSpace} and \eqref{eq:InitialCondInterface}. Then there exists  $T_{**}  \in I,$ such that \eqref{momEqAloneBar}--\eqref{shellEqAloneBar}  admits a unique strong solution $(\overline{\bv}, \eta)$ on  $I_{**} := (0, T_{**}].$
\end{theorem}

\subsection{Linearised system}
Suppose that the data $(\overline{\varrho}, \bv_0, \eta_0, \eta_*)$ satisfies
\begin{align}
	&\overline{\varrho} \in L^\infty\big(I;W^{3,2}(\Omega)\big)\cap W^{1,\infty}\big( I;W^{2,2}(\Omega)\big),
	\; \tag{\ref{eq:FixedDensitySpace}} 
	\\
	& \eta_0\in W^{5,2}(\omega), \qquad \eta_*\in W^{3,2}(\omega),\qquad \bv_0 \in W^{3,2}(\Omega_{\eta_0}). \label{initialcond}
\end{align}
Let the functions $(\mathbf{h}, \mathbf{H})$ satisfy
\begin{align}
	&\mathbf{h} \in L^2(I;W^{2,2}(\Omega))\cap W^{1,2}(I;L^{2}(\Omega)), \qquad
	\mathbf{h}(0)\in {W^{1,2}(\Omega)}, \nonumber
	\\&
	\mathbf{H} \in L^2(I;W^{3,2}(\Omega))\cap W^{1,2}(I;W^{1,2}(\Omega)), 
	\qquad \mathbf{H}(0) = 0,   \label{sourcecond}
\end{align}
where $\mathbf{h}(0)$ and $\mathbf{H}(0)$  are  the initial values of $\mathbf{h}$ and $\mathbf{H}$ respectively.  Now we linearise the system \eqref{momEqAloneBar}--\eqref{shellEqAloneBar} as follows:
\begin{align}
	J_{\eta_0}\overline{\varrho}_0\,\partial_t\overline{\bv} - \Div\left[\mu \mathbf{A}_{\eta_0}\nabla\overline{\bv} + \dfrac{ \lambda + \mu }{J_{\eta_0}} \left(\mathbf{B}_{\eta_0}\colon \nabla\overline\bv\right)\mathbf{B}_{\eta_0} - a\mathbf{B}_{\eta_0} \overline{\varrho}^\gamma \right]  
	= \mathbf{h}-\Div{\mathbf{H}} , \label{linMomAlone}
	\\
	\partial_t^2\eta - \partial_t\Dely \eta + \Dely^2\eta
	=
   \bn^\intercal\left[\mathbf{H} - \mu \mathbf{A}_{\eta_0}\nabla\overline{\bv} - \dfrac{\lambda + \mu}{J_{\eta_0}}\left(\mathbf{B}_{\eta_0}\colon \nabla\overline\bv \right)\mathbf{B}_{\eta_0}   + a\mathbf{B}_{\eta_0}\overline{\varrho}^\gamma \right]\circ\bm{\varphi} \bn ,\label{linShelAlone}
	\end{align}
with  $\overline{\bv}  \circ \bm{\varphi}  =(\partial_t\eta)\bn$ on $I\times \omega$. \\

As a preliminary step for the fixed-point scheme, we establish well-posedness for the system \eqref{linMomAlone}–\eqref{linShelAlone}. 
The following proposition provides existence, uniqueness, and the required maximal-regularity estimate under compatible data assumptions specified therein.

\begin{proposition}
	\label{thm:transformedSystem}
Suppose that the dataset $(\overline{\varrho}, \mathbf{\bv}_0, \eta_0, \eta_*, \mathbf{h}, \mathbf{H})$ satisfy \eqref{eq:FixedDensitySpace}--\eqref{sourcecond} supplemented by  the compatibility condition
 \begin{equation}\label{eq:CC}
 \begin{aligned}
  \Bigg( \dfrac{1}{\overline{\varrho}_{0}J_{\eta_0}} \Big(  \Div\big( \overline{\bm{\tau} }_0 -  \mathbf{H}(0) \big)   + \mathbf{h}(0)  \Big)  \Bigg)\circ \bm{\varphi} 
 & =  \Bigg( \Dely\eta_{*} - \Dely^{2}\eta_{0}  + \bn^\intercal\Big(  \mathbf{H}(0)  - \overline{\bm{\tau} }_0 \Big)\circ\bm{\varphi}\bn   \Bigg)\bn  \quad \text{on} \;\; \omega,
  \end{aligned}  \tag{CC}
 \end{equation}\\
 where 
 \[  \overline{\bm{\tau} }_0 :=  \mu \mathbf{A}_{\eta_0}\nabla\overline{\bv}_0 + \dfrac{ \lambda + \mu }{J_{\eta_0}} \big(\mathbf{B}_{\eta_0}\colon \nabla\overline{\bv}_0\big)\mathbf{B}_{\eta_0}  - a\overline{\varrho}_{0}^{\gamma}\mathbf{B}_{\eta_0} . \]  
	Then there exists a unique  strong solution $( \overline{\bv}, \eta  )$ of \eqref{linMomAlone}--\eqref{linShelAlone} satisfying
\begin{equation}
\begin{aligned}
& \left(\lambda + 2\mu\right)\sup_I\int_\Omega|\partial_t\nabla\overline\bv|^2\dx +\sup_I \int_\omega\left(|\partial_t^2\naby\eta|^2+|\partial_t\naby\Dely\eta|^2\right)\dy + \int_I\int_\Omega |\partial_t^2\overline\bv|^2 \dx\dt
\\
&\qquad   + \left(\lambda + 2\mu\right)\int_I\int_\Omega \left( |\partial_t\nabla^{2}\overline\bv|^{2}  + |\Delta^{2}\overline\bv|^{2} \right)\dx\dt
\\
&\qquad  +\int_I\int_\omega\left(|\partial_t^2\Dely\eta|^2+|\partial_t^3\eta|^2+|\partial_t\Dely^2\eta|^2+|\Dely^3\eta|^2\right)\dy\dt 
\\
& \lesssim 
\Vert \eta_0\Vert_{W^{5,2}(\omega)}^2
+
\Vert \eta_*\Vert_{W^{3,2}(\omega)}^2
+
\Vert \overline{\varrho}_0\Vert_{W^{3,2}(\Omega)}^2
+
\Vert\overline\bv_0\Vert_{W^{3,2}(\Omega)}^2  +
\Vert \mathbf{h}(0)\Vert_{W^{1,2}(\Omega)}^2
+
\Vert \mathbf{H}(0)\Vert_{W^{2,2}(\Omega)}^2
\\
&\qquad +\int_I \left(
\Vert \mathbf{h}\Vert_{W^{2,2}(\Omega)}^2
+
\Vert \mathbf{H}\Vert_{W^{3,2}(\Omega)}^2
+
\Vert\partial_t \mathbf{h}\Vert_{L^{2}(\Omega)}^2
+
\Vert\partial_t\mathbf{H}\Vert_{W^{1,2}(\Omega)}^2  \right)
\\
& \qquad +  \Vert \overline{\varrho}^{\gamma}(t)\Vert_{L^{2}\big( I; W^{3,2}(\Omega ) \big) }^2
+
\Vert \partial_t\overline{\varrho}^{\gamma}(t)\Vert_{L^{2}\big(I; W^{1,2}(\Omega ) \big)}^2   
 \label{linearestimate}
\end{aligned}
\end{equation}
	\end{proposition}


\begin{proof}
 Following  \cite[Section 3.2]{breit2023ladyzhenskaya},  we construct the solution with a two-level $\epsilon$-regularised Galerkin scheme. We first add to the left-hand side of the shell equation the structural corrector $\epsilon\partial_t\Dely^2 \eta$ with $0 <\epsilon \ll 1.$ This provides the extra damping  needed to control the higher-order time--boundary terms that arise when the shell equation is time-differentiated and tested with $\partial_{t}^3\eta$ (cf.~\eqref{eq:biLaplacian}), and it yields uniform bounds at the discrete level without invoking  Stokes theory.  We then build a finite-dimensional approximation that encodes the kinematic boundary condition and remove the one-dimensional kernel spanned by the constant function. Concretely, the spatial parts of the shell operator  ($\Dely$ and  $\Dely^2$) annihilate constants, so the displacement $\eta$ is determined only up to an additive time-dependent constant. To restore coercivity and uniqueness we work in the weighted zero-mean subspace
\[\mathcal{Z}_0 :=  \left\{ \xi \in W^{2,2}(\omega) \colon \int\limits_{\partial\Omega} \xi\circ\varphi^{-1} \dH = 0  \right\} =  \left\{ \xi \in W^{2,2}(\omega) \colon \int\limits_{\omega} \xi \sigma \dy  = 0 \right\},  \]
where $\sigma = |\partial_1\varphi \times \partial_2 \varphi|.$

\noindent Let $\{\psi_j\}_{j\in \mathbb{N}}$ be  $L^{2}(\omega)$-orthonormal eigenfunctions of $-\Dely$ on $\omega$ with periodic boundary conditions. We discard the constant eigenfunction $\psi_0 \equiv 1$  and,  if necessary project  the remaining eigenfunctions to $\mathcal{Z}_0 .$ After a Gram-Schmidt step (if necessary), we obtain  an $L^{2}(\omega)$-orthonormal basis $\{\psi_j\}_{j\in \mathbb{N}} \subset \mathcal{Z}_0.$
We approximate the boundary velocity (not $\eta$ itself) by 
\[\partial_t \eta_N (t, \by) = \sum\limits_{j = 1}^{N} p_j(t) \psi_j (\by) , \quad \text{so that} \quad \eta_N (t, \by) =  \eta_0(\by) +  \sum\limits_{j = 1}^{N} \left( \int_{o}^{t} p_j(s) \ds\right) \psi_j(\by),  \]
for some time-dependent coefficients $\{p_j\}_{j\in \mathbb{N}}.$
In order to incorporate  the kinematic boundary condition \eqref{interfaceCond}  in the finite-dimensional approximation space, we use a bounded right-inverse of the trace operator 
\[\mathcal{E} \colon W^{2,2}(\partial\Omega) \to  W^{5/2, 2}(\Omega), \quad \text{such that } \quad  \mathcal{E} (v)_{\mid_{\partial\Omega}} = v , \]
and define the lifted boundary basis functions  $\bm{\psi}_j :=  \mathcal{E} \big( \left(\psi_j\bn \right)\circ \bm{\varphi}^{-1} \big). $ 
Moreover, let $\{\bm{\phi}_j\}_{j\in \mathbb{N}} \subseteq W^{1,2}_{0}(\Omega)$  be $L^{2}(\Omega)$-orthonormal eigenfunctions of the Dirichlet Laplacian.
We then enumerate the fluid velocity basis as  $\{ \bm{\Lambda}_j\}_{j\in \mathbb{N}} = \{\bm{\phi}_j\}_{j\in \mathbb{N}}  \cup \{\bm{\psi}_j\}_{j\in \mathbb{N}}, $ and adopt the single-sum Galerkin ansatz 
\[\bv_N (t, \bx) =  \sum\limits_{j = 1}^{N} a_j (t) \bm{\Lambda}_j (\bx),   \]
with time-dependent coefficients $a_j \colon I \to\mathbb{R}. $
Let $\mathcal{J}_{\partial} \subset\{1,\dots,N\} $ denote the index set of lifted boundary basis functions, so that 
\[ \bm{\Lambda}_j = \bm{\psi}_j \quad \text{for all }\; j \in \mathcal{J}_\partial, \]
while $ \bm{\Lambda}_j = \bm{\phi}_j  $ for all $j \notin  \mathcal{J}_\partial.$
Accordingly,  we make the identification  
\[a_j(t) = p_j(t) \quad \text{for all } \; j \in  \mathcal{J}_\partial \;\; \text{and all }\; t \in I.  \]
Hence, by construction,  the kinematic boundary condition holds at the discrete level,  
\[\bv_N \circ \bm{\varphi} = (\partial_t \eta_N) \bn \ \  \text{on} \ \omega \ \ \text{for all }t\in I .\]

\noindent To streamline notation, we define the bulk inner product and bulk stiffness 
\[m_{\Omega}(\bu, \bv) :=  \int\limits_{\Omega} J_{\eta_0}\overline{\varrho}_0 \bu\cdot \bv \dx, \quad  
b_{\Omega}(\bu, \bv) :=  \int\limits_{\Omega} \left( \mu \mathbf{A}_{\eta_0}\nabla\bu\colon \nabla\bv  + \dfrac{\lambda + \mu}{J_{\eta_0}} (\mathbf{B}_{\eta_0} \colon \nabla\bu)(\mathbf{B}_{\eta_0} \colon \nabla\bv)   \right) \dx . \]
We further define the boundary traction pairing and data pairings
\[t_{\omega}(\bv, \phi) :=  \int\limits_{\omega} \bn^\intercal \left(  \mu \mathbf{A}_{\eta_0}\nabla\bv +   \dfrac{\lambda + \mu}{J_{\eta_0}} (\mathbf{B}_{\eta_0} \colon \nabla\bv)\mathbf{B}_{\eta_0} \right)\circ \bm{\varphi} \bn \phi \dy, \quad l_{\omega}(\phi)  :=  \int\limits_{\omega} \big( \mathbf{H} +  a\mathbf{B}_{\eta_0}\overline{\varrho}^{\gamma} \big)\circ \bm{\varphi}  \bn \phi \dy,  \]
\[l_{\Omega}(\bv) := \int\limits_{\Omega} \big( \mathbf{h}\cdot\bv + \mathbf{H}\colon \nabla\bv + a\mathbf{B}_{\eta_0}\overline{\varrho}^{\gamma} \colon \nabla\bv   \big) \dx.  \]
Similarly, we introduce the forms 
\[m_\eta (\eta, \phi) := \int\limits_\omega \eta\phi \dy, \quad  d_{\eta} (\eta, \phi) := \int\limits_\omega \naby\eta\cdot \naby\phi \dy, \quad k_\eta (\eta, \phi) :=  \int\limits_\omega \Dely\eta\Dely\phi \dy, \quad b_\eta (\eta, \phi) := \int\limits_\omega\Dely \eta\Dely\phi \dy.  \]
Testing the momentum equation with $\bm{\Lambda}_j$ and the shell equation with $\psi_j$ yields, for $j = 1, \ldots, N, $  
\begin{align*}
m_{\Omega}(\partial_t \bv_N, \bm{\Lambda}_j) + b_{\Omega} (\bv_N, \bm{\Lambda}_j) = l_{\Omega}(\bm{\Lambda}_j) + t_{\omega}\big(\bv_N, (\bm{\Lambda}_j\circ\bm{\varphi} )  \big) - l_{\omega}( \bm{\Lambda}_j\circ\bm{\varphi} ),
\\
m_\eta \left( \partial_{t}^2 \eta_N, \psi_j \right) + d_{\eta} \left( \partial_t \eta_N, \psi_j \right)  + k_\eta (\eta_N, \psi_j) +  \epsilon b_\eta \left(\partial_t \eta_N, \psi_j \right)  =   l_{\omega}(\psi_j) - t_{\omega}(\bv_N, \psi_j).
\end{align*}
Introduce the matrices and the load vectors by setting  ( for  $i, j \in \{1, \ldots, N\}$): 
\[\left[\mathbf{M}_{\bv} \right]_{i,j} =  m_{\Omega}\left( \bm{\Lambda}_i , \bm{\Lambda}_j  \right), \quad \left[\mathbf{K}_{\bv} \right]_{i,j} =  b_{\Omega}\left( \bm{\Lambda}_i , \bm{\Lambda}_j  \right), \quad \left[\mathbf{T}_{\omega} \right]_{i,j} = t_{\omega}\left( \bm{\Lambda}_j , \bm{\Lambda}_i\circ\bm{\varphi}  \right),  \]
\[\left[\mathbf{F}_{\Omega} \right]_{i} = l_{\Omega}(\bm{\Lambda}_i), \quad  \left[\mathbf{F}_{\omega} \right]_{i} = l_{\omega}(\bm{\Lambda}_i\circ\bm{\varphi}  ), \quad  \left[\mathbf{L}_{\omega} \right]_{i} = l_{\omega}(\psi_i ),  \]
\[\left[\mathbf{M}_{\eta} \right]_{i,j} = m_\eta ( \psi_i,\psi_j ), \quad \left[\mathbf{D}_{\eta} \right]_{i,j} = d_\eta ( \psi_i,\psi_j ), \quad \left[\mathbf{K}_{\eta} \right]_{i,j} =  k_\eta ( \psi_i,\psi_j ) ,  \]
\[\left[\mathbf{C}_{\eta} \right]_{i,j} = b_\eta ( \psi_i,\psi_j ) ,  \quad  \left[\mathbf{R}_{\eta} \right]_{j}  =  \int\limits_{\omega}  \Dely\eta_0 \Dely\psi_j \dy.  \]
Let $\mathbf{a}(t) = \left(a_1(t),\ldots, a_N(t) \right)^\intercal, \; \mathbf{p}(t) = \left( p_1(t), \ldots, p_N(t) \right)^\intercal$ and $\mathbf{q}(t) = \int_{0}^{t} \mathbf{p}(s)\ds = \left( q_1(t), \ldots, q_N(t)  \right)^\intercal .$ Then the Galerkin coefficients satisfy
\begin{align}\label{eq:MomenGalerkin}
\mathbf{M}_{\bv} \dfrac{\dd \mathbf{a}}{\dt} + \left( \mathbf{K}_{\bv} - \mathbf{T}_{\omega} \right)\mathbf{a} =   \mathbf{F}_{\Omega}  - \mathbf{F}_{\omega} 
\end{align}
\begin{align}\label{eq:EtaGalerkin}
\mathbf{M}_{\eta} \dfrac{\dd^2 \mathbf{q}}{\dt} + \left(\mathbf{D}_{\eta} + \epsilon \mathbf{C}_{\eta} \right) \dfrac{\dd \mathbf{q}}{\dt} + \mathbf{K}_{\eta} \mathbf{q} = \mathbf{L}_{\omega} - \mathbf{T}_{\omega}{\mathbf{a}}  - \mathbf{R}_{\eta}. 
\end{align}
For $\bm{z}(t) = \left(\mathbf{a}(t),  \mathbf{p}(t),   \mathbf{q}(t) \right)^\intercal ,$ the semi-discrete Galerkin system \eqref{eq:MomenGalerkin}--\eqref{eq:EtaGalerkin}  is equivalent to a first-order linear ODE
\begin{align}\label{eq:SemiDiscrete}
\dfrac{\dd \bm{z}(t) }{\dt} =  \mathbf{G}_\epsilon \bm{z}(t) + \bm{f} (t).
\end{align}
The mass matrices $\mathbf{M}_{\bv},\, \mathbf{M}_{\eta}$ are symmetric positive definite, so $\mathbf{G}_\epsilon$ is well-defined. Indeed, all blocks are time-independent and bounded, and $\bm{f} $ inherits the assumed regularity of the data. Hence, the right-hand side is globally Lipschitz in $\bm{z}, $  and the Cauchy--Lipschitz (Picard--Lindel\"of) theorem yields a unique $C^{1}$-in-time solution for the Galerkin coefficients for any admissible  initial data. 
In particular, every operation  used in the subsequent steps -- time differentiation, testing with high-order time derivative, and integration by parts -- is rigorously justified at the discrete level. 

\noindent This is the framework within which the subsequent steps establish the estimate \eqref{linearestimate} and enable the limit $N \to \infty$ and finally $\epsilon \to 0.$

\medskip

 
\noindent \textit{Step 1:} differentiating \eqref{linMomAlone} and \eqref{linShelAlone}   with respect to time,  we deduce that  
\[ \widehat \eta = \partial_t \eta , \quad \widehat \bv = \partial_t \overline{\bv} \]
satisfy

 \begin{equation}\label{eq:LinShellEqTime}
\left\{\begin{aligned}
& \partial_t^2 \widehat\eta - \partial_t\Dely \widehat\eta + \Dely^2 \widehat\eta =  \bn^\intercal \left[\partial_t\mathbf{H} -\mu \mathbf{A}_{\eta_0}\nabla\widehat{\bv} - \dfrac{ \lambda + \mu }{J_{\eta_0}} \left(\mathbf{B}_{\eta_0}\colon \nabla\widehat{\bv}\right)\mathbf{B}_{\eta_0} + a\mathbf{B}_{\eta_0}\partial_t\overline{\varrho}^\gamma \right]\circ\bm{\varphi} \bn
&& \text{in }  I\times\omega
,\\
&\widehat\eta(0,\by) = \widehat\eta_0(\by), \quad (\partial_t \widehat\eta)(0, \by) = \widehat\eta_*(\by)
&& \text{} \forall\, \by \in \omega,
\end{aligned}\right.
\end{equation}

 \begin{equation}\label{eq:LinMomentEqTime}
\left\{\begin{aligned}
&J_{\eta_0}\overline{\varrho}_0\,\partial_t\widehat{\bv}-\Div\left[\mu \mathbf{A}_{\eta_0}\nabla\widehat{\bv} + \dfrac{ \lambda + \mu }{J_{\eta_0}} \left(\mathbf{B}_{\eta_0}\colon \nabla\widehat{\bv}\right)\mathbf{B}_{\eta_0} - a\mathbf{B}_{\eta_0}\partial_t\overline{\varrho}^\gamma \right] 
	= \partial_t \mathbf{h} - \Div{\partial_t \mathbf{H}} \phantom{ + \widehat\eta =   }
&& \text{in }   I\times\Omega,
\\
& \widehat \bv(0, \bx) = \widehat \bv_{0}(\bx)
&&  \text{} \forall\,  \bx \in \Omega,
\end{aligned}\right.
\end{equation}\\
where 
\begin{equation*}
\begin{aligned}
\widehat{\eta}_0 &= \eta_{*}, \quad
\widehat{\eta}_{*} = \Dely\eta_* - \Dely^{2}\eta_0 
+ \bn^\intercal  \left[\mathbf{H}(0) -\mu \mathbf{A}_{\eta_0}\nabla\overline{\bv}_0 - \dfrac{ \lambda + \mu }{J_{\eta_0}} \left(\mathbf{B}_{\eta_0}\colon \nabla\overline{\bv}_0\right)\mathbf{B}_{\eta_0} + a\mathbf{B}_{\eta_0}\overline{\varrho}_{0}^\gamma \right]\circ\bm{\varphi}\bn, \\
\widehat{\bv}_{0} &= \dfrac{1}{J_{\eta_0}\overline{\varrho}_0}\left( 
\Div \left[\mu \mathbf{A}_{\eta_0}\nabla\overline{\bv}_0 + \dfrac{ \lambda + \mu }{J_{\eta_0}} \left(\mathbf{B}_{\eta_0}\colon \nabla\overline{\bv}_0\right)\mathbf{B}_{\eta_0} - a\mathbf{B}_{\eta_0}\overline{\varrho}_{0}^\gamma \right]  + \mathbf{h}(0) - \Div\mathbf{H}(0) \right).
\end{aligned}
\end{equation*}\\
Consider the pair of test functions $\big( \partial^{2}_{t}\widehat\eta, \partial_{t}\widehat{\bv} \big) $ for \eqref{eq:LinShellEqTime} and \eqref{eq:LinMomentEqTime}, we get respectively 
\begin{equation}\label{eq:weakformTimeEta}
\begin{aligned}
&\int_I\int_\omega |\partial^{3}_{t}\eta|^{2} \dy\dt - \int_I\int_\omega \partial^{2}_{t} \Dely\eta \cdot \partial^{3}_{t}\eta \dy\dt +  \int_I\int_\omega \partial_{t}\Dely^{2}\eta \cdot \partial^{3}_{t}\eta \dy\dt \\
&= \int_I\int_\omega\bn^\intercal  \left[\partial_t \mathbf{H} -\mu \mathbf{A}_{\eta_0}\partial_t \nabla\overline{\bv} - \dfrac{ \lambda + \mu }{J_{\eta_0}} \left(\mathbf{B}_{\eta_0}\colon \partial_t \nabla\overline{\bv}\right)\mathbf{B}_{\eta_0} + a\mathbf{B}_{\eta_0}\partial_t\overline{\varrho}^\gamma\right]\circ\bm{\varphi}\bn \cdot \partial^{3}_{t}\eta \dy\dt,  
\end{aligned}
\end{equation}

and 
\begin{equation}\label{eq:weakformTimeV}
\begin{aligned}
& \int_I\int_\Omega J_{\eta_0}\varrho_0 |\partial^{2}_{t} \overline{\bv}|^2 \dx\dt -  \int_I\int_\Omega \Div\left[\mu \mathbf{A}_{\eta_0}\partial_t \nabla\overline{\bv} + \dfrac{ \lambda + \mu }{J_{\eta_0}} \left(\mathbf{B}_{\eta_0}\colon \partial_t \nabla\overline{\bv}\right)\mathbf{B}_{\eta_0} - a\mathbf{B}_{\eta_0}\partial_t\overline{\varrho}^\gamma \right] \cdot \partial^{2}_{t}\overline{\bv} \dx\dt \\
&= \int_I\int_\Omega \big(\partial_t \mathbf{h} - \partial_t \Div\mathbf{H}  \big) \cdot \partial^{2}_{t}\overline{\bv} \dx\dt.
\end{aligned}
\end{equation}\\

However, observe that  \begin{align}\label{eq:biLaplacian}
 - \int_I\int_\omega \partial_{t}\Dely^{2}\eta \cdot \partial^{3}_{t}\eta \dy\dt  = \int_I\int_\omega \partial_{t}\big( \partial_{t}\naby\Dely\eta \cdot \partial^{2}_{t}\naby\eta \big) \dy\dt  +   \int_I\int_\omega |\partial^{2}_{t}\Dely\eta|^2 \dy\dt , 
 \end{align}
and (due to the kinematic boundary condition \eqref{interfaceCond})  
\begin{equation*}
\begin{aligned}
& \int_I\int_\Omega \Div\left[ - \partial_t \mathbf{H} + \mu \mathbf{A}_{\eta_0}\partial_t \nabla\overline{\bv} + \dfrac{ \lambda + \mu }{J_{\eta_0}} \left(\mathbf{B}_{\eta_0}\colon \partial_t \nabla\overline{\bv}\right)\mathbf{B}_{\eta_0} - a\mathbf{B}_{\eta_0}\partial_t\overline{\varrho}^\gamma \right] \cdot \partial^{2}_{t}\overline{\bv} \dx\dt \\
& =  \int_I\int_{\partial\Omega}  \left(\left[ - \partial_t \mathbf{H} + \mu \mathbf{A}_{\eta_0}\partial_t \nabla\overline{\bv} + \dfrac{ \lambda + \mu }{J_{\eta_0}} \left(\mathbf{B}_{\eta_0}\colon \partial_t \nabla\overline{\bv}\right)\mathbf{B}_{\eta_0} - a\mathbf{B}_{\eta_0}\partial_t\overline{\varrho}^\gamma \right]  \partial^{2}_{t}\overline{\bv}\right)\cdot \bn \dH\dt \\
& \quad -   \int_I\int_\Omega \left( - \partial_t \mathbf{H} + \mu \mathbf{A}_{\eta_0}\partial_t \nabla\overline{\bv} + \dfrac{ \lambda + \mu }{J_{\eta_0}} \left(\mathbf{B}_{\eta_0}\colon \partial_t \nabla\overline{\bv}\right)\mathbf{B}_{\eta_0} -a\mathbf{B}_{\eta_0}\partial_t\overline{\varrho}^\gamma \right) \colon \partial^{2}_{t}\nabla\overline{\bv}  \dx\dt \\
& = \int_I\int_\omega  \bn^\intercal \left[ - \partial_t \mathbf{H} + \mu \mathbf{A}_{\eta_0}\partial_t \nabla\overline{\bv} + \dfrac{ \lambda + \mu }{J_{\eta_0}} \left(\mathbf{B}_{\eta_0}\colon \partial_t \nabla\overline{\bv}\right)\mathbf{B}_{\eta_0} - a\mathbf{B}_{\eta_0}\partial_t\overline{\varrho}^\gamma \right] \circ \bm{\varphi}\bn \cdot \partial^{3}_{t}\eta \dy\dt \\
& \quad -   \int_I\int_\Omega \left( - \partial_t \mathbf{H} + \mu \mathbf{A}_{\eta_0}\partial_t \nabla\overline{\bv} + \dfrac{ \lambda + \mu }{J_{\eta_0}} \left(\mathbf{B}_{\eta_0}\colon \partial_t \nabla\overline{\bv}\right)\mathbf{B}_{\eta_0} - a\mathbf{B}_{\eta_0}\partial_t\overline{\varrho}^\gamma \right) \colon \partial^{2}_{t}\nabla\overline{\bv}  \dx\dt.
\end{aligned}
\end{equation*}\\

Thus, adding \eqref{eq:weakformTimeEta} and \eqref{eq:weakformTimeV} together yields, 
\begin{equation}\label{eq:StepOnepriorEstim}
\begin{aligned}
&\int_I\int_\omega |\partial^{3}_{t}\eta|^{2} \dy\dt +  \dfrac{1}{2} \int_I \int_\omega  \partial_t \big( | \partial^{2}_{t} \naby\eta |^2 \big) \dy\dt  + \int_I\int_\Omega J_{\eta_0}\varrho_0 |\partial^{2}_{t} \overline{\bv}|^2 \dx\dt \\
&\quad +   \dfrac{\mu}{2} \int_I\int_\Omega \partial_t \big( \mathbf{A}_{\eta_0} \partial_{t}\nabla\overline{\bv}:\partial_{t}\nabla\overline{\bv} \big) \dx\dt  +  \dfrac{ \lambda + \mu }{2J_{\eta_0}}\int_I\int_\Omega \partial_t \left( |\mathbf{B}_{\eta_0}\colon \partial_t \nabla\overline{\bv}|^2  \right) \\
&=  \int_I\int_\omega \partial_{t}\big( \partial_{t}\naby\Dely\eta \cdot \partial^{2}_{t}\naby\eta \big) \dy\dt  +   \int_I\int_\omega |\partial^{2}_{t}\Dely\eta|^2 \dy\dt  +  \int_I\int_\Omega \partial_t \mathbf{h}  \cdot \partial^{2}_{t}\overline{\bv} \dx\dt \\
& \quad +  \int_I\int_\Omega \partial_t \mathbf{H}\colon \partial^{2}_{t}\nabla\overline{\bv}  \dx\dt   + \int_I\int_\Omega a\mathbf{B}_{\eta_0}\partial_t\overline{\varrho}^\gamma \colon \partial_t^{2} \nabx\overline{\bv} \dx\dt. 
\end{aligned}
\end{equation}

Due to the regularity of $\overline{\varrho}, $ the last term in \eqref{eq:StepOnepriorEstim} requires special attention.  Indeed,

\begin{align}  
\left| \int_I\int_\Omega a\mathbf{B}_{\eta_0}\partial_t\overline{\varrho}^\gamma \colon \partial_t^{2} \nabx\overline{\bv} \dx\dt \right| &\leq \int_I \int_{\partial\Omega} \left| \Big(a\mathbf{B}_{\eta_0}\partial_t\overline{\varrho}^\gamma  \partial_t^{2} \overline{\bv} \Big)\bn \right| \dH\dt  +  \int_I \int_\Omega \left| a \Div\left( \mathbf{B}_{\eta_0}\partial_t\overline{\varrho}^\gamma \right)  \partial_t^{2} \overline{\bv} \right| \dx\dt  \nonumber
\\
&\lesssim   \Vert \partial_t\overline{\varrho}^\gamma  \Vert_{L^{2}\left( I; L^{2}(\partial\Omega)  \right)}       \Vert \partial_t^{3} \eta \Vert_{L^{2}\left( I; L^{2}(\omega) \right)}   \nonumber
\\
&\quad + \Vert \partial_t\overline{\varrho}^\gamma \Vert_{L^{2}\left( I; W^{1,2}(\Omega)  \right)}  \Vert \partial_t^{2}\overline{\bv} \Vert_{L^{2}\left( I; L^{2}(\Omega)  \right)}  \nonumber
\\
& \lesssim \Vert \partial_t\overline{\varrho}^\gamma \Vert_{L^{2}\left( I; W^{1,2}(\Omega)  \right)}   \Big(  \Vert \partial_t^{3} \eta \Vert_{L^{2}\left( I; L^{2}(\omega) \right)}  +  \Vert \partial_t^{2}\overline{\bv} \Vert_{L^{2}\left( I; L^{2}(\Omega)  \right)} \Big), 
\end{align}

To handle the remaining terms on the right-hand side of \eqref{eq:StepOnepriorEstim}, we apply H\"older's inequality and Young's inequality, and make use of  the uniform boundedness of $J_{\eta_0}, \, \mathbf{B}_{\eta_0} $ and  $  \mathbf{A}_{\eta_0}. $ These yield the following estimate:
 
\begin{equation}\label{eq:priorFirstEstimate}
\begin{aligned}
&\int_I\int_\omega |\partial^{3}_{t}\eta|^{2} \dy\dt  + \int_I \int_\omega  \partial_t \left( | \partial^{2}_{t} \naby\eta |^2 \right) \dy\dt  +   \left(\lambda + 2\mu\right)\int_I\int_\Omega \partial_t \left(|\partial_t\nabla\overline\bv|^2\right) \dx\dt   +  \int_I\int_\Omega |\partial^{2}_{t} \overline{\bv}|^2 \dx\dt 
\\
&\lesssim  \overline{\kappa}_0 \sup_I \int_\omega|\partial_t^2\naby\eta|^2\dy   +  c(\overline{\kappa}_0) \sup_I \int_\omega |\partial_t\naby\Dely\eta|^2\dy  + \int_I\int_\omega |\partial^{2}_{t}\Dely\eta|^2 \dy\dt
 \\
&\quad  + \overline{\kappa}_1\int_I\int_\Omega |\partial^{2}_{t} \overline{\bv}|^2 \dx\dt  +  c(\overline{\kappa}_1)\int_I 
\Vert \partial_{t}\mathbf{h}\Vert_{L^{2}(\Omega)}^2\dt +   \overline{\kappa}_2\int_I\int_\omega |\partial^{3}_{t}\eta|^{2} \dy\dt  
\\
&\quad   + c(\overline{\kappa}_2)\int_I \Vert \partial_{t}\mathbf{H}\Vert_{L^{2}(\partial\Omega)}^2\dt  + \overline{\kappa}_3 \int_I\int_\Omega |\partial^{2}_{t} \overline{\bv}|^2 \dx\dt + c(\overline{\kappa}_3)\int_I \Vert \partial_{t}\nabla\mathbf{H}\Vert_{L^{2}(\Omega)}^2\dt
\\
& \quad +  c(\overline{\kappa}_4) \Vert \partial_t\overline{\varrho}^\gamma \Vert_{L^{2}\left( I; W^{1,2}(\Omega)  \right)}^2  +  \overline{\kappa}_4 \int_I\int_\omega |\partial^{3}_{t}\eta|^{2} \dy\dt  + \overline{\kappa}_4 \int_I\int_\Omega |\partial^{2}_{t} \overline{\bv}|^2 \dx\dt 
\end{aligned}
\end{equation}
for arbitrary  $\overline{\kappa}_i > 0, \; i \in \{0,\ldots, 4\}.  $ \\ 

Therefore by choosing $\overline{\kappa}_i, \; i \in \{0,\ldots, 4\} $ small enough and expanding terms involving the data, \eqref{eq:priorFirstEstimate}  yields
\begin{equation}\label{eq:FirstEstimate}
\begin{aligned}
&\int_I\int_\omega |\partial^{3}_{t}\eta|^{2} \dy\dt  + \sup_I \int_\omega  | \partial^{2}_{t} \naby\eta |^2 \dy  + \int_I\int_\Omega |\partial^{2}_{t} \overline{\bv}|^2 \dx\dt  + \left(\lambda + 2\mu\right)\sup_I\int_\Omega   |\partial_t\nabla\overline\bv|^2 \dx \\
&\lesssim   \Vert \eta_0\Vert_{W^{5,2}(\omega)}^2    + \Vert \eta_*\Vert_{W^{3,2}(\omega)}^2 + \Vert \overline{\varrho}_0\Vert_{W^{3,2}(\Omega)}^{2\gamma}   + \Vert \overline{\bv}_0\Vert_{W^{3,2}(\Omega)}^2    + \Vert \mathbf{h}(0)\Vert_{W^{1,2}(\Omega)}^2    + \Vert \mathbf{H}(0)\Vert_{W^{2,2}(\Omega)}^2 
 \\
&\quad  +  c(\kappa_0)\sup_I \int_\omega |\partial_t\naby\Dely\eta|^2\dy  + \int_I\int_\omega |\partial^{2}_{t}\Dely\eta|^2 \dy\dt 
\\
&\quad  + \int_I\left( \Vert \partial_{t}\mathbf{h}\Vert_{L^{2}(\Omega)}^2 +  \Vert \partial_{t}\mathbf{H}\Vert_{W^{1,2}(\Omega)}^2 \right)\dt + \Vert \partial_t\overline{\varrho}^\gamma \Vert_{L^{2}\left( I; W^{1,2}(\Omega)  \right)}^2 . 
\end{aligned}
\end{equation}


\medskip

\noindent \textit{Step 2:}\label{step:two} We now test \eqref{eq:LinShellEqTime}  with $-\partial_t\Dely\widehat{\eta}, $  and obtain 
\begin{equation*}
\begin{aligned}
&\dfrac{1}{2}\int_I\int_\omega \partial_t\left(|\partial^{2}_{t}\naby\eta|^{2}\right)  \dy\dt + \int_I\int_\omega |\partial^{2}_{t} \Dely\eta|^2 \dy\dt +  \dfrac{1}{2}\int_I\int_\omega \partial_{t}\left( | \partial_{t}\naby\Dely\eta|^{2} \right) \dy\dt \\
&= - \int_I\int_\omega \bn^\intercal  \left[\partial_t \mathbf{H} -\mu \mathbf{A}_{\eta_0}\partial_t \nabla\overline{\bv} - \dfrac{ \lambda + \mu }{J_{\eta_0}} \left(\mathbf{B}_{\eta_0}\colon \partial_t \nabla\overline{\bv}\right)\mathbf{B}_{\eta_0}  + a\mathbf{B}_{\eta_0}\partial_t\overline{\varrho}^\gamma  \right]\circ\bm{\varphi} \bn \cdot \partial^{2}_{t}\Dely\eta \dy\dt, 
\end{aligned}
\end{equation*}\\
whence 

\begin{equation}\label{eq:weakformTimeLaplaceEta}
\begin{aligned}
&\dfrac{1}{2}\int_I\int_\omega \partial_t\left(|\partial^{2}_{t}\naby\eta|^{2}\right)  \dy\dt + \int_I\int_\omega |\partial^{2}_{t} \Dely\eta|^2 \dy\dt +  \dfrac{1}{2}\int_I\int_\omega \partial_{t}\left( | \partial_{t}\naby\Dely\eta|^{2} \right) \dy\dt 
\\
&\lesssim \int_I\left( \Vert \partial_t\mathbf{H}\Vert_{W^{1/2, 2}(\partial\Omega)} +  \left( \lambda + 2\mu \right)\Vert \partial_t\naby\overline{\bv}\Vert_{W^{1/2,2}(\partial\Omega)}  \right)\Vert \partial^{2}_{t}\Dely\eta \Vert_{W^{-1/2,2}(\omega)} \dt 
\\
& \quad + \Vert \partial_t\overline{\varrho}^\gamma \Vert_{L^{2}\left( I; W^{1/2,2}(\partial\Omega)  \right)}  \left( \int_I \Vert \partial_t^2 \Dely\eta \Vert_{L^{2}(\omega)}^{2}   \right)^{1/2}
\\
&\lesssim \kappa \int_I \left( \Vert \partial_t\mathbf{H}\Vert^{2}_{W^{1,2}(\Omega)} + \left( \lambda + 2\mu \right)\Vert \partial_t\naby\overline{\bv}\Vert^{2}_{W^{1,2}(\Omega)} + \Vert \partial^{2}_{t}\Dely\eta \Vert_{L^{2}(\omega)}^2 \right) \dt 
\\
&\quad + c(\kappa)\Vert \partial_t\overline{\varrho}^\gamma \Vert_{L^{2}\left( I; W^{1,2}(\Omega)  \right)}^2 +   c(\kappa) \int_I\int_\omega |\partial^{2}_{t}\naby\eta|^{2} \dy\dt.
\end{aligned}
\end{equation}\\
For suitable choice of $\kappa > 0 $ and applying Gr\"onwall's inequality to the last term, \eqref{eq:weakformTimeLaplaceEta}  yields the following estimate

\begin{equation}\label{eq:SecondEstimate}
\begin{aligned}
&\sup_I \int_\omega \left( |\partial^{2}_{t}\naby\eta|^{2} +  |\partial_t \naby\Dely\eta|^{2}  \right)\dy   +  \int_I\int_\omega |\partial^{2}_{t} \Dely\eta|^2 \dy\dt 
\\
& \lesssim  c(\kappa)\left( \Vert \eta_0\Vert_{W^{5,2}(\omega)}^2    + \Vert \eta_*\Vert_{W^{3,2}(\omega)}^2 +  \Vert \overline{\varrho}_0\Vert_{W^{3,2}(\Omega)}^{2\gamma}    + \Vert \overline{\bv}_0\Vert_{W^{3,2}(\Omega)}^2  + \Vert \mathbf{h}(0)\Vert_{W^{1,2}(\Omega)}^2  + \Vert \mathbf{H}(0)\Vert_{W^{2,2}(\Omega)}^2  \right)
\\
&\quad + \kappa\int_I \left( \Vert \partial_{t}\mathbf{H}\Vert_{W^{1,2}(\Omega)}^2   + \left( \lambda + 2\mu \right) \Vert \partial_t\nabla^{2}\overline{\bv}\Vert^{2}_{L^{2}(\Omega)}\right) \dt
 + c(\kappa) \Vert \partial_t\overline{\varrho}^\gamma \Vert_{L^{2}\left( I; W^{1,2}(\Omega)  \right)}^2 .
\end{aligned}
\end{equation}


\medskip

\noindent \textit{Step 3: } \label{step:three} Consider the initial/boundary-value problem 
 \begin{equation}\label{eq:IBVP}
\left\{\begin{aligned}
&J_{\eta_0}\overline{\varrho}_0\,\partial_t\widehat{\bv} - \Div\left[\mu \mathbf{A}_{\eta_0}\nabla\widehat{\bv} + \dfrac{ \lambda + \mu }{J_{\eta_0}} \left(\mathbf{B}_{\eta_0}\colon \nabla\widehat{\bv}\right)\mathbf{B}_{\eta_0} -  a\mathbf{B}_{\eta_0}\partial_t\overline{\varrho}^\gamma  \right]  
	= \partial_t\mathbf{h} - \Div{\partial_t\mathbf{H}} \phantom{ + \widehat\eta =   }
&&   \text{ in } I\times\Omega,
\\
& \widehat{\bv}\circ\bm{\varphi}  = \left( \partial_t\widehat{\eta}\right)\bn  &&   \text{ in }  I\times\omega,
\\
& \widehat{\bv}(0, \cdot) = \widehat{\bv}_{0}(\cdot)
&&  \text{ in } \Omega,
\end{aligned}\right.
\end{equation}\\ 
 with
\[ \widehat \eta = \partial_t \eta , \quad \widehat \bv = \partial_t \overline{\bv}. \]
By setting $\underline{\bv} := \widehat{\bv}\circ\bfPsi_{\eta_0}^{-1},  $ \eqref{eq:IBVP} becomes
\begin{equation}\label{eq:IBVPnew}
\left\{\begin{aligned}
&\partial_t\underline{\bv} - \mu\varrho_0^{-1}\Delta\underline{\bv} - (\lambda + \mu)\varrho_0^{-1}\nabla\Div\underline{\bv}
	=  \left( \varrho_0 J_{\eta_0}\right)^{-1}\left(  \partial_t\mathbf{h} - \Div{ \partial_t\mathbf{H}} \right)\circ \bfPsi_{\eta_0}^{-1} - a\varrho_0^{-1}\partial_t (\nabx\varrho^{\gamma} )
& \text{ in }&  I\times\Omega_{\eta_0},
\\
& \underline{\bv}\circ{\bm{\varphi}}_{\eta_0}  = \left( \partial_t\widehat\eta\right)\bn  & \text{ in }& I\times\omega,
\\
&\underline{\bv}(0, \cdot) =\underline{\bv}_{0}
&  \text{ in } & \Omega_{\eta_0}.
\end{aligned}\right.
\end{equation}\\
Then by maximal regularity for vector-valued parabolic initial-boundary value problems \phantom{value problems} (see \cite[Section 2, Theorem 2.1 \& Theorem 2.2 ]{denk2007optimal})
\begin{equation*}
\begin{aligned}
	\int_I\int_{\Omega_{\eta_0}}\big(\vert
 \partial_t \underline{\bv}\vert^2+ \left(\lambda + 2\mu \right)\vert \nabx^2\underline{\bv}
 \vert^2 \big)\dx\dt
 &\lesssim
 \int_I\Vert \partial_t \widehat{\eta}\Vert_{W^{3/2,2}(\omega)}^2\dt 
 +
 \Vert \partial_t \widehat{\eta}\Vert_{W^{3/4,2}(I;L^2(\omega))}^2 + \int_{\Omega_{\eta_0}}|\nabla\underline \bv_0|^2\dx
 \\
 &\quad+
 \int_I\int_{\Omega_{\eta_0}}
 \big(\vert
 \partial_t\mathbf{h}\circ \bm{\Psi}_{\eta_0}^{-1}
\vert^2
+
\vert(\divx \partial_t \mathbf{H} )\circ \bm{\Psi}_{\eta_0}^{-1}\vert^2    + \vert \partial_t(\nabx\varrho^{\gamma})\vert^2 \big)\dx\dt .
\end{aligned}
\end{equation*}

Recasting the integration back to the reference configuration $\Omega, $ and interpolating the regularity of the shell, we obtain for any $\kappa > 0 \; \text{and}\; \overline\kappa > 0, $ 
\begin{equation*}
\begin{aligned}
	&	\int_I\int_{\Omega}\big(\vert
 \partial_{t} \widehat{\bv}\vert^2+ \left(\lambda + 2\mu \right) \vert \nabx^2\widehat{\bv}
 \vert^2 \big)\dx\dt\\
  & \lesssim
 \kappa
 \int_I\Vert \partial_t\Dely\widehat{ \eta}\Vert_{L^{2}(\omega) }^2\dt
 +
 c(\kappa)
 \sup_I\Vert \partial_t\naby\widehat{ \eta}\Vert_{L^{2}(\omega) }^2
 +
 \overline\kappa \int_I\Vert \partial_t^2\widehat{ \eta}\Vert_{L^{2}(\omega) }^2\dt
 \\
 &\quad + c(\overline\kappa)
 \int_I\int_{\Omega}
 \big( 
 \vert
\partial_t \mathbf{h}
\vert^2
+ 
\vert    \nabx\partial_t\mathbf{H}  \vert^2   + \vert \partial_t(\nabx\overline{\varrho}^{\gamma})\vert^2  \big)\dx\dt 
+ c(\overline\kappa)
\int_{\Omega }|\nabla\widehat{\bv}_0|^2\dx.
\end{aligned}
\end{equation*}
Hence, using the regularity of $\overline{\varrho}^{\gamma} $  (see~\eqref{eq:rhogammaregularity}),  and expanding $ \nabla\widehat{\bv}_0, $we obtain the following estimate:  
\begin{equation}
\begin{aligned}\label{eq:ThirdEstimate}
	&	\int_I\int_{\Omega}\left(\vert
 \partial_{t}^{2} \overline{\bv}\vert^2+  \left(\lambda + 2\mu \right) \vert \partial_t\nabx^2\overline{\bv}
 \vert^2 \right)\dx\dt\\
  & \lesssim 
 c(\overline\kappa)\left(  \Vert \overline{\varrho}_0\Vert_{W^{3,2}(\Omega)}^{2\gamma} + \Vert \overline{\bv}_0\Vert_{W^{3,2}(\Omega)}^2    +  \Vert \mathbf{h}(0)\Vert_{W^{1,2}(\Omega)}^2   
 + \Vert \mathbf{H}(0)\Vert_{W^{2,2}(\Omega)}^2 \right) + 
 \kappa\int_I\int_\omega |\partial_{t}^{2}\Dely \eta|^2 \dy\dt
 \\
 &\quad +
 c(\kappa)\sup_I\int_\omega |\partial_{t}^{2}\naby\eta|^2\dy  +
\overline\kappa \int_I\int_\omega | \partial_t^3\eta|^2\dy\dt
\\
&\quad  + 
 c(\overline\kappa)\int_I
 \big( 
 \Vert
\partial_t \mathbf{h}
\Vert_{L^{2}(\Omega)}^{2}
+
\Vert  \partial_t\mathbf{H}  \Vert_{W^{1,2}(\Omega)}^{2}\big)\dt   +  c(\overline{\kappa}) \Vert \partial_t\overline{\varrho}^\gamma \Vert_{L^{2}\left( I; W^{1,2}(\Omega)  \right)}^2  .
\end{aligned}
\end{equation}


\medskip

\noindent \textit{Step 4:}\label{step:four}  Testing \eqref{eq:LinShellEqTime} with $ \Dely^{2}\widehat\eta $,   we obtain:

\begin{equation} \label{eq:FourthEstimate}
\begin{aligned} 
& \int_I \int_\omega |\partial_t \Dely^{2}\eta|^2 \dy\dt 
 \\
&\lesssim 
  \int_I\int_\omega \left( |\partial^{3}_{t}\eta|^{2} \dy\dt +  |\partial^{2}_{t}\Dely\eta|^{2} \right) \dy\dt 
  +
  \int_I  \Big( \Vert\partial_t\mathbf{H}\Vert_{L^{2}(\partial \Omega)}^2  + (\lambda + 2\mu)\Vert \partial_{t}\nabx\overline{\bv}\Vert_{L^{ 2}(\partial \Omega)}^2  \Big) \dt 
  \\
  &\quad + \Vert \partial_t\overline{\varrho}^\gamma \Vert_{L^{2}\left( I; L^{2}(\partial\Omega)  \right)}^2
   \\
&\lesssim 
   \int_I\int_\omega \left( |\partial^{3}_{t}\eta|^{2} \dy\dt +  |\partial^{2}_{t}\Dely\eta|^{2} \right) \dy\dt  +   \Vert \partial_t\overline{\varrho}^\gamma \Vert_{L^{2}\left( I; W^{1,2}(\Omega)  \right)}^2
  \\
  & \quad +  \int_I  \Big( \Vert\partial_t\mathbf{H}\Vert_{W^{1,2}(  \Omega)}^2  + \Vert \partial_{t}\nabx\overline{\bv}\Vert_{W^{1,2}(  \Omega)}^2  \Big) \dt. 
\end{aligned}
\end{equation}
A bound for the  last term follows directly form the maximal regularity estimate established in \hyperref[step:three]{\textit{Step 3},} allowing for the inclusion of the desired terms in the final estimate.\\


\medskip

\noindent\textit{Step 5: }  Consider the original linearized momentum equation \eqref{linMomAlone} with time  treated as `\textbf{frozen}'. Similar to \hyperref[step:three]{\textit{Step 3} }\!\!,  the change of variable  $ \underline{\underline{\bv}} := \overline{\bv}\circ\bfPsi_{\eta_0}^{-1} $ yields the following elliptic equation
\begin{equation}\label{eq:IBVPnew}
\left\{\begin{aligned}
& \mu\Delta\underline{\underline{\bv}} + (\lambda + \mu)\nabla\Div\underline{\underline{\bv}}
	=  \varrho_0\partial_t\underline{\underline{\bv}} + a\nabx\varrho^{\gamma} - J_{\eta_0}^{-1}\left(  \mathbf{h} - \Div{ \mathbf{H}} \right)\circ \bfPsi_{\eta_0}^{-1}
& \text{ in }&  I\times\Omega_{\eta_0},
\\
& \underline{\underline{\bv}}\circ{\bm{\varphi}}_{\eta_0}  = \left( \partial_t \eta \right)\bn  & \text{ in }& I\times\omega.
\end{aligned}\right.
\end{equation}
Then by maximal regularity property
\begin{equation*}
\begin{aligned}
(\lambda + 2\mu)\int_I \Vert  \underline{\underline{\bv}}
 \Vert_{W^{4,2}(\Omega_{\eta_0})}^{2}\dt
  & \lesssim
  \int_I \Vert \partial_t \underline{\underline{\bv}}
 \Vert_{W^{2,2}(\Omega_{\eta_0})}^{2}\dt
 +
 \int_I\Vert \partial_t \eta\Vert^{2}_{W^{7/2, 2}(\omega)}\dt  
 \\
 &\quad+
 \int_I 
 \left( \Vert
 \mathbf{h}\circ \bm{\Psi}_{\eta_0}^{-1}
 \Vert_{W^{2,2}(\Omega_{\eta_0})}^2
+
\Vert(\divx  \mathbf{H} )\circ \bm{\Psi}_{\eta_0}^{-1} \Vert_{W^{2,2}(\Omega_{\eta_0})}^2 \right)\dt 
\\
&\quad + \int_I \Vert \nabx\varrho^{\gamma}\Vert_{W^{2,2}(\Omega_{\eta_0})}^2 \dt.
\end{aligned}
\end{equation*}
Of note, $\varrho^{\gamma}(t) \in W^{3,2}(\Omega_{\eta_0})$ (see Proposition ~\ref{prop:mainFP}).  Therefore, rewriting the integration in the reference configuration $\Omega$ and using the interpolated shell regularity, we obtain 
\begin{equation}\label{FifthEstimate}
\begin{aligned}
(\lambda + 2\mu)\int_I \Vert  \overline{\bv}
 \Vert_{W^{4,2}(\Omega)}^{2}\dt
  & \lesssim
  \int_I \int_\Omega |\partial_t \nabla^{2}\overline{\bv}
 |^{2}\dx\dt
 +
 \int_I\int_\omega | \partial_t\Dely^{2}\eta|^{2} \dy\dt   + \sup_I \int_\omega |\partial_t\naby\Dely\eta|^2 \dy
 \\
 &\quad+
 \int_I 
 \left( \Vert
 \mathbf{h}
 \Vert_{W^{2,2}(\Omega)}^2
+
\Vert \mathbf{H} \Vert_{W^{3,2}(\Omega)}^2 \right)\dt   + \Vert \overline{\varrho}^\gamma \Vert_{L^{2}\left(I; W^{3,2}(\Omega)  \right)}^{2}  .
\end{aligned}
\end{equation}


\medskip

\noindent \textit{Step 6:} Similar to \hyperref[step:four]{\textit{Step 4} }\!\!, we apply  the Laplace operator to \eqref{linShelAlone} and test the resulting equation with $\Dely^{3}\eta .$ This yields the following estimate
\begin{equation} \label{SixthEstimate}
\begin{aligned} 
& \int_I \int_\omega |\Dely^{3}\eta|^2 \dy\dt 
 \\
&\lesssim 
  \int_I\int_\omega \left( |\partial_{t}\Dely^2\eta|^{2} + |\partial^{2}_{t}\Dely\eta|^{2} \right) \dy\dt  
  +
  \int_I  \left( \Vert\mathbf{H}\Vert_{W^{3,2}(  \Omega)}^2 + (\lambda + 2\mu)\Vert \overline{\bv}\Vert_{W^{4,2}(  \Omega)}^2  \right) \dt 
  \\
  &\quad + \Vert \overline{\varrho}^\gamma \Vert_{L^{2}\left(I; W^{3,2}(\Omega)  \right)}^{2}.
\end{aligned}
\end{equation}

\end{proof}



\subsection{Banach fixed-point argument}\label{subsec:BanachFixed-PointSub}

To establish the existence and uniqueness of the nonlinear system \eqref{momEqAloneBar}-\eqref{shellEqAloneBar}, we rely on the Banach fixed-point theorem. This approach is motivated  by the observation that any solution to   \eqref{momEqAloneBar}-\eqref{shellEqAloneBar} corresponds to a solution of the linearised system  \eqref{linMomAlone}-\eqref{linShelAlone}, with source terms  
\[ \left(\mathbf{h}, \mathbf{H}  \right) =  \left( \mathbf{h}_\eta, \mathbf{H}_\eta  \right).  \]
For this purpose, we introduce the notation
\[\mathbb{W}^{s, m}_{q,p} (I\times \Omega) :=  W^{s,q}\big( I; W^{m, p}(\Omega) \big) \quad \forall\, s, m \geq 0, \;\; \forall\,  q, p \in  (0, \infty].  \]
We further introduce the following spaces: 
 \begin{equation*}
\begin{aligned}
  \mathcal{X}_{\overline\bv} & :=    \mathbb{W}^{0, 4}_{2,2} (I\times \Omega)  \cap \mathbb{W}^{1, 2}_{2,2} (I\times \Omega)  \cap \mathbb{W}^{2, 0}_{2,2} (I\times \Omega) \cap \mathbb{W}^{1, 1}_{\infty,2} (I\times \Omega) , &
\\
  \mathcal{X}_{\eta} &:=  \mathbb{W}^{1, 4}_{2,2} (I\times \omega)  \cap \mathbb{W}^{1, 3}_{\infty, 2} (I\times \omega) \cap \mathbb{W}^{2, 1}_{\infty,2} (I\times \omega) \cap \mathbb{W}^{2, 2}_{2,2} (I\times \omega) \cap \mathbb{W}^{3, 0}_{2,2} (I\times \omega)  \cap \mathbb{W}^{0, 6}_{2,2} (I\times \omega), &
\end{aligned}
\end{equation*}\\
endowed respectively with the norms

\begin{equation*}
\begin{aligned}
  \Vert \overline\bv\Vert_{ \raisebox{-1.5ex}{$\mathcal{X}$}_{\overline\bv} } & :=  \sqrt{\lambda + 2\mu} \left( \int_I\int_\Omega |\Delta^{2}\overline\bv|^{2}\dx\dt  \right)^{1/2}   + \sqrt{\lambda + 2\mu} \left( \int_I\int_\Omega |\partial_t\nabla^{2}\overline\bv|^{2} \dx\dt \right)^{1/2} &
\\
&\;\;\quad + \left( \int_I\int_\Omega |\partial^{2}_t \overline\bv|^{2}\dx\dt  \right)^{1/2} +  \sqrt{\lambda + 2\mu}\sup_I \left( \int_\Omega |\partial_t\nabla\overline{\bv}|^{2} \dx\right)^{1/2},  & 
\\[0.25cm]
 \Vert \eta\Vert_{ \raisebox{-1.5ex}{$\mathcal{X}$}_{\eta} } &:=   \left( \int_I\int_\omega |\partial_t \Dely^{2}\eta|^{2} \dy\dt \right)^{1/2} +  \sup_I \left( \int_\omega |\partial_t \naby\Dely\eta|^{2} \dy \right)^{1/2} + \sup_I \left( \int_\omega |\partial^{2}_t \naby\eta|^{2} \dy \right)^{1/2}  &
\\
&\;\;\quad +   \left( \int_I\int_\omega |\partial^{2}_t \Dely\eta|^{2} \dy\dt \right)^{1/2} +  \left( \int_I\int_\omega |\partial^{3}_t \eta|^{2} \dy\dt \right)^{1/2}   +  \left( \int_I\int_\omega |\Dely^{3}\eta|^{2} \dy\dt \right)^{1/2}. &
\end{aligned}
\end{equation*}\\
Finally, we endow the product space  $\mathcal{X}_{\overline\bv} \times \mathcal{X}_{\eta} $ with the norm: 

\begin{equation*}
\hspace*{-7.5cm}\begin{aligned}
\Vert \left(\overline{\bv}, \eta   \right)\Vert_{ \raisebox{-1.5ex}{$ \mathcal{X}$}_{\overline\bv}  \raisebox{-1.5ex}{$\times \mathcal{X}$}_{\eta} } & :=    \Vert \overline\bv\Vert_{ \raisebox{-1.5ex}{$\mathcal{X}$}_{\overline\bv} }  + \Vert \eta\Vert_{ \raisebox{-1.5ex}{$\mathcal{X}$}_{\eta} }. 
\end{aligned}
\end{equation*}\\
Furthermore, we consider the subspace of initial data
\[\mathcal{I} := \Big\{ \left( \overline{\bv}_0, \eta_0, \eta_{*}  \right) \in W^{3,2}(\Omega)\times W^{5,2}(\omega)\times  W^{3,2}(\omega)\colon \overline{\bv}_{0}\circ\bm{\varphi} = \eta_{*}\bn \quad \text{on } \omega  \Big\}, \]
endowed with the norm
\[\Vert \left( \overline{\bv}_0, \eta_0, \eta_{*}  \right) \Vert_{\mathcal{I} } :=  \Vert  \overline{\bv}_0  \Vert_{W^{3,2}(\Omega)} +  \Vert  \eta_0  \Vert_{W^{5,2}(\omega)} + \Vert  \eta_*  \Vert_{W^{3,2}(\omega)}.  \]
We also consider  the space of source terms
\begin{align*}
\mathcal{S}    :=  \Big\{  &(\mathbf{h}, \mathbf{H}) \colon  \mathbf{h} \in \mathbb{W}^{0, 2}_{2,2} (I\times \Omega) \cap \mathbb{W}^{1, 0}_{2,2} (I\times \Omega),  \\
&  \mathbf{H} \in \mathbb{W}^{0, 3}_{2,2} (I\times \Omega) \cap\mathbb{W}^{1, 1}_{2,2} (I\times \Omega),\; \mathbf{h}(0) \in W^{1,2}(\Omega), \; \mathbf{H}(0) \in W^{2,2}(\Omega) \Big\}, 
\end{align*} 
with the norm

\begin{equation*}
\begin{aligned}
\Vert \left( \mathbf{h}, \mathbf{H} \right) \Vert_{\mathcal{S}} &:= \Vert \mathbf{h} \Vert_{\mathbb{W}^{0, 2}_{2,2} (I\times \Omega) \cap \mathbb{W}^{1, 0}_{2,2} (I\times \Omega) }  +  \Vert \mathbf{H} \Vert_{\mathbb{W}^{0, 3}_{2,2} (I\times \Omega) \cap\mathbb{W}^{1, 1}_{2,2} (I\times \Omega) } 
\\
&\;\;\quad + \Vert \mathbf{h}(0) \Vert_{W^{1,2}(\Omega)}  + \Vert \mathbf{H}(0) \Vert_{W^{2,2}(\Omega)}. 
\end{aligned}
\end{equation*}\\ 
For all $T > 0 $ and $R > 0, $ we introduce a ball-like  closed subset  $\mathcal{K}_{T, R} \subset \mathcal{S} $ -- fundamental for ensuring the well-posedness of the solution to the  nonlinear problem \eqref{momEqAloneBar}-\eqref{shellEqAloneBar} -- defined  by 
 
 \begin{align*}
\mathcal{K}_{T, R}  :=  \Bigg\{ & ( \mathbf{h}, \mathbf{H}) \in \mathcal{S}  \colon   \mathbf{h}(0) =  -J_{\eta_0}\overline{\varrho}_0\nabla\overline{\bv}_0\;\! \partial_t \bfPsi^{-1}_{\eta} \circ \bfPsi_{\eta}(0,\cdot) - \overline{\varrho}_0 \overline{\bv}_0 \big( \nabla\overline{\bv}_0\colon \mathbf{B}_{\eta_0} \big), 
\\
&  \mathbf{H}(0) =  0, \; \text{and} \;\;  \Vert \left( \mathbf{h}, \mathbf{H} \right) \Vert_{\mathcal{S}}  \leq R  \Bigg\}.  
\end{align*}

\begin{remark} 
The  closed set $\mathcal{K}_{T, R} $ is nonempty. Indeed, by the  uniform boundedness of $ \overline{\varrho}_0, \overline{\bv}_0, \eta_0 \; \text{and}\; \eta_*,  $  one may assume  -- after possibly  increasing the radius $R > 0 $ -- that 
 \[ \Vert \underline{\mathbf{h}}_0 \Vert_{W^{2,2}(\Omega)} \leq \dfrac R 2 ,  \]
 where  \[\underline{\mathbf{h}}_0 :=   -J_{\eta_0}\overline{\varrho}_0\nabla\overline{\bv}_0\;\! \partial_t \bfPsi^{-1}_{\eta} \circ \bfPsi_{\eta}(0,\cdot) - \overline{\varrho}_0 \overline{\bv}_0 \big( \nabla\overline{\bv}_0\colon \mathbf{B}_{\eta_0}\big). \] \\
Hence, for  any $T \leq 1, \;   \left( \underline{\mathbf{h}}_0, 0 \right) \in \mathcal{K}_{T, R}. $
\end{remark}  

\begin{remark}
For $ \left( \overline{\bv}_0, \eta_0, \eta_{*}  \right) \in \mathcal{I} $ and $ ( \mathbf{h}, \mathbf{H}) \in  \mathcal{K}_{T, R} $, the compatibility condition  \eqref{eq:CC} is satisfied. 
\end{remark}

 \begin{proof}[\textit{Proof of Theorem \ref{thm:transformedSystem1}}] 
 
Let $ c > 0 $ be arbitrary but fixed, and let  $R > 0 $ be sufficiently large such that 

\[ c\bigg( \Vert \left( \overline{\bv}_0, \eta_0, \eta_{*}  \right) \Vert_{\mathcal{I} }^{2} +  \sup_{t\in I}\Big( \Vert \overline{\varrho}^{\gamma}(t)\Vert_{W^{3,2}(\Omega )}^2
+
\Vert \partial_t\overline{\varrho}^{\gamma}(t)\Vert_{W^{2,2}(\Omega )}^2
\Big) \bigg) \leq \dfrac R 8.  \] \\
Consider the map 
\[ \mathcal{F} \colon \mathcal{K}_{T, R} \to  \mathcal{K}_{T, R},  \quad \left( \mathbf{h}, \mathbf{H} \right) \mapsto \Big( \mathbf{h}_\eta (\overline{\bv}), \, \mathbf{H}_\eta ( \overline{\bv} ) \Big),   \] 
where the couple $\left( \overline{\bv}, \eta  \right) \; \text{appearing in }  \left( \mathbf{h}_\eta, \mathbf{H}_\eta \ \right),   $  is the solution of the linearised problem 
 \eqref{linMomAlone}-\eqref{linShelAlone},  with source term $\left( \mathbf{h}, \mathbf{H} \right)  $  and initial condition $ \left( \overline{\bv}_0, \eta_0, \eta_{*}  \right). $ 
 
 We aim to prove that the mapping $\mathcal{F} $ is a strict contraction for a suitably small time $ T > 0. $ To streamline the argument, we present the proof in two steps.\\

 \medskip
 
 \noindent\textbf{Step 1: }  $\mathcal{F}-$invariance of the set $\mathcal{K}_{T, R}. $ \\  
 First, observe that the estimate in Proposition ~\ref{thm:transformedSystem}  is of the form 
 
 \begin{align}  
 \Vert \left( \overline{\bv}, \eta   \right)\Vert_{ \raisebox{-1.5ex}{$ \mathcal{X}$}_{\overline\bv}  \raisebox{-1.5ex}{$\times \mathcal{X}$}_{\eta} }^2  & \leq Ce^{CT} \bigg( \Vert \left(  \overline{\bv}_0, \eta_0, \eta_{*}  \right) \Vert_{\mathcal{I} }^{2}  + T\sup_{t\in I}\left( \Vert \overline{\varrho}^{\gamma}(t)\Vert_{W^{3,2}(\Omega )}^2 + \Vert \partial_t\overline{\varrho}^{\gamma}(t)\Vert_{W^{2,2}(\Omega )}^2 \right)   \nonumber 
 \\
 & \quad +  \Vert \left( \mathbf{h}, \mathbf{H} \right) \Vert_{\mathcal{S}}^{2}  \bigg),  \label{eq:StepOneInvarianceEstim}
 \end{align}
 -- a consequence of the application of Gr\"onwall's lemma -- 
where the constant $C > 0$ depends affinely on $\mu $ and $\lambda. $ 

Without loss of generality, we assume that $T \leq 1, $ and that the constant $C, $ could depend on $R. $ Then, from the estimate \eqref{eq:StepOneInvarianceEstim}, we deduce that 
\begin{equation}\label{eq:UniformBound}
 \Vert \left( \overline{\bv}, \eta   \right)\Vert_{  \raisebox{-1.5ex}{$\mathcal{X}$}_{\overline\bv}  \raisebox{-1.5ex}{$\times \mathcal{X}$}_{\eta} } \leq C.
\end{equation}\\
However, from the continuity of the determinant function, we have that 

\begin{align}
\Vert J_{\eta_0} -J_{\eta} \Vert_{\mathbb{W}^{0, 3}_{\infty, 2} (I\times \Omega)  }
  & \leq 
  \widetilde{C}\Vert \nabla\bfPsi_{\eta_0} -  \nabla\bfPsi_\eta   \Vert_{\mathbb{W}^{0, 3}_{\infty, 2} (I\times \Omega)  }  \nonumber
 \\
 &\leq
   \widetilde{C}\Vert \bfPsi_{\eta_0} -  \bfPsi_\eta   \Vert_{\mathbb{W}^{0, 4}_{\infty, 2} (I\times \Omega)  }  \nonumber
  \\
  & \leq  \overline{C}\Vert \eta_0 -  \eta   \Vert_{\mathbb{W}^{0, 4}_{\infty, 2} (I\times \omega) }  \label{ineq:HanzawaT}
  \\
  & \leq \overline{C} T^{1/2} \Vert \partial_t \eta   \Vert_{\mathbb{W}^{0, 4}_{2, 2} (I\times \omega)  }   \nonumber
  \\
  & \leq CT^{1/2},         \nonumber
\end{align}
where \eqref{ineq:HanzawaT} follows from the Hanzawa transform (cf. ~ Proposition \ref{prop:estimatePsiEta}).  \\
Moreover, 

\begin{align*}
\Vert \mathbf{B}_{\eta_0} - \mathbf{B}_{\eta} \Vert_{\mathbb{W}^{0, 3}_{\infty, 2} (I\times \Omega)  }
  & \leq 
 \Vert J_{\eta_0} \Vert_{ W^{3,2}(\Omega) } \left\Vert \Big(\nabla\bfPsi_{\eta_0}^{-1}\circ \bfPsi_{\eta_0}\Big)^{\intercal} -  \Big(\nabla\bfPsi_{\eta}^{-1}\circ\bfPsi_{\eta}\Big)^{\intercal}   \right\Vert_{\mathbb{W}^{0, 3}_{\infty, 2} (I\times \Omega)  }  
 \\
 &\quad +   \Vert J_{\eta_0} - J_\eta \Vert_{\mathbb{W}^{0, 3}_{\infty, 2} (I\times \Omega) } \left\Vert \Big( \nabla\bfPsi_{\eta}^{-1}\circ\bfPsi_{\eta}\Big)^{\intercal} \right\Vert_{\mathbb{W}^{0, 3}_{\infty, 2} (I\times \Omega) }
  \\
   & \leq \overline{C}_1 \Vert \eta_0 - \eta  \Vert_{\mathbb{W}^{0, 4}_{\infty, 2} (I\times \omega)  } + \overline{C}_2  \Vert J_{\eta_0} - J_\eta \Vert_{\mathbb{W}^{0, 3}_{\infty, 2} (I\times \Omega) } 
  \\
   & \leq \overline{C}_1 T^{1/2} \Vert \partial_t \eta  \Vert_{\mathbb{W}^{0, 4}_{2, 2} (I\times \omega)  } + \overline{C}_3  T^{1/2}
  \\
  & \leq CT^{1/2},        
\end{align*}  

i.e., 
\begin{equation}
\Vert \mathbf{B}_{\eta_0} - \mathbf{B}_{\eta} \Vert_{\mathbb{W}^{0, 3}_{\infty, 2} (I\times \Omega)  } \leq CT^{1/2}.
\end{equation}\\
It also holds that 
\begin{equation}
 \Vert \partial_t \bfPsi_{\eta}^{-1}\circ \bfPsi_\eta \Vert_{\mathbb{W}^{0, 2}_{\infty, 2} (I\times \Omega) }  \leq C.
\end{equation}\\
Moreover, due to the properties of $\bfPsi_{\eta} $ and $\bfPsi_{\eta}^{-1}, $  together with \eqref{eq:UniformBound}, $\mathbf{B}_\eta $ and $J_\eta$  satisfy 
\begin{equation}
\Vert \mathbf{B}_{\eta} \Vert_{\mathbb{W}^{0, 3}_{\infty, 2} (I\times \Omega)  } \leq C, \quad  \Vert J_{\eta} \Vert_{\mathbb{W}^{0, 3}_{\infty, 2} (I\times \Omega)  } \leq C. 
\end{equation}\\
Therefore, we deduce that  

\begin{align*}
\Vert \mathbf{h}_\eta \Vert_{\mathbb{W}^{0, 2}_{2, 2} (I\times \Omega)  }  &\leq   \Vert J_{\eta_0} - J_{\eta} \Vert_{\mathbb{W}^{0, 2}_{\infty, 2} (I\times \Omega) }  \Vert  \overline{\varrho}_0 \Vert_{W^{2,2}(\Omega) }  \Vert \partial_t\overline{\bv} \Vert_{\mathbb{W}^{0, 2}_{2, 2} (I\times \Omega) }
\\
& \quad + \Vert J_{\eta} \Vert_{\mathbb{W}^{0, 2}_{\infty, 2} (I\times \Omega) } \Vert \overline{\varrho}_0 - \overline{\varrho} \Vert_{\mathbb{W}^{0, 2}_{\infty, 2} (I\times \Omega) } \Vert \partial_t\overline{\bv} \Vert_{\mathbb{W}^{0, 2}_{2, 2} (I\times \Omega) }
\\
& \quad +  \Vert  J_{\eta} \Vert_{\mathbb{W}^{0, 2}_{\infty, 2} (I\times \Omega) } \Vert  \overline{\varrho} \Vert_{\mathbb{W}^{0, 2}_{\infty, 2} (I\times \Omega) }  \Vert \overline{\bv} \Vert_{\mathbb{W}^{0, 4}_{2, 2} (I\times \Omega) }  \Vert \partial_t \bfPsi_{\eta}^{-1}\circ \bfPsi_\eta \Vert_{\mathbb{W}^{0, 2}_{\infty, 2} (I\times \Omega) }
\\
& \quad +  \Vert  \overline{\varrho} \Vert_{\mathbb{W}^{0, 2}_{\infty, 2} (I\times \Omega) }  \Vert \overline{\bv} - \overline{\bv}_0 \Vert_{\mathbb{W}^{0, 2}_{\infty, 2} (I\times \Omega) }  \Vert \overline{\bv} \Vert_{\mathbb{W}^{0, 4}_{2, 2} (I\times \Omega) }  \Vert \mathbf{B}_{\eta}\Vert_{\mathbb{W}^{0, 2}_{\infty, 2} (I\times \Omega) } 
\\
& \quad +  \Vert  \overline{\varrho} \Vert_{\mathbb{W}^{0, 2}_{\infty, 2} (I\times \Omega) }  \Vert  \overline{\bv}_0 \Vert_{W^{2,2}(\Omega) }  \Vert \overline{\bv} \Vert_{\mathbb{W}^{0, 4}_{2, 2} (I\times \Omega) }  \Vert \mathbf{B}_{\eta}\Vert_{\mathbb{W}^{0, 2}_{\infty, 2} (I\times \Omega) }
\\
& \leq  CT^{1/2} +  c\bigg( \Vert \left( \overline{\bv}_0, \eta_0, \eta_{*}  \right) \Vert_{\mathcal{I} }^{2} + \sup_{t\in I}\left( \Vert \overline{\varrho}^{\gamma}(t)\Vert_{W^{3,2}(\Omega )}^2 + \Vert \partial_t\overline{\varrho}^{\gamma}(t)\Vert_{W^{2,2}(\Omega )}^2 \right)  \bigg),
\end{align*}
whence

\begin{equation}
\Vert \mathbf{h}_\eta \Vert_{\mathbb{W}^{0, 2}_{2, 2} (I\times \Omega)  } \leq CT^{1/2} + \dfrac{R}{8}.
\end{equation}\\
In addition, we estimate $ \partial_t\mathbf{h}_\eta, $ yielding 
\begin{align*}
\Vert \partial_t\mathbf{h}_\eta \Vert_{\mathbb{W}^{0, 0}_{2, 2} (I\times \Omega)  }  &\leq  \Vert \partial_t \big( J_\eta \overline{\varrho}\big) \Vert_{\mathbb{W}^{0, 0}_{\infty, \infty} (I\times \Omega) } \Vert \partial_{t}\overline{\bv} \Vert_{\mathbb{W}^{0, 0}_{2, 2} (I\times \Omega) }
\\
& \quad + \Vert J_{\eta_0} - J_{\eta} \Vert_{\mathbb{W}^{0, 0}_{\infty, \infty} (I\times \Omega)  } \Vert \overline{\varrho}_0 \Vert_{L^{\infty}(\Omega) } \Vert \partial_{t}^{2}\overline{\bv} \Vert_{\mathbb{W}^{0, 0}_{2, 2} (I\times \Omega) }
\\
& \quad +  \Vert J_{\eta} \Vert_{\mathbb{W}^{0, 0}_{\infty, \infty} (I\times \Omega)  }\Vert \overline{\varrho}_0 - \overline{\varrho} \Vert_{\mathbb{W}^{0, 0}_{\infty, \infty} (I\times \Omega)  } \Vert \partial_{t}^{2}\overline{\bv} \Vert_{\mathbb{W}^{0, 0}_{2, 2} (I\times \Omega) }
\\
& \quad + \Vert \partial_t \big( J_\eta \overline{\varrho}\big) \Vert_{\mathbb{W}^{0, 0}_{\infty, \infty} (I\times \Omega) }  \Vert \nabla\overline{\bv} \Vert_{\mathbb{W}^{0, 0}_{2, 2} (I\times \Omega) } \left\Vert \partial_t \bfPsi_{\eta}^{-1}\circ \bfPsi_\eta  \right\Vert_{\mathbb{W}^{0, 0}_{\infty, \infty} (I\times \Omega) }
\\
& \quad + \Vert  J_\eta \overline{\varrho} \Vert_{\mathbb{W}^{0, 0}_{\infty, \infty} (I\times \Omega)}  \Vert \partial_t \nabla\overline{\bv} \Vert_{\mathbb{W}^{0, 0}_{2,2} (I\times \Omega) } \left\Vert \partial_t \bfPsi_{\eta}^{-1}\circ \bfPsi_\eta  \right\Vert_{\mathbb{W}^{0, 0}_{\infty, \infty} (I\times \Omega) }
\\
& \quad + \Vert  J_\eta \overline{\varrho} \Vert_{\mathbb{W}^{0, 0}_{\infty, \infty} (I\times \Omega)}  \Vert  \nabla\overline{\bv} \Vert_{\mathbb{W}^{0, 0}_{2, 2} (I\times \Omega) } \left\Vert \partial_t \Big(\partial_t \bfPsi_{\eta}^{-1}\circ \bfPsi_\eta \Big) \right\Vert_{\mathbb{W}^{0, 0}_{\infty, \infty} (I\times \Omega) }
\\
& \quad + \Vert \partial_t \overline{\varrho} \Vert_{\mathbb{W}^{0, 0}_{\infty, \infty} (I\times \Omega)} \Vert \overline{\bv} - \overline{\bv}_0\Vert_{\mathbb{W}^{0, 0}_{\infty, \infty} (I\times \Omega) }  \Vert \nabla\overline{\bv} \Vert_{\mathbb{W}^{0, 0}_{2, 2} (I\times \Omega) }  \left\Vert \mathbf{B}_\eta \right \Vert_{\mathbb{W}^{0, 0}_{\infty, \infty} (I\times \Omega) }  
\\
& \quad + \Vert \partial_t \overline{\varrho} \Vert_{\mathbb{W}^{0, 0}_{\infty, \infty} (I\times \Omega) } \Vert  \overline{\bv}_0\Vert_{L^{\infty}(\Omega) }  \Vert \nabla\overline{\bv} \Vert_{\mathbb{W}^{0, 0}_{2, 2} (I\times \Omega) }  \left\Vert \mathbf{B}_\eta \right \Vert_{\mathbb{W}^{0, 0}_{\infty, \infty} (I\times \Omega) } 
\\
& \quad + \Vert  \overline{\varrho} \Vert_{\mathbb{W}^{0, 0}_{\infty, \infty} (I\times \Omega) } \Vert \overline{\bv} - \overline{\bv}_0\Vert_{\mathbb{W}^{0, 0}_{\infty, \infty} (I\times \Omega) }  \Vert \partial_t \nabla\overline{\bv} \Vert_{\mathbb{W}^{0, 0}_{2, 2} (I\times \Omega) }  \left\Vert \mathbf{B}_{\eta}\right \Vert_{\mathbb{W}^{0, 0}_{\infty, \infty} (I\times \Omega) }
\\
& \quad + \Vert  \overline{\varrho} \Vert_{\mathbb{W}^{0, 0}_{\infty, \infty} (I\times \Omega) } \Vert \overline{\bv}_0\Vert_{L^{\infty}(\Omega)  }  \Vert \partial_t \nabla\overline{\bv} \Vert_{\mathbb{W}^{0, 0}_{2, 2} (I\times \Omega) }  \left\Vert \mathbf{B}_{\eta} \right \Vert_{\mathbb{W}^{0, 0}_{\infty, \infty} (I\times \Omega) }
\\
& \quad + \Vert   \overline{\varrho} \Vert_{\mathbb{W}^{0, 0}_{\infty, \infty} (I\times \Omega) } \Vert \overline{\bv} - \overline{\bv}_0\Vert_{\mathbb{W}^{0, 0}_{\infty, \infty} (I\times \Omega) }  \Vert  \nabla\overline{\bv} \Vert_{\mathbb{W}^{0, 0}_{2, 2} (I\times \Omega) } \left\Vert \partial_t \mathbf{B}_\eta \right \Vert_{\mathbb{W}^{0, 0}_{\infty, \infty} (I\times \Omega) } 
\\
& \quad +  \Vert   \overline{\varrho} \Vert_{\mathbb{W}^{0, 0}_{\infty, \infty} (I\times \Omega) } \Vert  \overline{\bv}_0\Vert_{L^{\infty}(\Omega)  }  \Vert  \nabla\overline{\bv} \Vert_{\mathbb{W}^{0, 0}_{2, 2} (I\times \Omega) } \left\Vert \partial_t \mathbf{B}_\eta \right \Vert_{\mathbb{W}^{0, 0}_{\infty, \infty} (I\times \Omega) } 
\\
& \leq CT^{1/2}  + c\bigg( \Vert \left( \overline{\bv}_0, \eta_0, \eta_{*}  \right) \Vert_{\mathcal{I} }^{2} +  \sup_{t\in I}\left( \Vert \overline{\varrho}^{\gamma}(t)\Vert_{W^{3,2}(\Omega )}^2 + \Vert \partial_t\overline{\varrho}^{\gamma}(t)\Vert_{W^{2,2}(\Omega )}^2 \right) \bigg).
\end{align*}\\
Consequently, we have 
\begin{equation}
\Vert \partial_t\mathbf{h}_\eta \Vert_{\mathbb{W}^{0, 0}_{2,2} (I\times \Omega)  } \leq CT^{1/2} + \dfrac{R}{8}. 
\end{equation}\\
We now turn to the estimate for  $ \mathbf{H}_\eta, $ for which we obtain

\begin{align*}
\Vert \mathbf{H}_\eta \Vert_{\mathbb{W}^{0, 3}_{2,2} (I\times \Omega) } &\leq \mu \left \Vert \mathbf{A}_{\eta_0} - \mathbf{A}_{\eta}  \right \Vert_{\mathbb{W}^{0, 3}_{\infty, 2} (I\times \Omega) } \Vert \nabla\overline{\bv} \Vert_{\mathbb{W}^{0, 3}_{2,2} (I\times \Omega) } 
\\
&\quad + (\lambda + \mu) \left\Vert \dfrac{1}{J_{\eta_0}} -  \dfrac{1}{J_{\eta}}  \right\Vert_{\mathbb{W}^{0, 3}_{\infty, 2} (I\times \Omega) }  \Vert \nabla\overline{\bv} \Vert_{\mathbb{W}^{0, 3}_{2,2} (I\times \Omega) }  \left \Vert \mathbf{B}_{\eta_0} \right\Vert_{ W^{3,2}(\Omega) }^{2}  
\\
& \quad + (\lambda + \mu) \left\Vert \dfrac{1}{J_{\eta}}  \right\Vert_{\mathbb{W}^{0, 3}_{\infty,2} (I\times \Omega) }  \Vert \nabla\overline{\bv} \Vert_{\mathbb{W}^{0, 3}_{2,2} (I\times \Omega) }  \left \Vert \mathbf{B}_{\eta_0} - \mathbf{B}_{\eta} \right \Vert_{\mathbb{W}^{0, 3}_{\infty, 2} (I\times \Omega) }  \Vert \mathbf{B}_{\eta_0} \Vert_{ W^{3,2}(\Omega) }
\\
& \quad + (\lambda + \mu) \left\Vert \dfrac{1}{J_{\eta}}  \right\Vert_{\mathbb{W}^{0, 3}_{\infty, 2} (I\times \Omega) }  \Vert \nabla\overline{\bv} \Vert_{\mathbb{W}^{0, 3}_{2,2} (I\times \Omega) } \left \Vert \mathbf{B}_{\eta_0} - \mathbf{B}_{\eta} \right\Vert_{\mathbb{W}^{0, 3}_{\infty,2} (I\times \Omega) }  \Vert \mathbf{B}_{\eta} \Vert_{\mathbb{W}^{0, 3}_{\infty, 2} (I\times \Omega) }
\\
& \quad + a T^{1/2}  \left \Vert \mathbf{B}_{\eta} - \mathbf{B}_{\eta_0}\right \Vert_{\mathbb{W}^{0, 3}_{\infty, 2} (I\times \Omega) } \Vert  \overline{\varrho}^{\gamma} \Vert_{\mathbb{W}^{0, 3}_{\infty, 2} (I\times \Omega) } 
\\
& \leq CT^{1/2}.  
\end{align*}\\
As a result,

\begin{equation}
\Vert \mathbf{H}_\eta \Vert_{\mathbb{W}^{0, 3}_{2,2} (I\times \Omega) }  \leq CT^{1/2}. 
\end{equation} \\
Lastly, we derive the estimate for $ \partial_t \mathbf{H}_\eta, $ which satisfies

\begin{align*}
\Vert \partial_t \mathbf{H}_\eta \Vert_{\mathbb{W}^{0, 1}_{2,2} (I\times \Omega) }  &\leq \mu  \Vert \mathbf{A}_{\eta_0} - \mathbf{A}_{\eta}   \Vert_{\mathbb{W}^{0, 1}_{\infty, \infty} (I\times \Omega) } \Vert \partial_t \nabla\overline{\bv} \Vert_{\mathbb{W}^{0, 1}_{2,2} (I\times \Omega) } 
\\
&\quad + \mu \Vert \partial_t \mathbf{A}_\eta \Vert_{\mathbb{W}^{0, 1}_{\infty, \infty} (I\times \Omega) }  \Vert \nabla\overline{\bv} \Vert_{\mathbb{W}^{0, 1}_{2,2} (I\times \Omega) }
\\
& \quad + (\lambda + \mu) \left\Vert \dfrac{1}{J_{\eta_0} } \right\Vert_{W^{1, \infty}(\Omega) }  \Vert \mathbf{B}_{\eta_0} \Vert_{ W^{1,\infty}(\Omega) }^{2} \Vert \partial_t \nabla\overline{\bv} \Vert_{\mathbb{W}^{0, 1}_{2,2} (I\times \Omega) }  
\\
& \quad +  (\lambda + \mu) \left\Vert  \dfrac{1}{J_{\eta} } \right\Vert_{\mathbb{W}^{0, 1}_{\infty, \infty} (I\times \Omega)  } \Vert \mathbf{B}_{\eta} \Vert_{\mathbb{W}^{0, 1}_{\infty, \infty} (I\times \Omega) }^2 \Vert \partial_t \nabla\overline{\bv} \Vert_{\mathbb{W}^{0, 1}_{2,2} (I\times \Omega) } 
\\
& \quad + (\lambda + \mu) \left\Vert \partial_t \left( \dfrac{1}{J_{\eta} } \right)  \right\Vert_{\mathbb{W}^{0, 1}_{\infty, \infty} (I\times \Omega) }   \Vert \mathbf{B}_{\eta} \Vert_{\mathbb{W}^{0, 1}_{\infty, \infty} (I\times \Omega) }^{2}   \Vert \nabla\overline{\bv} \Vert_{\mathbb{W}^{0, 1}_{2,2} (I\times \Omega) }
\\
& \quad + 2 (\lambda + \mu)  \left\Vert \dfrac{1}{J_{\eta} } \right\Vert_{\mathbb{W}^{0, 1}_{\infty, \infty} (I\times \Omega) } \Vert \partial_t \mathbf{B}_{\eta} \Vert_{\mathbb{W}^{0, 1}_{\infty, \infty} (I\times \Omega) }\Vert \mathbf{B}_{\eta} \Vert_{\mathbb{W}^{0, 1}_{\infty, \infty} (I\times \Omega) }  \Vert \nabla\overline{\bv} \Vert_{\mathbb{W}^{0, 1}_{2,2} (I\times \Omega) }
\\
&\quad + a T^{1/2} \Vert \partial_t \mathbf{B}_{\eta} \Vert_{\mathbb{W}^{0, 1}_{\infty, \infty} (I\times \Omega) }  \Vert  \overline{\varrho}^{\gamma} \Vert_{\mathbb{W}^{0, 1}_{\infty, 2} (I\times \Omega) } 
\\
& \quad +  a T^{1/2} \Vert \mathbf{B}_{\eta} - \mathbf{B}_{\eta_0} \Vert_{\mathbb{W}^{0, 0}_{\infty, \infty} (I\times \Omega) }  \Vert  \partial_t \left( \overline{\varrho}^{\gamma}\right) \Vert_{\mathbb{W}^{0, 1}_{\infty, 2} (I\times \Omega) }
\\
& \leq CT^{1/2} +  c\bigg( \Vert \left( \overline{\bv}_0, \eta_0, \eta_{*}  \right) \Vert_{\mathcal{I} }^{2} + \sup_{t\in I}\left( \Vert \overline{\varrho}^{\gamma}(t)\Vert_{W^{3,2}(\Omega )}^2 + \Vert \partial_t\overline{\varrho}^{\gamma}(t)\Vert_{W^{2,2}(\Omega )}^2 \right)  \bigg).
\end{align*}\\
Hence,

\begin{equation}
\Vert \partial_t \mathbf{H}_\eta \Vert_{\mathbb{W}^{0, 1}_{2,2} (I\times \Omega) } \leq CT^{1/2} + \dfrac{R}{8}.
\end{equation}\\
Combining all the estimates above, we deduce that 
\begin{equation}
\Vert (\mathbf{h}_\eta, \mathbf{H}_\eta) \Vert_{\mathcal{S}}  = \left\Vert \mathcal{F}( \mathbf{h}, \mathbf{H}) \right\Vert_{\mathcal{S}} \leq CT^{1/2} +  \dfrac{R}{ 2}.
\end{equation} \\
Accordingly,  for a suitably small choice of $T > 0, $  we have 
\[ \mathcal{F}\left( \mathcal{K}_{T, R} \right) \subset  \mathcal{K}_{T, R},  \]
whence  $ \mathcal{K}_{T, R} $ is $\mathcal{F}-$invariant. \\

\medskip

\noindent\textbf{Step 2: }  $\mathcal{F} $ is a strict contraction. \\
Let  $\left( \overline{\bv}_i, \eta_i    \right),  \; i \in \{1, 2\}, $ be respectively the solution of the linearised system  \eqref{linMomAlone}-\eqref{linShelAlone},   with source terms $ \left( \mathbf{h}_i,  \mathbf{H}_i \right)  \in \mathcal{K}_{T, R}, $ and initial condition $\left( \overline{\bv}_0, \eta_0, \eta_*  \right).$ 
We introduce the following variables:
\[\left( \overline{\mathbf{h}}, \overline{\mathbf{H}}   \right) := \left( \mathbf{h}_1,  \mathbf{H}_1 \right) - \left( \mathbf{h}_2,  \mathbf{H}_2 \right), \quad  \left( \overline{ \overline{\bv}}, \overline{ \overline{\eta}}   \right) :=  \left( \overline{\bv}_1, \eta_1    \right)  - \left(  \overline{\bv}_2, \eta_2    \right).   \]

Then due to the linear nature of the system \eqref{linMomAlone}-\eqref{linShelAlone}, $ \left( \overline{ \overline{\bv}}, \overline{ \overline{\eta}}   \right) $ solves 
\eqref{linMomAlone}-\eqref{linShelAlone} with source term $\left( \overline{\mathbf{h}}, \overline{\mathbf{H}}   \right), $ and initial condition $(0,0,0).$
Thus, by Proposition \ref{thm:transformedSystem}, it follows that 
\begin{equation}\label{eq:FpDiffnormSol}
\Vert \left( \overline{ \overline{\bv}}, \overline{ \overline{\eta}}   \right) \Vert_{ \raisebox{-1.5ex}{$ \mathcal{X}$}_{\overline\bv}  \raisebox{-1.5ex}{$\times \mathcal{X}$}_{\eta} } \leq C \Vert \left( \overline{\mathbf{h}}, \overline{\mathbf{H}}   \right) \Vert_{\mathcal{S}}.
\end{equation}\\
Moreover, we have that
\begin{equation*}
\begin{aligned}
\Vert J_{\eta_1} - J_{\eta_2} \Vert_{\mathbb{W}^{0, 3}_{\infty, 2} (I\times \Omega) } &\leq \overline{C} \Vert \nabla\bfPsi_{\eta_1} - \nabla\bfPsi_{\eta_2}\Vert_{\mathbb{W}^{0, 3}_{\infty, 2} (I\times \Omega) }
\\
&\leq \overline{C} \Vert \bfPsi_{\eta_1} - \bfPsi_{\eta_2}\Vert_{\mathbb{W}^{0, 4}_{\infty, 2} (I\times \Omega) }
\\
&\leq \overline{C} \Vert \eta_1 - \eta_2 \Vert_{\mathbb{W}^{0, 4}_{\infty, 2} (I\times \omega) }
\\
&\leq \overline{C}T^{1/2}\Vert \partial_t \left( \eta_1 - \eta_2 \right)  \Vert_{\mathbb{W}^{0, 4}_{2, 2} (I\times \omega) }, 
\end{aligned}
\end{equation*}\\
whence 

\begin{equation}\label{eq:EstimJeta}
\Vert J_{\eta_1} - J_{\eta_2} \Vert_{\mathbb{W}^{0, 3}_{\infty, 2} (I\times \Omega) } \leq  CT^{1/2} \Vert \left( \overline{\mathbf{h}}, \overline{\mathbf{H}}   \right) \Vert_{\mathcal{S}}.
\end{equation}\\
Similarly, one has 

\begin{equation*}
\begin{aligned}
\Vert \mathbf{B}_{\eta_1} - \mathbf{B}_{\eta_2} \Vert_{\mathbb{W}^{0, 3}_{\infty, 2} (I\times \Omega) } &\leq \Vert J_{\eta_1}\Vert_{\mathbb{W}^{0, 3}_{\infty, 2} (I\times \Omega) } \Vert \nabla\bfPsi_{\eta_1}^{-1}\circ \bfPsi_{\eta_1} - \nabla\bfPsi_{\eta_2}^{-1}\circ \bfPsi_{\eta_2} \Vert_{\mathbb{W}^{0, 3}_{\infty, 2} (I\times \Omega) } 
\\
&\quad   + \Vert J_{\eta_1} - J_{\eta_2}\Vert_{\mathbb{W}^{0, 3}_{\infty, 2} (I\times \Omega) } \Vert \nabla\bfPsi_{\eta_2}^{-1}\circ \bfPsi_{\eta_2} \Vert_{\mathbb{W}^{0, 3}_{\infty, 2} (I\times \Omega) } 
\\
&\leq C_1 \Vert J_{\eta_1}\Vert_{\mathbb{W}^{0, 3}_{\infty, 2} (I\times \Omega) } \Vert \eta_1 - \eta_2 \Vert_{\mathbb{W}^{0, 4}_{\infty, 2} (I\times \omega) } 
\\
&\quad + C_2 \Vert J_{\eta_1} - J_{\eta_2}\Vert_{\mathbb{W}^{0, 3}_{\infty, 2} (I\times \Omega) } \left( 1 + \Vert \eta_2 \Vert_{\mathbb{W}^{0, 4}_{\infty, 2} (I\times \omega) }   \right). 
\end{aligned}
\end{equation*}\\
Thus, 

\begin{equation}\label{eq:EstimBeta}
\Vert  \mathbf{B}_{\eta_1} -  \mathbf{B}_{\eta_2} \Vert_{\mathbb{W}^{0, 3}_{\infty, 2} (I\times \Omega) } \leq  CT^{1/2} \Vert \left( \overline{g}, \overline{\mathbf{h}}, \overline{\mathbf{H}}   \right) \Vert_{\mathcal{S}}.
\end{equation}\\

\begin{equation*}
\begin{aligned}
\Vert \mathbf{A}_{\eta_1} - \mathbf{A}_{\eta_2} \Vert_{\mathbb{W}^{0, 3}_{\infty, 2} (I\times \Omega) } &\leq \Vert \nabla\bfPsi_{\eta_1}^{-1} \circ \bfPsi_{\eta_1} \Vert_{\mathbb{W}^{0, 3}_{\infty, 2} (I\times \Omega) } \Vert  \mathbf{B}_{\eta_1} -  \mathbf{B}_{\eta_2} \Vert_{\mathbb{W}^{0, 3}_{\infty, 2} (I\times \Omega) } 
\\
&\quad + \left\Vert \mathbf{B}_{\eta_2} \right\Vert_{\mathbb{W}^{0, 3}_{\infty, 2} (I\times \Omega) }  \Vert \nabla\bfPsi_{\eta_1}^{-1}\circ \bfPsi_{\eta_1} - \nabla\bfPsi_{\eta_2}^{-1} \circ \bfPsi_{\eta_2} \Vert_{\mathbb{W}^{0, 3}_{\infty, 2} (I\times \Omega) }, 
\end{aligned}
\end{equation*}\\
 whence
 
 \begin{equation}
\Vert  \mathbf{A}_{\eta_1} -  \mathbf{A}_{\eta_2} \Vert_{\mathbb{W}^{0, 3}_{\infty, 2} (I\times \Omega) } \leq  CT^{1/2} \Vert \left( \overline{\mathbf{h}}, \overline{\mathbf{H}}   \right) \Vert_{\mathcal{S}}.
\end{equation}\\
 Therefore, we deduce that 
 
 \begin{equation*}
 \begin{aligned}
 \Vert \mathbf{h}_{\eta_1} - \mathbf{h}_{\eta_2} \Vert_{\mathbb{W}^{0, 2}_{2, 2} (I\times \Omega) } &\leq   \Vert J_{\eta_0} - J_{\eta_2} \Vert_{\mathbb{W}^{0, 2}_{\infty,  2} (I\times \Omega) }  \Vert \overline{\varrho}_0  \Vert_{ W^{2,2}(\Omega)  }\Vert \partial_t \left( \overline{\bv}_1 - \overline{\bv}_2 \right) \Vert_{\mathbb{W}^{0, 2}_{2, 2} (I\times \Omega) } 
 \\
 & \quad +  \Vert J_{\eta_0} \Vert_{W^{2,2}(\Omega)}  \Vert \overline{\varrho}_0 - \overline{\varrho} \Vert_{\mathbb{W}^{0, 2}_{\infty, 2} (I\times \Omega) } \Vert \partial_t \left( \overline{\bv}_1 - \overline{\bv}_2 \right) \Vert_{\mathbb{W}^{0, 2}_{2, 2} (I\times \Omega) } 
 \\
 & \quad + \Vert J_{\eta_2} - J_{\eta_1} \Vert_{\mathbb{W}^{0, 2}_{\infty, 2} (I\times \Omega) }  \Vert  \overline{\varrho}  \Vert_{\mathbb{W}^{0, 2}_{\infty, 2} (I\times \Omega) }  \Vert \partial_t \overline{\bv}_1 \Vert_{\mathbb{W}^{0, 2}_{2, 2} (I\times \Omega) } 
 \\
 & \quad + \Vert J_{\eta_2} - J_{\eta_1} \Vert_{\mathbb{W}^{0, 2}_{\infty, 2} (I\times \Omega)  }  \Vert \overline{\varrho} \Vert_{\mathbb{W}^{0, 2}_{\infty, 2} (I\times \Omega) }  \Vert  \nabla\overline{\bv}_1 \Vert_{\mathbb{W}^{0, 2}_{2, 2} (I\times \Omega) } \Vert \partial_t \bfPsi_{\eta_1}^{-1}\circ \bfPsi_{\eta_1} \Vert_{\mathbb{W}^{0, 2}_{\infty, 2} (I\times \Omega) } 
 \\
 & \quad +  \Vert J_{\eta_2} \overline{\varrho} \Vert_{\mathbb{W}^{0, 2}_{\infty, 2} (I\times \Omega) }  \Vert  \nabla ( \overline{\bv}_2 - \overline{\bv}_1 )\Vert_{\mathbb{W}^{0, 2}_{2, 2} (I\times \Omega) } \Vert \partial_t \bfPsi_{\eta_1}^{-1}\circ \bfPsi_{\eta_1} \Vert_{\mathbb{W}^{0, 2}_{\infty, 2} (I\times \Omega) }
 \\
 & \quad +  \Vert J_{\eta_2} \overline{\varrho} \Vert_{\mathbb{W}^{0, 2}_{\infty, 2} (I\times \Omega) }  \Vert  \nabla \overline{\bv}_2 \Vert_{\mathbb{W}^{0, 2}_{2, 2} (I\times \Omega) } \Vert \partial_t \bfPsi_{\eta_2}^{-1}\circ \bfPsi_{\eta_2} - \partial_t \bfPsi_{\eta_1}^{-1}\circ \bfPsi_{\eta_1} \Vert_{\mathbb{W}^{0, 2}_{\infty, 2} (I\times \Omega) }
 \\
 & \quad +  \Vert  \overline{\varrho} \Vert_{\mathbb{W}^{0, 2}_{\infty, 2} (I\times \Omega) }  \Vert   \overline{\bv}_2 - \overline{\bv}_1 \Vert_{\mathbb{W}^{0, 2}_{\infty, 2} (I\times \Omega) } \Vert \nabla\overline{\bv}_1 \Vert_{\mathbb{W}^{0, 2}_{2, 2} (I\times \Omega) } \Vert \mathbf{B}_{\eta_1}\Vert_{\mathbb{W}^{0, 2}_{\infty, 2} (I\times \Omega) }
 \\
 & \quad + \Vert  \overline{\varrho} \Vert_{\mathbb{W}^{0, 2}_{\infty, 2} (I\times \Omega) }  \Vert   \overline{\bv}_2 - \overline{\bv}_0 \Vert_{\mathbb{W}^{0, 2}_{\infty, 2} (I\times \Omega) } \Vert \nabla(\overline{\bv}_2 - \overline{\bv}_1) \Vert_{\mathbb{W}^{0, 2}_{2, 2} (I\times \Omega) } \Vert \mathbf{B}_{\eta_1} \Vert_{\mathbb{W}^{0, 2}_{\infty, 2} (I\times \Omega) }
 \\
 & \quad +  \Vert  \overline{\varrho} \Vert_{\mathbb{W}^{0, 2}_{\infty, 2} (I\times \Omega) }  \Vert  \overline{\bv}_0 \Vert_{W^{2,2}(\Omega) } \Vert \nabla(\overline{\bv}_2 - \overline{\bv}_1) \Vert_{\mathbb{W}^{0, 2}_{2, 2} (I\times \Omega) } \Vert \mathbf{B}_{\eta_1}\Vert_{\mathbb{W}^{0, 2}_{\infty, 2} (I\times \Omega) }
 \\
 & \quad +  \Vert  \overline{\varrho} \Vert_{\mathbb{W}^{0, 2}_{\infty, 2} (I\times \Omega) }   \Vert   \overline{\bv}_2 - \overline{\bv}_0 \Vert_{\mathbb{W}^{0, 2}_{\infty, 2} (I\times \Omega) }  \Vert \nabla\overline{\bv}_2  \Vert_{\mathbb{W}^{0, 2}_{2, 2} (I\times \Omega) } \left\Vert  \mathbf{B}_{\eta_2} - \mathbf{B}_{\eta_1}\right \Vert_{\mathbb{W}^{0, 2}_{\infty, 2} (I\times \Omega) }
 \\
 & \quad + \Vert  \overline{\varrho} \Vert_{\mathbb{W}^{0, 2}_{\infty, 2} (I\times \Omega) }   \Vert  \overline{\bv}_0 \Vert_{ W^{2,2}(\Omega) }  \Vert \nabla\overline{\bv}_2  \Vert_{\mathbb{W}^{0, 2}_{2, 2} (I\times \Omega) } \left\Vert \mathbf{B}_{\eta_2} - \mathbf{B}_{\eta_1}\right \Vert_{\mathbb{W}^{0, 2}_{\infty, 2} (I\times \Omega) },
 \end{aligned}
 \end{equation*}\\
whence

 \begin{equation}\label{eq:hdiffEstimate}
\Vert \mathbf{h}_{\eta_1} - \mathbf{h}_{\eta_2} \Vert_{\mathbb{W}^{0, 2}_{2, 2} (I\times \Omega) }  \leq  CT^{1/4} \Vert \left( \overline{\mathbf{h}}, \overline{\mathbf{H}}   \right) \Vert_{\mathcal{S}}.
\end{equation}\\ 
Furthermore, we have 

\begin{align*}
\Vert \mathbf{H}_{\eta_1} - \mathbf{H}_{\eta_2} \Vert_{\mathbb{W}^{0, 3}_{2, 2} (I\times \Omega) } &\leq \mu \left \Vert \mathbf{A}_{\eta_2} - \mathbf{A}_{\eta_1}  \right \Vert_{\mathbb{W}^{0, 3}_{\infty, 2} (I\times \Omega) } \Vert \nabla\overline{\bv}_1 \Vert_{\mathbb{W}^{0, 3}_{2, 2} (I\times \Omega) } 
\\
& \quad +  \mu \left \Vert \mathbf{A}_{\eta_0} - \mathbf{A}_{\eta_2}  \right \Vert_{\mathbb{W}^{0, 3}_{\infty, 2} (I\times \Omega) } \Vert \nabla(\overline{\bv}_1 - \overline{\bv}_2) \Vert_{\mathbb{W}^{0, 3}_{2, 2} (I\times \Omega) } 
\\
& \quad + (\lambda + \mu) \left\Vert \dfrac{1}{J_{\eta_0}}  \right\Vert_{ W^{3,2}(\Omega) }  \left \Vert \mathbf{B}_{\eta_0} \right\Vert_{ W^{3,2}(\Omega) }^{2}  \Vert \nabla(\overline{\bv}_1 - \overline{\bv}_2) \Vert_{\mathbb{W}^{0, 3}_{2, 2} (I\times \Omega) } 
\\
& \quad +  (\lambda + \mu) \left\Vert \dfrac{1}{J_{\eta_2} } -  \dfrac{1}{J_{\eta_1} }  \right\Vert_{\mathbb{W}^{0, 3}_{\infty, 2} (I\times \Omega)  } \left \Vert \mathbf{B}_{\eta_1} \right\Vert_{\mathbb{W}^{0, 3}_{\infty, 2} (I\times \Omega) }^{2}  \Vert \nabla \overline{\bv}_1 \Vert_{\mathbb{W}^{0, 3}_{2, 2} (I\times \Omega) } 
\\
& \quad + (\lambda + \mu) \left\Vert \dfrac{1}{J_{\eta_2}}  \right\Vert_{\mathbb{W}^{0, 3}_{\infty, 2} (I\times \Omega) } \!\!\!  \left \Vert \mathbf{B}_{\eta_2} - \mathbf{B}_{\eta_1} \right \Vert_{\mathbb{W}^{0, 3}_{\infty, 2} (I\times \Omega) }   \Vert \nabla\overline{\bv}_1 \Vert_{\mathbb{W}^{0, 3}_{2, 2} (I\times \Omega) }    \Vert \mathbf{B}_{\eta_1} \Vert_{\mathbb{W}^{0, 3}_{\infty, 2} (I\times \Omega) }
\\
& \quad + (\lambda + \mu) \left\Vert \dfrac{1}{J_{\eta_2}}  \right\Vert_{\mathbb{W}^{0, 3}_{\infty, 2} (I\times \Omega) }  \left \Vert \mathbf{B}_{\eta_2} \right \Vert_{\mathbb{W}^{0, 3}_{\infty, 2} (I\times \Omega) }^2   \Vert \nabla (\overline{\bv}_2 - \overline{\bv}_1) \Vert_{\mathbb{W}^{0, 3}_{2, 2} (I\times \Omega) }  
\\
& \quad +  (\lambda + \mu) \left\Vert \dfrac{1}{J_{\eta_2}}  \right\Vert_{\mathbb{W}^{0, 3}_{\infty, 2} (I\times \Omega) } \!\!\! \left \Vert \mathbf{B}_{\eta_2} - \mathbf{B}_{\eta_1} \right \Vert_{\mathbb{W}^{0, 3}_{\infty, 2} (I\times \Omega) }   \Vert \nabla\overline{\bv}_1 \Vert_{\mathbb{W}^{0, 3}_{2, 2} (I\times \Omega) }    \Vert \mathbf{B}_{\eta_2} \Vert_{\mathbb{W}^{0, 3}_{\infty, 2} (I\times \Omega) }
\\
& \quad + a T^{1/2}  \left \Vert \mathbf{B}_{\eta_1} - \mathbf{B}_{\eta_2}\right \Vert_{\mathbb{W}^{0, 3}_{\infty, 2} (I\times \Omega) } \Vert  \overline{\varrho}^{\gamma} \Vert_{\mathbb{W}^{0, 3}_{\infty, 2} (I\times \Omega) }, 
\end{align*}\\
whence 

\begin{equation}
\Vert  \mathbf{H}_{\eta_1} - \mathbf{H}_{\eta_2}   \Vert_{\mathbb{W}^{0, 3}_{2, 2} (I\times \Omega) }  \leq  CT^{1/4} \Vert \left( \overline{\mathbf{h}}, \overline{\mathbf{H}}   \right) \Vert_{\mathcal{S}}.
\end{equation}\\  
In a similar fashion, we obtain the estimates 

 \begin{equation}
\Vert \partial_t \left( \mathbf{h}_{\eta_1} - \mathbf{h}_{\eta_2}  \right)  \Vert_{\mathbb{W}^{0, 0}_{2, 2} (I\times \Omega) }  \leq  CT^{1/4} \Vert \left(\overline{\mathbf{h}}, \overline{\mathbf{H}}   \right) \Vert_{\mathcal{S}},
\end{equation}\\ 
and

\begin{equation}\label{eq:HdiffEstimate}
\Vert  \partial_t \left( \mathbf{H}_{\eta_1} - \mathbf{H}_{\eta_2} \right)   \Vert_{\mathbb{W}^{0, 1}_{2, 2} (I\times \Omega) }  \leq  CT^{1/4} \Vert \left( \overline{\mathbf{h}}, \overline{\mathbf{H}}   \right) \Vert_{\mathcal{S}}.
\end{equation}\\ 
Thus, 

\begin{equation}
\Vert  \mathcal{F}\left( \mathbf{h}_1,  \mathbf{H}_1 \right) -  \mathcal{F}\left( \mathbf{h}_2,  \mathbf{H}_2 \right)   \Vert_{\mathcal{S}} \leq  CT^{1/4} \Vert \left( \overline{\mathbf{h}}, \overline{\mathbf{H}}   \right) \Vert_{\mathcal{S}}.
\end{equation}\\ 
Hence, by choosing $T >0 $ sufficiently small, $\mathcal{F} $ is a strict contraction.  \\

\end{proof} 
 \noindent  This concludes the proof of Theorem \ref{thm:transformedSystem1}.

 
 \section{Local Strong Solutions}\label{sec:localstrongfixedpoint}
 We now establish the existence of a local-in-time strong solution of the fully coupled system \eqref{eq:ShellEq}-\eqref{eq:ContMomentEq}. Our approach relies on combining the subproblems into a single fixed-point framework. For notational consistency, we continue to work on the time interval $ I = (0, T) $  -- with $ T> 0 $ implicitly restricted to ensure that all subsequent estimates remain valid -- and introduce the spaces 
\begin{equation*}
\begin{aligned}
  \mathcal{X}_{\overline{\rho}} &:=   \mathbb{W}^{0, 3}_{\infty,2} \left(I\times \Omega\right)\cap \mathbb{W}^{1, 2}_{\infty,2} \left(I\times \Omega\right), &
\end{aligned}
\end{equation*}
endowed with the norm
\begin{equation*}
\begin{aligned}
 \Vert \overline{\rho}\Vert_{ \raisebox{-1.5ex}{$\mathcal{X}$}_{\overline{\rho}} } &:=  \sup_{t\in I}\Big( \Vert \overline{\rho}(t)\Vert_{W^{3,2}(\Omega)}
+
\Vert \partial_t\overline{\rho}(t)\Vert_{W^{2,2}(\Omega )} 
\Big). &
\end{aligned}
\end{equation*}
 \noindent For $ \overline{\rho} \in  \mathcal{X}_{\overline\rho},  $ let  $(\overline\bv, \eta) $ be the unique strong solution of the momentum-structure subproblem corresponding to the system \eqref{momEqAloneBar}--\eqref{shellEqAloneBar}, with data $ \left( \overline{\rho}, \overline{\bv}_0, \eta_0, \eta_*\right)$. Given such a pair $(\overline\bv, \eta), $ we define $\overline{\rho}^{\#}$ as the unique strong solution of \eqref{rhoEquAloneTransform}--\eqref{initialCondSolvSubProTransform} with dataset $\left(\overline{\rho}_0, \overline\bv, \eta  \right)$. \\
This induces the mapping $\mathbf{F} = \mathbf{F}_1 \circ \mathbf{F}_2 $ where 
\[
\mathbf{F}(\overline{\rho}) = \overline{\rho}^{\#}, \qquad 
\mathbf{F}_2(\overline{\rho}) = (\overline\bv, \eta) , \quad \text{and} \;\;\;  \mathbf{F}_1(\overline\bv, \eta) = \overline{\rho}^{\#}. 
\]
 We shall consider $\mathbf{F} $ on the closed ball 
 \[\mathcal{B}_R := \bigg\{ \overline{\rho} \in \mathcal{X}_{\overline\rho} \colon  \Vert \overline{\rho}\Vert_{ \raisebox{-1.5ex}{$\mathcal{X}$}_{\overline\rho} } \leq R   \bigg\},  \]
 for $R > 0 $ chosen sufficiently large so that, for every $\overline\rho \in \mathcal{B}_R $, the  source terms determined by $\overline\rho $ in the linearised momentum-structure subproblem remain within the functional class $\mathcal{K}_{T, R} $  for which the fixed-point argument applies (cf. Section~\ref{subsec:BanachFixed-PointSub}).  \\
 The remaining part of this section is devoted to showing that, for a sufficiently small time $T > 0 $, the map  $\mathbf{F} \colon  \mathcal{X}_{\overline\rho} \to  \mathcal{X}_{\overline\rho} $ maps $ \mathcal{B}_R $ into itself and satisfies a strict contraction property. This yields the existence of a unique fixed point of  $\mathbf{F} $, and thereby of a strong solution to the fully coupled system \eqref{eq:ShellEq}--\eqref{interfaceCond}.\\

\medskip

\noindent\textbf{Step 1: }  $\mathbf{F}\big( \mathcal{B}_R  \big) \subset \mathcal{B}_R  $. \\
Let $\overline\rho \in \mathcal{B}_R $,  then by the a priori estimate \eqref{eq:ContSubProbEstimate},  

\begin{equation}
 \begin{aligned} \label{eq:lrhoHatEstimate}
\sup_{t\in I} &\Big( \Vert \overline{\rho}^{\#}(t)\Vert_{W^{3,2}(\Omega )}^2 
+
\Vert \partial_t\overline{\rho}^{\#}(t)\Vert_{W^{2,2}(\Omega )}^2 
\Big)
\\&\lesssim
 \Vert  \overline{\rho}_0\Vert_{W^{3,2}(\Omega)}^2  
\Bigg(1 + \sup\limits_I \Vert \partial_t \eta \Vert_{W^{3,2}(\omega)}^2  + \int_I\Vert  \overline\bv
\Vert_{W^{4,2}(\Omega )}^2 \dt
\\
&\quad + 
\int_I\Vert  \partial_{t}^2 \overline\bv
\Vert_{L^{2}(\Omega)}^2 \dt
\Bigg)
  \exp{\bigg( c\int_I \big(\Vert \partial_t \eta \Vert_{W^{4,2}(\omega)} + 
\Vert  \overline\bv\Vert_{W^{4,2}(\Omega)}\big)  \dt \bigg)}.
\end{aligned}
\end{equation} 
\noindent Moreover, using \eqref{eq:StepOneInvarianceEstim}, it follows that  
\begin{equation}
 \begin{aligned} \label{eq:rrhoHatEstimate}
\sup_{t\in I} &\Big( \Vert \overline{\rho}^{\#}(t)\Vert_{W^{3,2}(\Omega )}^2 
+
\Vert \partial_t \overline{\rho}^{\#}(t)\Vert_{W^{2,2}(\Omega)}^2 
\Big)
\\&\leq
 C\Vert  \overline{\rho}_0\Vert_{W^{3,2}(\Omega)}^2 e^{CT} 
\bigg(\Vert \left(  \overline{\bv}_0, \eta_0, \eta_{*}  \right) \Vert_{\mathcal{I} }^{2}  +  T\sup_{t\in I}\left( \Vert \overline{\rho}\,^{\gamma}(t)\Vert_{W^{3,2}(\Omega )}^2 + \Vert \partial_t\overline{\rho}\,^{\gamma}(t)\Vert_{W^{2,2}(\Omega )}^2 \right) 
\bigg),
\end{aligned}
\end{equation}  
where the constant $C > 0 $ depends on the tubular neighbourhood radius $L > 0$, and linearly on $\mu $ and $\lambda$.  \\
Hence,  up to increasing the radius  $R > 0 $,  we obtain -- combining \eqref{eq:rrhoHatEstimate} and \eqref{eq:pressureEstimate} -- that for  $T> 0 $ small enough,  $\Vert \overline{\rho}^{\#}\Vert_{ \raisebox{-1.5ex}{$\mathcal{X}$}_{\overline\rho} } \leq R $. \\

\medskip

\noindent\textbf{Step 2: }  $\mathbf{F} $ is a strict contraction. \\
For each $i \in \{ 1, 2\}$, let $\overline{\rho}^{\#}_i $ denote the unique strong solution of the continuity equation \eqref{rhoEquAloneTransform}--\eqref{initialCondSolvSubProTransform} associated with the dataset $(\overline{\rho}_0, \overline{\bv}_i, \eta_i)$, that is, $ \overline{\rho}^{\#}_i = \mathbf{F}_1(\overline{\bv}_i, \eta_i) $.  Here  $(\overline{\bv}_i, \eta_i) $ is the unique strong solution of the momentum-structure subproblem corresponding to the system \eqref{momEqAloneBar}--\eqref{shellEqAloneBar}, with dataset $ \left( \overline{\rho}_i, \overline{\bv}_0, \eta_0, \eta_*\right) $, that is, $\mathbf{F}_2(\overline{\rho}_i) = (\overline{\bv}_i, \eta_i) $.  In line with the contraction framework, we aim to bound the difference $\overline{\rho}^{\#}_{1,2} :=  \overline{\rho}^{\#}_1 - \overline{\rho}^{\#}_2 $   in the norm $ \mathcal{X}_{\overline\rho} $ in terms of $ \overline{\bv}_{1,2} := \overline{\bv}_1 - \overline{\bv}_2 $ and $\eta_{1,2} := \eta_1 - \eta_2 $ in suitable norms. \\
For this purpose, we first observe that $\overline{\rho}^{\#}_{1,2}$ solves 
\begin{equation}\label{eq:ContEqFixedComb}
\left\{\begin{aligned}
&J_{\eta_1}\partial_t \overline{\rho}^{\#}_{1,2}
= 
\overline{g}_{\eta_{1,2}}  &\text{ for all }(t,\bx)\in I\times\Omega,\\
&\overline{\rho}^{\#}_{1,2} (0,\bx)= 0 &\text{ for all } \bx\in \Omega,
\end{aligned}\right.
\end{equation}
 with source term 
  \begin{align*}
\overline{g}_{\eta_{1,2}} & =  \big(J_{\eta_2} - J_{\eta_1}\big)\partial_t\overline{\rho}^{\#}_2  -  \overline{\rho}^{\#}_{1,2} \mathbf{B}_{\eta_2} \colon \nabla\overline{\bv}_2 +  \overline{\rho}^{\#}_{1} \big(  \mathbf{B}_{\eta_2} - \mathbf{B}_{\eta_1}\big) \colon \nabla\overline{\bv}_2  - \big( \mathbf{B}_{\eta_1} \colon \nabla\overline{\bv}_{1,2}  \big) \overline{\rho}^{\#}_{1}  
\\
&\quad + \big(J_{\eta_2} - J_{\eta_1}\big)  \nabla\overline{\rho}^{\#}_2\cdot \partial_t \bfPsi_{\eta_2}^{-1}\circ \bfPsi_{\eta_2}  - J_{\eta_1} \nabla\overline{\rho}^{\#}_{1,2} \cdot \partial_t \bfPsi_{\eta_2}^{-1}\circ \bfPsi_{\eta_2} 
\\
&\quad + J_{\eta_1} \nabla\overline{\rho}^{\#}_1 \cdot \big(  \partial_t \bfPsi_{\eta_2}^{-1}\circ \bfPsi_{\eta_2} -  \partial_t \bfPsi_{\eta_1}^{-1}\circ \bfPsi_{\eta_1} \big)  - \overline{\bv}_{1,2} \mathbf{B}_{\eta_2}  \nabla\overline{\rho}^{\#}_2  + \overline{\bv}_{1} \big(  \mathbf{B}_{\eta_2} - \mathbf{B}_{\eta_1}\big) \nabla\overline{\rho}^{\#}_2 - \overline{\bv}_{1} \mathbf{B}_{\eta_1} \nabla\overline{\rho}^{\#}_{1,2}.
\end{align*}
 An estimate in $ \mathcal{X}_{\overline{\rho}} $ for  $\overline{\rho}^{\#}_{1,2} $  cannot be closed due to the presence of terms involving $\nabla\overline{\rho}^{\#}_i $ in the source term, which would require  higher regularity than available. Hence, we perform the contraction argument in the weaker but more general topological space 
 \[ \mathcal{Y}_{\overline{\rho}} :=   \mathbb{W}^{0, 2}_{\infty,2} \left(I\times \Omega\right)\cap \mathbb{W}^{1, 1}_{\infty,2} \left(I\times \Omega\right), \]
equipped with its canonical norm  $\Vert \cdot \Vert_{ \raisebox{-1.5ex}{$\mathcal{Y}$}_{\overline{\rho}} }  $.
 
 As a preliminary step, observe that analogously to \eqref{eq:EstimJeta} and \eqref{eq:EstimBeta}, the following bounds hold: 
 \begin{align}
\Vert J_{\eta_2} - J_{\eta_1} \Vert_{\mathbb{W}^{0, 1}_{\infty,2} \left(I\times \Omega\right)} 
&\leq c_1 T^{1/2} \Vert \partial_t \eta_{1,2} \Vert_{\mathbb{W}^{0, 2}_{2,2} \left(I\times \Omega\right)},  \label{eq:GeoEstim1}
\\[0.15cm]
\Vert \mathbf{B}_{\eta_2} - \mathbf{B}_{\eta_1}\Vert_{\mathbb{W}^{0, 1}_{\infty,2} \left(I\times \Omega\right) }
&\leq c_2(\eta_1, \eta_2)T^{1/2} \Vert \partial_t \eta_{1,2} \Vert_{\mathbb{W}^{0, 2}_{2,2} \left(I\times \Omega\right)},
\\[0.15cm]
\Vert \partial_t \bfPsi_{\eta_2}^{-1} \circ \bfPsi_{\eta_2} 
- \partial_t \bfPsi_{\eta_1}^{-1} \circ \bfPsi_{\eta_1} \Vert_{\mathbb{W}^{0, 1}_{\infty,2} \left(I\times \Omega\right)} 
&\leq c_3(\Omega) T^{1/2}  \Vert \partial_{t}^2 \eta_{1,2}\Vert_{\mathbb{W}^{0, 2}_{2,2} \left(I\times \Omega\right)}.  \label{eq:GeoEstim3}
\end{align}
 Applying the operator norm $\mathbb{W}^{0, 1}_{\infty,2} \left(I\times \Omega\right)$  to both sides of \eqref{eq:ContEqFixedComb}, one obtains
\begin{equation}\label{eq:TimeDerivEstim}
\begin{aligned}
\Vert \partial_{t} \overline{\rho}^{\#}_{1,2}  \Vert_{\mathbb{W}^{0, 1}_{\infty,2} \left(I\times \Omega\right) } & \leq  c \Bigg( T^{1/2} \Big(  \Vert \partial_{t}^2 \eta_{1,2} \Vert_{\mathbb{W}^{0, 2}_{2,2} \left(I\times \Omega\right) } +  \Vert \overline{\bv}_{1,2} \Vert_{\mathbb{W}^{1, 2}_{2,2} \left(I\times \Omega\right) }  \Big) + \Vert  \overline{\rho}^{\#}_{1,2}  \Vert_{\mathbb{W}^{0, 2}_{\infty,2} \left(I\times \Omega\right) } \Bigg),
\end{aligned}
\end{equation} 
with  constant $c = c\Big( \Omega, \eta_1, \eta_2,   \overline{\bv}_1,  \overline{\bv}_2, \overline{\rho}^{\#}_1, \overline{\rho}^{\#}_2  \Big)$.\\

 \noindent To estimate $ \Vert  \overline{\rho}^{\#}_{1,2}  \Vert_{\mathbb{W}^{0, 2}_{\infty,2} \left(I\times \Omega\right) }$, we recast \eqref{eq:ContEqFixedComb} as a linear transport-type system: 
 
 \begin{equation}\label{eq:LinTransportTypSyst}
\left\{\begin{aligned}
& \partial_t \overline{\rho}^{\#}_{1,2} + \overline{\mathbf{u}}_{eff}^{\#} \cdot \nabla \overline{\rho}^{\#}_{1,2}  +  \delta_{\Div}^{\#} \, \overline{\rho}^{\#}_{1,2}
= 
\overline{f}_{\eta_{1,2}}  &\text{ for all }(t,\bx)\in I\times\Omega,\\
&\overline{\rho}^{\#}_{1,2} (0,\bx)= 0 &\text{ for all } \bx\in \Omega,
\end{aligned}\right.
\end{equation}
 with  \[  \overline{\mathbf{u}}_{eff}^{\#} = \Big( \partial_t \bfPsi_{\eta_1}^{-1}\circ \bfPsi_{\eta_1} + \dfrac{1}{J_{\eta_1}}  \mathbf{B}^{\intercal}_{\eta_1}\overline{\bv}_1 \Big),\quad  \delta_{\Div}^{\#} = \dfrac{1}{J_{\eta_1}} \mathbf{B}_{\eta_1} \colon \nabla\overline{\bv}_1, \quad \text{and}   \]
 
  \begin{align*}
\overline{f}_{\eta_{1,2}}  & =  \nabla\overline{\rho}^{\#}_2 \cdot \big(  \partial_t \bfPsi_{\eta_1}^{-1}\circ \bfPsi_{\eta_1} -  \partial_t \bfPsi_{\eta_2}^{-1}\circ \bfPsi_{\eta_2} \big) +   \left(  \dfrac{1}{J_{\eta_1}} -  \dfrac{1}{J_{\eta_2}} \right)  \mathbf{B}_{\eta_1} \overline{\bv}_{1}  +  \dfrac{1}{J_{\eta_2}} \left(   \mathbf{B}_{\eta_1} -   \mathbf{B}_{\eta_2} \right) \overline{\bv}_{1}  +  \dfrac{1}{J_{\eta_2}} \mathbf{B}_{\eta_2} \overline{\bv}_{1, 2} 
\\
& \quad + \left(  \dfrac{1}{J_{\eta_1}} -  \dfrac{1}{J_{\eta_2}} \right)  \mathbf{B}_{\eta_1}\colon  \nabla\overline{\bv}_{1}  +  \dfrac{1}{J_{\eta_2}} \left(   \mathbf{B}_{\eta_1} -   \mathbf{B}_{\eta_2} \right)\colon  \nabla\overline{\bv}_{1}  +  \dfrac{1}{J_{\eta_2}} \mathbf{B}_{\eta_2} \colon \nabla\overline{\bv}_{1, 2} .
\end{align*}
Since $\overline{f}_{\eta_{1,2}}  \in \mathbb{W}^{0, 2}_{1,2} \left(I\times \Omega\right) $,\;  and the pointwise estimate  \[ \left\Vert \big( \delta_{\Div}^{\#} \,  \overline{\rho}^{\#}_{1,2} \big)(t)  \right\Vert_{W^{2,2}(\Omega)} \leq  \Vert \delta_{\Div}^{\#}\Vert_{W^{2,2}(\Omega)}  \Vert \overline{\rho}^{\#}_{1,2} (t)  \Vert_{W^{2,2}(\Omega)}  \]  
 holds for a.e. $t \in I$,  with $\Vert \delta_{\Div}^{\#}\Vert_{W^{2,2}(\Omega)} \in L^{1}(I)$,  \cite[Chapter 3, Theorem 3.14]{BahouriCheminDanchin2011} then yields -- for a constant $C > 0$ -- the estimate 
 
 \begin{equation}\label{eq:LinTransportTypEstim}
  \Vert  \overline{\rho}^{\#}_{1,2}  \Vert_{\mathbb{W}^{0, 2}_{\infty,2} \left(I\times \Omega\right) } \leq \Vert  \overline{f}_{\eta_{1,2}} \Vert_{ \mathbb{W}^{0, 2}_{1,2} \left(I\times \Omega\right) }   \exp\bigg(C \int_I \Big( \big\Vert \nabla \overline{\mathbf{u}}_{eff}^{\#} \big\Vert_{W^{2,2}(\Omega)}  + \Vert  \delta_{\Div}^{\#} \Vert_{W^{2,2}(\Omega)} \Big) \dt \bigg). 
 \end{equation}
 Exploiting the structure of $\overline{f}_{\eta_{1,2}}$  and standard bounds for the geometric coefficients -- cf.   \eqref{eq:GeoEstim1}-\eqref{eq:GeoEstim3} -- the estimate \eqref{eq:LinTransportTypEstim} further reduces to 
 \begin{equation}\label{eq:LinTransportTypEstimPrior}
  \Vert  \overline{\rho}^{\#}_{1,2}  \Vert_{\mathbb{W}^{0, 2}_{\infty,2} \left(I\times \Omega\right) } \leq \overline{c} T^{1/2}e^{\overline{c}T} \Big( \Vert \partial_t \eta_{1,2} \Vert_{\mathbb{W}^{0, 3}_{2,2} \left(I\times \Omega\right)\cap \mathbb{W}^{1, 2}_{2,2} \left(I\times \Omega\right) } +  \Vert \overline{\bv}_{1,2} \Vert_{\mathbb{W}^{1, 2}_{2,2} \left(I\times \Omega\right) \cap \mathbb{W}^{0, 3}_{2,2} \left(I\times \Omega\right) }  \Big), 
 \end{equation}
 with a constant $\overline{c} = \overline{c}\Big( \Omega, \eta_1, \eta_2,   \overline{\bv}_1,  \overline{\bv}_2, \overline{\rho}^{\#}_1, \overline{\rho}^{\#}_2  \Big) $.
 Hence, combining the estimates  \eqref{eq:TimeDerivEstim} and \eqref{eq:LinTransportTypEstimPrior}, we obtain  
\begin{equation}\label{eq:LinTransportTypEstimFinal}
  \Vert  \overline{\rho}^{\#}_{1,2}  \Vert_{\raisebox{-1.5ex}{$\mathcal{Y}$}_{\overline{\rho}}  } \leq cT^{1/2} e^{cT}\Vert \left( \overline{\bv}_{1,2}, \eta_{1,2}   \right)\Vert_{  \raisebox{-1.5ex}{$ \mathcal{Y}$}_{\overline\bv}  \raisebox{-1.5ex}{$\times \mathcal{Y}$}_{\eta}}, 
 \end{equation} 
 where  $c = c\Big( \Omega, \eta_1, \eta_2,   \overline{\bv}_1,  \overline{\bv}_2, \overline{\rho}^{\#}_1, \overline{\rho}^{\#}_2  \Big) > 0 $ and the spaces 
 
 \begin{equation*}
\begin{aligned}
  \mathcal{Y}_{\overline\bv} & :=    \mathbb{W}^{0, 3}_{2,2} (I\times \Omega)  \cap \mathbb{W}^{1, 2}_{2,2} (I\times \Omega), &
\\
  \mathcal{Y}_{\eta} &:=  \mathbb{W}^{1, 3}_{2,2} (I\times \omega)  \cap \mathbb{W}^{2, 2}_{2, 2} (I\times \omega), &
\end{aligned}
\end{equation*}\\
are endowed respectively with the norms
\begin{equation*}
\begin{aligned}
  \Vert \overline\bv\Vert_{ \raisebox{-1.5ex}{$\mathcal{Y}$}_{\overline\bv} } & :=  \sqrt{\lambda + 2\mu} \left( \int_I\int_\Omega |\nabla\Delta\overline\bv|^{2}\dx\dt  \right)^{1/2}   + \sqrt{\lambda + 2\mu} \left( \int_I\int_\Omega |\partial_t\nabla^{2}\overline\bv|^{2} \dx\dt \right)^{1/2} &
\\[0.25cm]
 \Vert \eta\Vert_{ \raisebox{-1.5ex}{$\mathcal{Y}$}_{\eta} } &:=   \left( \int_I\int_\omega |\partial_t \naby\Dely\eta|^{2} \dy\dt \right)^{1/2} + \left( \int_I\int_\omega |\partial_{t}^2 \Dely\eta|^{2} \dy\dt \right)^{1/2}. &
\end{aligned}
\end{equation*}\\

To complete the  contraction estimate for the fixed-point operator, it remains to control the pair $\left( \overline{\bv}_{1,2}, \eta_{1,2}   \right) $  in terms of the difference $\overline{\rho}_{1,2} := \overline{\rho}_1 - \overline{\rho}_2 $ in the  $ \mathcal{Y}_{\overline{\rho}}\, -$norm.  \\
 For this purpose, observe that  $\left( \overline{\bv}_{1,2}, \eta_{1,2}   \right) $  satisfies the momentum-structure system 
\begin{align}
 \mathbf{\mathcal{L}}_{\overline{\bv}} (\overline{\bv}_{1,2}) &= \mathbf{h}_{1,2} - \Div{\mathbf{H}_{1,2}}  - \Div\big( a\mathbf{B}_{\eta_0} \left( \overline{\rho}_{1}^{\,\gamma} - \overline{\rho}_{2}^{\,\gamma}\right) \big), 
 \label{momEqAloneBarDiff}
 \\
 \mathbf{\mathcal{L}}_{\eta} (\eta_{1,2}) &=  \bn^\intercal \left[\mathbf{H}_{1,2} - \mu \mathbf{A}_{\eta_0}\nabla\overline{\bv}_{1,2} - \dfrac{\lambda + \mu}{J_{\eta_0}}\left(\mathbf{B}_{\eta_0}\colon \nabla\overline{\bv}_{1,2} \right)\mathbf{B}_{\eta_0}  + a \mathbf{B}_{\eta_0}\big( \overline{\rho}_{1}^{\,\gamma} - \overline{\rho}_{2}^{\,\gamma}\big)  \right]\circ\bm{\varphi} \bn ,
 \label{shellEqAloneBarDiff}
\end{align} 
 with  $\overline{\bv}_{1,2}  \circ \bm{\varphi}  = (\partial_t\eta_{1,2})\bn$ on $I\times \omega$, and source terms  
 \[ \mathbf{h}_{1,2} := \mathbf{h}_{\eta_1}(\overline{\bv}_1)\!\left[\,\overline{\rho}_1\right] - \mathbf{h}_{\eta_2}(\overline{\bv}_2)\!\left[\,\overline{\rho}_2\right],  \qquad  \mathbf{H}_{1,2} := \mathbf{H}_{\eta_1}(\overline{\bv}_1)\!\left[\,\overline{\rho}_1\right] - \mathbf{H}_{\eta_2}(\overline{\bv}_2)\!\left[\,\overline{\rho}_2 \right].  \]
 Of note, this notation highlights the dependence of $\mathbf{h}_{\eta}(\overline{\bv}) $ and $\mathbf{H}_{\eta}(\overline{\bv})$ on the prescribed density $\overline{\rho}$, which we now indicate explicitly via square brackets. \\
 The linear differential operators $ \mathbf{\mathcal{L}}_{\overline{\bv}} $  and  $ \mathbf{\mathcal{L}}_{\eta} $ are defined respectively on the fixed reference domains $I\times\Omega $ and $I\times \omega$ by
 	\begin{align*}
	\mathbf{\mathcal{L}}_{\overline{\bv}}(\bm{\phi}) &:= J_{\eta_0}\overline{\rho}_0 \,\partial_t\bm{\phi} - \Div\left[\mu \mathbf{A}_{\eta_0}\nabla\bm{\phi} + \dfrac{ \lambda + \mu }{J_{\eta_0}} \left(\mathbf{B}_{\eta_0}\colon \nabla\bm{\phi} \right)\mathbf{B}_{\eta_0} \right]  ,
	\\
	\mathbf{\mathcal{L}}_{\eta}(\zeta) &:=  \partial_t^2\zeta - \partial_t\Dely \zeta + \Dely^2\zeta.
	\end{align*}
 Owing to the continuous dependence of the solution to the linearised momentum-structure subproblem on the input data -- cf.~\eqref{linearestimate} -- a corresponding energy estimate for the pair $\left( \overline{\bv}_{1,2}, \eta_{1,2}   \right) $ holds, up to the appropriate regularity dictated by lower-order norms of the data. More precisely, one obtains the following estimate\footnote{Importantly, the norm  terms for $\mathbf{h}_{1,2}$ and $\mathbf{H}_{1,2}$ stemming from \eqref{linearestimate} can be estimated similarly as in \eqref{eq:hdiffEstimate}--\eqref{eq:HdiffEstimate}. 
 
 In particular,  $\Vert \left( \overline{\bv}_{1,2}, \eta_{1,2}   \right)\Vert_{  \raisebox{-1.5ex}{$ \mathcal{Y}$}_{\overline\bv}  \raisebox{-1.5ex}{$\times \mathcal{Y}$}_{\eta}}$ appears additively in their estimate, which -- for $T$ small enough --  can be absorbed into the left-hand-side of the energy inequality.   }:  
\begin{align}\label{eq:EstimCombVeloDispl}
\Vert \left( \overline{\bv}_{1,2}, \eta_{1,2}   \right)\Vert_{  \raisebox{-1.5ex}{$ \mathcal{Y}$}_{\overline\bv}  \raisebox{-1.5ex}{$\times \mathcal{Y}$}_{\eta}} \leq  CT^{1/2}e^{CT} \Bigg(  \Vert  \overline{\rho}_{1}  -  \overline{\rho}_{2} \Vert_{\mathbb{W}^{0, 2}_{\infty,2} (I\times \Omega) \cap \mathbb{W}^{1, 1}_{\infty,2} (I\times \Omega)}  
+     \Vert  \overline{\rho}_{1}^{\,\gamma}  -  \overline{\rho}_{2}^{\,\gamma} \Vert_{\mathbb{W}^{0, 2}_{\infty,2} (I\times \Omega) \cap \mathbb{W}^{1, 1}_{\infty,2} (I\times \Omega)}    \Bigg),
\end{align}  
with a constant $C = C(\mu, \lambda, R,  \eta_1, \eta_2,   \overline{\bv}_1,  \overline{\bv}_2) > 0$. \\
Furthermore, since the density $\overline{\rho} $ is assumed to be bounded below and above, it follows that  the map 
\[ \overline{\rho} \longmapsto \overline{\rho}^{\,\gamma} \quad  \text{is Lipschitz continuous. } \]
 That is, 
 \[\Vert  \overline{\rho}_{1}^{\,\gamma}  -  \overline{\rho}_{2}^{\,\gamma} \Vert_{\mathbb{W}^{0, 2}_{\infty,2} (I\times \Omega) \cap \mathbb{W}^{1, 1}_{\infty,2} (I\times \Omega)}  \leq C(\gamma, R) \Vert \overline{\rho}_{1}  -  \overline{\rho}_{2} \Vert_{\mathbb{W}^{0, 2}_{\infty,2} (I\times \Omega) \cap \mathbb{W}^{1, 1}_{\infty,2} (I\times \Omega)} .  \]
 Thus, \eqref{eq:EstimCombVeloDispl} further reduces to 
 \begin{align}\label{eq:EstimCombVeloDisplFinal}
\Vert \left( \overline{\bv}_{1,2}, \eta_{1,2}   \right)\Vert_{  \raisebox{-1.5ex}{$ \mathcal{Y}$}_{\overline\bv}  \raisebox{-1.5ex}{$\times \mathcal{Y}$}_{\eta}} \leq  C T^{1/2} e^{CT} \Vert  \overline{\rho}_{1} - \overline{\rho}_{2} \Vert_{\mathbb{W}^{0, 2}_{\infty,2} (I\times \Omega) \cap \mathbb{W}^{1, 1}_{\infty,2} (I\times \Omega)},
\end{align} 
 with $C = C(\gamma, \mu, \lambda, R,  \eta_1, \eta_2,   \overline{\bv}_1,  \overline{\bv}_2) > 0$. \\
 Substituting \eqref{eq:EstimCombVeloDisplFinal} in \eqref{eq:LinTransportTypEstimFinal} yields 
 \begin{equation}\label{eq:FixPointFinalEstim}
  \Vert  \overline{\rho}^{\#}_{1,2}  \Vert_{\raisebox{-1.5ex}{$\mathcal{Y}$}_{\overline{\rho}}  } \leq c T e^{cT}\Vert \overline{\rho}_{1,2}\Vert_{\raisebox{-1.5ex}{$\mathcal{Y}$}_{\overline{\rho}}  } , 
 \end{equation} 
 with $c = c(\gamma, \mu, \lambda, R,  \eta_1, \eta_2,   \overline{\bv}_1,  \overline{\bv}_2)$. \\
 Hence,  for a suitable choice of $T > 0$,  the existence of the desired fixed point follows.  This concludes the proof of Theorem \ref{theo:mainresult}. \\

 \appendix
\section{Proof of Lemma \ref{lem:EquivProblem}}\label{appendix:proof}
Although it is common in the literature to derive the fixed-domain system through a weak formulation, we provide here a direct strong-form derivation.  This offers therefore, a transparent alternative to the usual weak-form approach. 

\noindent Recall that  for all $t \in I, $
\[\bfPsi_\eta(t,\cdot) \colon \Omega \to \Omega_\eta; \quad \overline{\bx} = (\overline{x}_1, \overline{x}_2, \overline{x}_3) \longmapsto \bx = \bfPsi_\eta(t, \overline{\bx}) = (x_1, x_2, x_3) \] 
is a smooth, one-to-one orientation preserving mapping.  

\noindent For clarity purpose, we suppress the $t$--dependence in the velocity field. Let $(e_i)_{i=1}^3$ be the canonical basis of $\mathbb{R}^3$ such that $\bv(\bx) = \sum\limits_{i = 1}^{3} v_i(\bx) e_i , \, \text{and} \; \overline{\bv}(\overline\bx) = \sum\limits_{i = 1}^{3} \overline{v}_i (\overline\bx) e_i $. We further define  
\begin{equation*}
\delta_{i,j} := e_i \cdot e_j = \begin{cases}
1 & \text{if} \quad i = j
\\
0 & \text{if} \quad i \ne j. 
\end{cases}
\end{equation*}
To avoid ambiguity in index computations, we indicate the variable only for coordinate derivatives $\left(\text{e.g.}\ \frac{\partial}{\partial x_i }\ \text{or}\ \frac{\partial}{\partial \overline{x}_i }\right).$

\noindent First, observe that  
\begin{align*}
\divx\big(\bv(\bx) \big) &= \mathlarger{\mathlarger{\sum}}\limits_{i = 1}^{3} \dfrac{\partial v_i(\bx) }{\partial x_i} 
\\
& =  \mathlarger{\mathlarger{\sum}}\limits_{i = 1}^{3}  \mathlarger{\mathlarger{\sum}}\limits_{j = 1}^{3}\, \dfrac{\partial \big( v_{i} \circ\bfPsi_{\eta}\big) (\overline\bx) }{\partial \overline{x}_j} \, \dfrac{\partial \overline{x}_j}{\partial x_i},  
\end{align*}
that is, 
\begin{equation}\label{eq:DivFixed}
\divx\big(\bv(\bx) \big) = \nabx \overline{\bv}(\overline{\bx})  \colon \Big(\nabx \bfPsi_{\eta}^{-1}\circ\bfPsi_{\eta}(\overline{\bx}) \Big)^{\intercal}.
\end{equation}
Using \eqref{eq:DivFixed}, we derive that 
\begin{align*}
\nabla\divx\big(\bv(\bx) \big) & =  \mathlarger{\mathlarger{\sum}}\limits_{i = 1}^{3}\, \dfrac{\partial}{\partial x_i} \Big[  \nabx \overline{\bv}(\overline{\bx})  \colon \big(\nabx \bfPsi_{\eta}^{-1}\circ\bfPsi_{\eta}(\overline{\bx}) \big)^{\intercal} \Big] e_i
\\
& = \mathlarger{\mathlarger{\sum}}\limits_{i = 1}^{3}\, \mathlarger{\mathlarger{\sum}}\limits_{j = 1}^{3} \dfrac{\partial}{\partial \overline{x}_j} \Big[  \nabx \overline{\bv}(\overline{\bx})  \colon \big(\nabx \bfPsi_{\eta}^{-1}\circ\bfPsi_{\eta}(\overline{\bx}) \big)^{\intercal}  \Big] \,  \dfrac{\partial \overline{x}_j}{\partial x_i} e_i
\\
& = \dfrac{1}{J_\eta} \, \mathlarger{\mathlarger{\sum}}\limits_{i = 1}^{3} \mathlarger{\mathlarger{\sum}}\limits_{j = 1}^{3} \, \dfrac{\partial}{\partial \overline{x}_j} \Big[  \nabx \overline{\bv}(\overline{\bx})  \colon \big(\nabx \bfPsi_{\eta}^{-1}\circ\bfPsi_{\eta}(\overline{\bx}) \big)^{\intercal} \Big] \, J_{\eta}  \dfrac{\partial \overline{x}_j}{\partial x_i} \, \big(e_i \otimes e_j \big)e_j
\\
& =  \dfrac{1}{J_\eta} \,  \left( \mathlarger{\mathlarger{\sum}}\limits_{i = 1}^{3} \mathlarger{\mathlarger{\sum}}\limits_{j = 1}^{3}\, J_{\eta}  \dfrac{\partial \overline{x}_j}{\partial x_i} \big(e_i \otimes e_j \big)  \right)   \left( \mathlarger{\mathlarger{\sum}}\limits_{j = 1}^{3}\,  \dfrac{\partial}{\partial \overline{x}_j} \Big[  \nabx \overline{\bv}(\overline{\bx})  \colon \big(\nabx \bfPsi_{\eta}^{-1}\circ\bfPsi_{\eta}(\overline{\bx}) \big)^{\intercal}\Big]  e_j   \right) 
\\
& =  \dfrac{1}{J_\eta} \, \mathbf{B}_{\eta}(\overline{\bx})  \nabla\left( \dfrac{1}{J_\eta} \big( \nabla\overline{\bv}(\overline{\bx})  \colon \mathbf{B}_{\eta}(\overline{\bx})  \big)  \right), 
\end{align*}
However, as $ \mathbf{B}_{\eta}$ is divergence free -- that is, Piola identity   \cite[Chapter 5, Lemma 5.8]{rindler2018calculus}, it follows that 
\begin{equation}\label{eq:GradientDivFixed}
\nabla\divx\bv = \dfrac{1}{J_\eta} \, \divx\left( \dfrac{1}{J_\eta} \big( \nabla\overline{\bv} \colon \mathbf{B}_{\eta} \big) \mathbf{B}_{\eta}  \right).
\end{equation}
In order to express $\Delta \bv $ in the reference configuration, we first compute the pull-back of $\nabla\bv$. Thus, 
\begin{align*}
\nabla\bv(\bx) & =  \mathlarger{\mathlarger{\sum}}\limits_{i = 1}^{3} \mathlarger{\mathlarger{\sum}}\limits_{j = 1}^{3} \, \dfrac{\partial v_j (\bx)}{\partial x_i} \big(e_i \otimes e_j \big)
\\
& = \mathlarger{\mathlarger{\sum}}\limits_{i = 1}^{3} \mathlarger{\mathlarger{\sum}}\limits_{j = 1}^{3} \mathlarger{\mathlarger{\sum}}\limits_{k = 1}^{3} \, \dfrac{\partial \big( v_{j} \circ\bfPsi_{\eta}\big) (\overline\bx) }{\partial \overline{x}_k}  \, \dfrac{\partial \overline{x}_k}{\partial x_i} \big(e_i \otimes e_j \big)
\\
& = \mathlarger{\mathlarger{\sum}}\limits_{i = 1}^{3} \mathlarger{\mathlarger{\sum}}\limits_{j = 1}^{3} \mathlarger{\mathlarger{\sum}}\limits_{k = 1}^{3}  \mathlarger{\mathlarger{\sum}}\limits_{l = 1}^{3} \, \dfrac{\partial \big( v_{j} \circ\bfPsi_{\eta}\big) (\overline\bx) }{\partial \overline{x}_k}  \, \dfrac{\partial \overline{x}_l}{\partial x_i} \, \delta_{k,l} \big(e_i \otimes e_j \big)
\\
& = \left( \mathlarger{\mathlarger{\sum}}\limits_{i = 1}^{3}  \mathlarger{\mathlarger{\sum}}\limits_{l = 1}^{3} \, \dfrac{\partial \overline{x}_l}{\partial x_i} \big(e_i \otimes e_l \big)   \right)  \left( \mathlarger{\mathlarger{\sum}}\limits_{j = 1}^{3} \mathlarger{\mathlarger{\sum}}\limits_{k = 1}^{3}  \, \dfrac{\partial \big( v_{j} \circ\bfPsi_{\eta}\big) (\overline\bx) }{\partial \overline{x}_k}  \big(e_k \otimes e_j \big)   \right),
\end{align*}
that is, 
\begin{equation}\label{eq:GradFixed}
 \nabla\bv(\bx) =   \Big(\nabx \bfPsi_{\eta}^{-1}\circ\bfPsi_{\eta}(\overline{\bx}) \Big)^{\intercal}\nabx \overline{\bv}(\overline{\bx}).
\end{equation}
From \eqref{eq:GradFixed}, it follows that 
\begin{align*}
\Delx\bv(\bx) & = \mathlarger{\mathlarger{\sum}}\limits_{i = 1}^{3} \, \mathlarger{\mathlarger{\sum}}\limits_{j = 1}^{3}  \dfrac{\partial \big[\nabx\bv(\bx)\big]_{j, i} }{\partial x_j} \, e_i
\\
& = \mathlarger{\mathlarger{\sum}}\limits_{i = 1}^{3} \, \mathlarger{\mathlarger{\sum}}\limits_{j = 1}^{3} \, \dfrac{\partial}{\partial x_j}  \left[ \Big(\nabx \bfPsi_{\eta}^{-1}\circ\bfPsi_{\eta}(\overline{\bx}) \Big)^{\intercal}\nabx \overline{\bv}(\overline{\bx}) \right]_{j,i} \, e_i 
\\
& = \mathlarger{\mathlarger{\sum}}\limits_{i = 1}^{3} \, \mathlarger{\mathlarger{\sum}}\limits_{j = 1}^{3} \, \mathlarger{\mathlarger{\sum}}\limits_{k = 1}^{3} \, \dfrac{\partial}{\partial \overline{x}_k} \left[ \Big(\nabx \bfPsi_{\eta}^{-1}\circ\bfPsi_{\eta}(\overline{\bx}) \Big)^{\intercal}\nabx \overline{\bv}(\overline{\bx}) \right]_{j,i} \, \dfrac{\partial \overline{x}_k}{\partial x_j} \, e_i
\\
& = \mathlarger{\mathlarger{\sum}}\limits_{i = 1}^{3} \, \mathlarger{\mathlarger{\sum}}\limits_{j = 1}^{3} \, \mathlarger{\mathlarger{\sum}}\limits_{k = 1}^{3} \,  \mathlarger{\mathlarger{\sum}}\limits_{l = 1}^{3} \,  \dfrac{\partial}{\partial \overline{x}_k} \left( \dfrac{\partial  \big( v_{i} \circ\bfPsi_{\eta}\big) (\overline\bx) }{\partial \overline{x}_l} \, \dfrac{\partial \overline{x}_l}{\partial x_j} \right) \, \dfrac{\partial \overline{x}_k}{\partial x_j} \, e_i
\\
& =  \mathlarger{\mathlarger{\sum}}\limits_{i = 1}^{3} \, \mathlarger{\mathlarger{\sum}}\limits_{j = 1}^{3} \, \mathlarger{\mathlarger{\sum}}\limits_{k = 1}^{3} \,  \mathlarger{\mathlarger{\sum}}\limits_{l = 1}^{3} \, \left[ \dfrac{\partial^{2}  \big( v_{i} \circ\bfPsi_{\eta}\big) (\overline\bx) }{\partial \overline{x}_k \partial \overline{x}_l} \, \dfrac{\partial \overline{x}_l}{\partial x_j} \,  \dfrac{\partial \overline{x}_k}{\partial x_j} \, e_i   +   \dfrac{\partial  \big( v_{i} \circ\bfPsi_{\eta}\big) (\overline\bx) }{\partial \overline{x}_l} \, \dfrac{\partial }{\partial  \overline{x}_k } \left( \dfrac{\partial \overline{x}_l}{\partial x_j} \right)  \, \dfrac{\partial \overline{x}_k}{\partial x_j} \, e_i   \right]
\\
& = \dfrac{1}{J_{\eta}} \,  \mathlarger{\mathlarger{\sum}}\limits_{i = 1}^{3} \, \mathlarger{\mathlarger{\sum}}\limits_{k = 1}^{3} \,  \mathlarger{\mathlarger{\sum}}\limits_{l = 1}^{3} \,   \dfrac{\partial^{2}  \big( v_{i} \circ\bfPsi_{\eta}\big) (\overline\bx) }{\partial \overline{x}_k \partial \overline{x}_l} \, \left(  \mathlarger{\mathlarger{\sum}}\limits_{j = 1}^{3} \, J_{\eta} \dfrac{\partial \overline{x}_l}{\partial x_j} \,  \dfrac{\partial \overline{x}_k}{\partial x_j} \right) e_i 
\\
& \quad + \dfrac{1}{J_\eta} \,  \mathlarger{\mathlarger{\sum}}\limits_{i = 1}^{3} \, \mathlarger{\mathlarger{\sum}}\limits_{l = 1}^{3} \,  \dfrac{\partial  \big( v_{i} \circ\bfPsi_{\eta}\big) (\overline\bx) }{\partial \overline{x}_l} \left( \mathlarger{\mathlarger{\sum}}\limits_{k = 1}^{3} \, \mathlarger{\mathlarger{\sum}}\limits_{j = 1}^{3}   \dfrac{\partial }{\partial  \overline{x}_k } \left( \dfrac{\partial \overline{x}_l}{\partial x_j} \right)  \, J_\eta \,  \dfrac{\partial \overline{x}_k}{\partial x_j} \right) e_i 
\\
& =  \dfrac{1}{J_{\eta}} \,  \mathlarger{\mathlarger{\sum}}\limits_{i = 1}^{3} \left( \mathlarger{\mathlarger{\sum}}\limits_{k = 1}^{3} \,  \mathlarger{\mathlarger{\sum}}\limits_{l = 1}^{3} \,   \dfrac{\partial^{2}  \big( v_{i} \circ\bfPsi_{\eta}\big) (\overline\bx) }{\partial \overline{x}_k \partial \overline{x}_l} \, \left[ \mathbf{A}_{\eta} (\overline{\bx}) \right]_{k,l} \right)\, e_i 
\\
& \quad +  \dfrac{1}{J_\eta} \,  \mathlarger{\mathlarger{\sum}}\limits_{i = 1}^{3} \, \mathlarger{\mathlarger{\sum}}\limits_{l = 1}^{3} \,  \dfrac{\partial  \big( v_{i} \circ\bfPsi_{\eta}\big) (\overline\bx) }{\partial \overline{x}_l} \left( \mathlarger{\mathlarger{\sum}}\limits_{k = 1}^{3} \, \mathlarger{\mathlarger{\sum}}\limits_{j = 1}^{3}   \dfrac{\partial }{\partial  \overline{x}_k } \left( \dfrac{\partial \overline{x}_l}{\partial x_j} \right)  \, J_\eta \,  \dfrac{\partial \overline{x}_k}{\partial x_j} \right) e_i 
\\
& = \dfrac{1}{J_{\eta}} \,  \mathlarger{\mathlarger{\sum}}\limits_{i = 1}^{3} \Big( \mathbf{A}_\eta (\overline{\bx}) \colon \nabla^{2}  \big( v_{i} \circ\bfPsi_{\eta}\big)(\overline{\bx})  \Big) e_i  +  \dfrac{1}{J_{\eta}} \,  \mathlarger{\mathlarger{\sum}}\limits_{i = 1}^{3}  \Big( \divx \big(\mathbf{A}_{\eta} (\overline{\bx})\big) \cdot \nabx \big( v_{i} \circ\bfPsi_{\eta}\big)(\overline{\bx})   \Big) e_i
\\
& =  \dfrac{1}{J_{\eta}} \,  \mathlarger{\mathlarger{\sum}}\limits_{i = 1}^{3}  \, \divx  \Big(  \mathbf{A}_{\eta} (\overline{\bx}) \nabx \big( v_{i} \circ\bfPsi_{\eta}\big)(\overline{\bx}) \Big) e_i .
\end{align*}
Consequently,
\begin{equation}\label{eq:LaplaceFixed}
\Delx\bv = \dfrac{1}{J_{\eta}} \,  \divx  \Big(  \mathbf{A}_{\eta} \nabx \overline{\bv} \Big). 
\end{equation}
Moreover, we have 
\begin{align*}
\partial_t \left(\varrho\bv \right) & = (\partial_t \varrho)\bv + \varrho\partial_t \bv
\\
& = - \divx (\varrho\bv) \bv  + \varrho \left( \partial_t \overline{\bv} + \nabx\overline{\bv} \cdot \partial_t \bfPsi_{\eta}^{-1}\circ \bfPsi_\eta  \right)
\\
& = - \big( \divx(\varrho \bv\otimes\bv) - \varrho\,\bv\,\divx\bv \big) + \overline{\varrho }\left( \partial_t \overline{\bv} + \nabx\overline{\bv} \cdot \partial_t \bfPsi_{\eta}^{-1}\circ \bfPsi_\eta  \right)
\\
& = -  \divx(\varrho \bv\otimes\bv) + \overline{\varrho } \,\overline\bv \Big( \nabx\overline{\bv} \colon \left( \nabx\bfPsi_{\eta}^{-1}\circ\bfPsi_\eta\right)^{\intercal} \Big) + \overline{\varrho }\left( \partial_t \overline{\bv} + \nabx\overline{\bv} \cdot \partial_t \bfPsi_{\eta}^{-1}\circ \bfPsi_\eta  \right), 
\end{align*}
that is, 
\begin{equation}\label{eq:rhsMomenFixed}
\partial_t \left(\varrho\bv \right)  + \divx(\varrho \bv\otimes\bv) =  \overline{\varrho } \,\overline\bv \Big( \nabx\overline{\bv} \colon \left( \nabx\bfPsi_{\eta}^{-1}\circ\bfPsi_\eta\right)^{\intercal} \Big) + \overline{\varrho }\left( \partial_t \overline{\bv} + \nabx\overline{\bv} \cdot \partial_t \bfPsi_{\eta}^{-1}\circ \bfPsi_\eta  \right).
\end{equation}
Likewise, applying the chain rule yields  
\begin{equation}\label{eq:pressureFixed}
\nabx p(\varrho) = a \left( \nabx\bfPsi_{\eta}^{-1}\circ\bfPsi_\eta\right)^{\intercal} \nabx\overline{\varrho}^{\,\gamma}.
\end{equation} 
Combining \eqref{eq:GradientDivFixed}, \eqref{eq:LaplaceFixed},  \eqref{eq:rhsMomenFixed} and \eqref{eq:pressureFixed} yields the desired result.

 \section*{Acknowledgements}
 The author thanks Dominic Breit, Romeo Mensah, and Pei Su for helpful discussions and constructive comments.
This work was funded by the Deutsche Forschungsgemeinschaft (DFG) -- Projektnummer 543675748.

\section*{Compliance with Ethical Standards}
\smallskip
\par\noindent
{\bf Conflict of Interest}. The authors declare that they have no conflict of interest.

\smallskip
\par\noindent
{\bf Data Availability}. Data sharing is not applicable to this article as no datasets were generated or analysed during the current study.

\end{document}